\documentclass[12pt]{scrartcl}
\usepackage{a4}
\usepackage{amsthm}
\usepackage{amsmath}
\usepackage{amssymb}
\usepackage{amsfonts}
\usepackage{mathrsfs}
\usepackage{dsfont}
\usepackage{latexsym}
\usepackage{color}
\usepackage{bbm,exscale}
\definecolor{Myblue}{rgb}{0,0,0.6}
\usepackage[a4paper,colorlinks,citecolor=Myblue,linkcolor=Myblue,urlcolor=Myblue,pdfpagemode=None]{hyperref}
\usepackage[square,numbers,sort&compress]{natbib}
\usepackage[all,cmtip]{xy}
\usepackage{tikz}
\usetikzlibrary{decorations.pathmorphing}
\usetikzlibrary{calc}
\usetikzlibrary{decorations.markings}
\usetikzlibrary{fadings,decorations.pathreplacing}
\usetikzlibrary{matrix,arrows}
\usetikzlibrary{arrows,calc,decorations.pathreplacing,decorations.markings,shapes.geometric,shadows}

\tikzset{
    string/.style={draw=#1, postaction={decorate}, decoration={markings,mark=at position .51 with {\arrow[draw=#1]{>}}}},
    costring/.style={draw=#1, postaction={decorate}, decoration={markings,mark=at position .51 with {\arrow[draw=#1]{<}}}},
    ostring/.style={draw=#1, postaction={decorate}, decoration={markings,mark=at position .47 with {\arrow[draw=#1]{>}}}},
    ustring/.style={draw=#1, postaction={decorate}, decoration={markings,mark=at position .56 with {\arrow[draw=#1]{>}}}},
    oostring/.style={draw=#1, postaction={decorate}, decoration={markings,mark=at position .43 with {\arrow[draw=#1]{>}}}},
    uustring/.style={draw=#1, postaction={decorate}, decoration={markings,mark=at position .59 with {\arrow[draw=#1]{>}}}},
    directed/.style={string=blue!50!black}, 
    odirected/.style={ostring=blue!50!black}, 
    udirected/.style={ustring=blue!50!black}, 
    oodirected/.style={oostring=blue!50!black}, 
    uudirected/.style={uustring=blue!50!black},     
    redirected/.style={costring= blue!50!black},
    redirectedgreen/.style={costring= green!50!black},
    directedgreen/.style={string= green!50!black},
}

\tikzset{-dot-/.style={decoration={
  markings,
  mark=at position 0.5 with {\fill circle (2pt);}},postaction={decorate}}}

\tikzset{
	Fdot/.style={circle, draw, fill, inner sep=0pt}, 
	Odot/.style={circle, draw, inner sep=0.1pt, minimum size=0.1cm}
	}

\def\nicedashedcolourscheme{\shadedraw[top color=blue!22, bottom color=blue!22, draw=gray, dashed]}
\def\nicedashedpalecolourscheme{\shadedraw[top color=blue!12, bottom color=blue!12, draw=gray, dashed]}
\def\nicecolourscheme{\shadedraw[top color=blue!22, bottom color=blue!22, draw=blue!22]}
\def\nicepalecolourscheme{\shadedraw[top color=blue!12, bottom color=blue!12, draw=white]}

  \vfuzz \hfuzz
\makeatletter
\newcommand{\raisemath}[1]{\mathpalette{\raisem@th{#1}}}
\newcommand{\raisem@th}[3]{\raisebox{#1}{$#2#3$}}
\makeatother

\newcommand{\E}{\text{e}}
\newcommand{\I}{\text{i}}
\newcommand{\B}{\mathcal{B}}
\newcommand{\Borb}{\B_{\mathrm{orb}}}
\newcommand{\Beq}{\B_{\mathrm{eq}}}
\newcommand{\C}{\mathds{C}}

\newcommand{\N}{\mathds{N}}
\newcommand{\Q}{\mathds{Q}}
\newcommand{\R}{\mathds{R}}
\newcommand{\Z}{\mathds{Z}}
\def\1{\ifmmode\mathrm{1\!l}\else\mbox{\(\mathrm{1\!l}\)}\fi}

\newcommand{\be}{\begin{equation}}
\newcommand{\ee}{\end{equation}}
\newcommand{\bes}{\begin{equation*}}
\newcommand{\ees}{\end{equation*}}

\newcommand{\Hom}{\operatorname{Hom}}
\newcommand{\End}{\operatorname{End}}
\newcommand{\modu}{\operatorname{mod}}
\def\LG{\mathcal{LG}}

\def\LGorb{\mathcal{LG}_{\mathrm{orb}}}
\def\LGeq{\mathcal{LG}_{\mathrm{eq}}}
\newcommand{\hmf}{\operatorname{hmf}}
\newcommand{\HMF}{\operatorname{HMF}}
\newcommand{\ev}{\operatorname{ev}}
\newcommand{\tev}{\widetilde{\operatorname{ev}}}
\newcommand{\coev}{\operatorname{coev}}
\newcommand{\tcoev}{\widetilde{\operatorname{coev}}}
\def\lra{\longrightarrow}
\def\lmt{\longmapsto}
\DeclareMathOperator{\tr}{tr}
\DeclareMathOperator{\str}{str}
\DeclareMathOperator{\Jac}{Jac}
\def\Re{R^{\operatorname{e}}}
\DeclareMathOperator{\Res}{Res}

\newcommand{\diml}{\dim_{\mathrm{l}}}
\newcommand{\dimr}{\dim_{\mathrm{r}}}

\newcommand{\dX}{{}^\dagger\hspace{-1.8pt}X}
\newcommand{\deqX}{{}^\star\hspace{-1.8pt}X} 
\newcommand{\dY}{{}^\dagger\hspace{-0.3pt}Y}
\newcommand{\dphi}{{}^\dagger\hspace{-0.9pt}\phi}

\newcommand\arxiv[2]      {\href{http://arXiv.org/abs/#1}{#2}}
\newcommand\doi[2]        {\href{http://dx.doi.org/#1}{#2}}
\newcommand\httpurl[2]    {\href{http://#1}{#2}}

\allowdisplaybreaks

\deffootnote[1em]{1em}{1em}
{\textsuperscript{\thefootnotemark}}

\theoremstyle{definition}
\newtheorem{definition}{Definition}
\newtheorem{proposition}[definition]{Proposition}
\newtheorem{theorem}[definition]{Theorem}
\newtheorem{lemma}[definition]{Lemma}
\newtheorem{corollary}[definition]{Corollary}
\newtheorem{remark}[definition]{Remark}

\newtheorem{example}[definition]{Example}

\numberwithin{equation}{section}
\numberwithin{definition}{section}
\numberwithin{figure}{section}

\newcommand\void[1]{}

\begin{document}

\title{Orbifold completion \\ of defect bicategories}
\author{Nils Carqueville$^*$ \quad Ingo Runkel$^\dagger$
\\[0.5cm]
 \normalsize{\tt \href{mailto:nils@carqueville.net}{nils@carqueville.net}} \quad
  \normalsize{\tt \href{mailto:ingo.runkel@uni-hamburg.de}{ingo.runkel@uni-hamburg.de}}\\[0.1cm]
  {\normalsize\slshape $^*$Arnold Sommerfeld Center for Theoretical Physics, }\\[-0.1cm]
  {\normalsize\slshape LMU M\"unchen, Theresienstra\ss e~37, D-80333 M\"unchen}\\[-0.1cm]
  {\normalsize\slshape $^*$Excellence Cluster Universe, Boltzmannstra\ss e~2, D-85748 Garching}\\[-0.1cm]
  {\normalsize\slshape $^*$Hausdorff Research Institute for Mathematics, }\\[-0.1cm]
  {\normalsize\slshape Poppelsdorfer Allee 45, D-53115 Bonn}\\[0.1cm]
  {\normalsize\slshape $^\dagger$Department Mathematik, Universit\"{a}t Hamburg, }\\[-0.1cm]
  {\normalsize\slshape Bundesstra\ss e 55, D-20146 Hamburg}\\[-0.1cm]
}
\date{}
\maketitle

\vspace{-12.8cm}
\hfill {\scriptsize Hamburger Beitr\"age zur Mathematik 453}

\vspace{-1.0cm}

\hfill {\scriptsize ZMP-HH/12-20}

\vspace{12cm}

\begin{abstract}
Orbifolds of two-dimensional quantum field theories have a natural formulation in terms of defects or domain walls. This perspective allows for a rich generalisation of the orbifolding procedure, which we study in detail for the case of topological field theories. Namely, a TFT with defects gives rise to a pivotal bicategory of ``worldsheet phases'' and defects between them. We develop a general framework which takes such a bicategory~$\B$ as input and returns its ``orbifold completion'' $\Borb$. The completion satisfies the natural properties $\B \subset \Borb$ and $(\Borb)_{\mathrm{orb}} \cong \Borb$, and it gives rise to various new equivalences and nondegeneracy results. When applied to TFTs, the objects in $\Borb$ correspond to generalised orbifolds of the theories in~$\B$. In the example of Landau-Ginzburg models we recover and unify conventional equivariant matrix factorisations, prove when and how (generalised) orbifolds again produce open/closed TFTs, and give nontrivial examples of new orbifold equivalences.
\end{abstract}

\thispagestyle{empty}
\newpage

\tableofcontents

\section{Introduction and summary}\label{introduction}

Orbifolding is a basic construction~in (quantum) field theory, string theory, algebraic geometry, and representation theory. The conventional setup is some ``theory'' (about which we will be less vague soon enough) together with a symmetry group. Gauging this symmetry amounts to restricting to the invariant sectors while simultaneously adding new twisted sectors. In this way the orbifold theory is constructed from the original one, and it often inherits desirable properties from the symmetry group. 

A slightly different look at the usual orbifold procedure allows for an immediate generalisation. This alternate point of view arises in the framework of two-dimensional field theories with defects. Later we will deal with this notion rigorously, but in the next few paragraphs we shall argue heuristically and develop some intuition. In this vein, let us consider a set of theories $a_1, a_2, \ldots$ that govern various domains or phases of a two-dimensional worldsheet. The different phases are separated from one another by one-dimensional oriented manifolds. These are called domain walls or defect lines $X_1, X_2, \ldots$, and they come with data that encodes to what extent they allow transfers between the theories of neighbouring phases. A typical patch of worldsheet with defects and field insertions (at the endpoints and junctions of defect lines) looks as follows: 
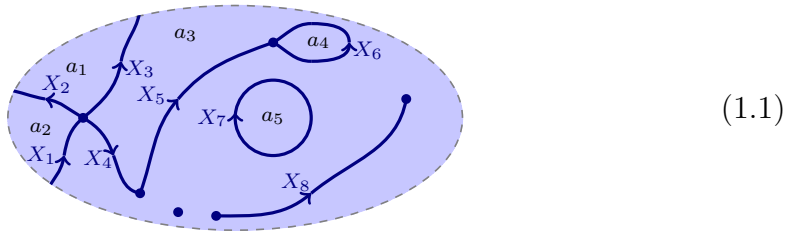
\be\label{eq:worldsheetwithdefects}
\begin{tikzpicture}[very thick,scale=0.5,color=blue!50!black, baseline]
\clip (0,0) ellipse (6cm and 3cm);
\nicedashedcolourscheme (0,0) ellipse (6cm and 3cm);

\draw (-4.15,1.3) [black] node {{\scriptsize $a_1$}};
\draw (-5.1,-0.3) [black] node {{\scriptsize $a_2$}};
\draw (-1.3,2.2) [black] node {{\scriptsize $a_3$}};
\draw (2.2,2.0) [black] node {{\scriptsize $a_4$}};
\draw (1,0) [black] node {{\scriptsize $a_5$}};

\draw (-5.1,-1) node {{\scriptsize $X_1$}};
\draw (-4.7,0.9) node {{\scriptsize $X_2$}};
\draw (-2.5,1.3) node {{\scriptsize $X_3$}};
\draw (-3.6,-1.1) node {{\scriptsize $X_4$}};
\draw (-2.2,0.6) node {{\scriptsize $X_5$}};
\draw (3.5,1.8) node {{\scriptsize $X_6$}};
\draw (-0.6,0) node {{\scriptsize $X_7$}};
\draw (1.6,-1.7) node {{\scriptsize $X_8$}};

\draw[->, very thick, out=60, in=260] (-5,-2) to (-4.5,-1);
\draw[very thick, out=80, in=220] (-4.5,-1) to (-4,0);
\filldraw (-4,0) circle (2.5pt);
\draw[->,very thick, out=150, in=350] (-4,0) to (-5,0.5);
\draw[very thick, out=170, in=350] (-5,0.5) to (-6,0.75);

\draw[->,very thick, out=-30, in=110] (-4,0) to (-3.2,-1);
\draw[very thick, out=-70, in=170] (-3.2,-1) to (-2.5,-2);
\filldraw (-2.5,-2) circle (2.5pt);

\draw[->,very thick, out=70, in=230] (-2.5,-2) to (-1.5,0.5);
\draw[very thick, out=50, in=200] (-1.5,0.5) to (1,2);
\filldraw (1,2) circle (2.5pt);

\draw[very thick, out=-30, in=180] (1,2) to (2,1.5);
\draw[->,very thick, out=0, in=270] (2,1.5) to (3,2);
\draw[very thick, out=90, in=0] (3,2) to (2,2.5);
\draw[very thick, out=180, in=30] (2,2.5) to (1,2);

\draw[very thick] (1,0) circle (1);
\draw[->,very thick, out=90, in=270] (0,0) to (0,0);

\filldraw (-1.5,-2.5) circle (2.5pt);

\filldraw (-0.5,-2.6) circle (2.5pt);
\draw[->,very thick, out=0, in=220] (-0.5,-2.6) to (2,-2);
\draw[very thick, out=40, in=260] (2,-2) to (4.5,0.5);
\filldraw (4.5,0.5) circle (2.5pt);

\draw[->,very thick, out=40, in=260] (-4,0) to (-3,1.5);
\draw[very thick, out=80, in=270] (-3,1.5) to (-2.5,3);
\end{tikzpicture}
\ee

What a field theory wants to do is to compute correlators, i.\,e.~the expectation values $\langle \ldots \rangle$ of fields inserted on points of worldsheets with defects. For simplicity, let us restrict to \textsl{topological} theories, meaning that the value of the correlators depends only on the isotopy classes of defect lines and field insertions. The precise functorial definition of such 2d TFTs with defects \cite{rs0808.1419, dkr1107.0495} is reviewed in Section~\ref{subsec:TFTwdefects}. 

To make contact with orbifolds let us consider two theories~$a$ and~$b$, and a defect $X: a\rightarrow b$ between them. Our goal is to compute all correlators in theory~$b$ only from knowledge of theory~$a$ and the defect~$X$. To achieve this, we make the additional assumption that~$X$ has invertible quantum dimension, which means that
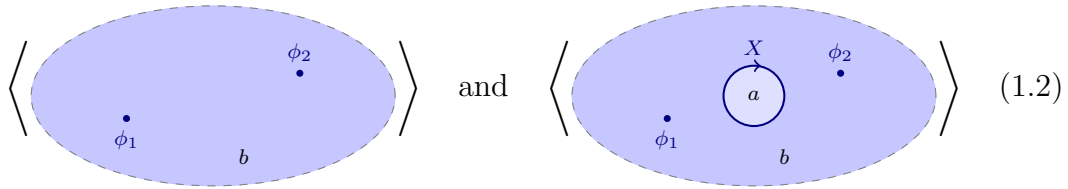
\be\label{eq:islandins}
\left\langle
\begin{tikzpicture}[thick,scale=0.4,color=blue!50!black, baseline]
\clip (0,0) ellipse (6cm and 3cm);
\nicedashedcolourscheme (0,0) ellipse (6cm and 3cm);

\draw (1,-2) [black] node {{\scriptsize $b$}};

\fill (-2.85,-0.75) circle (3.3pt) node[below] {{\scriptsize$\phi_1$}};
\fill (2.85,0.75) circle (3.3pt) node[above] {{\scriptsize$\phi_2$}};

\end{tikzpicture}
\right\rangle
\quad \text{and} \quad
\left\langle
\begin{tikzpicture}[thick,scale=0.4,color=blue!50!black, baseline]
\clip (0,0) ellipse (6cm and 3cm);
\nicedashedcolourscheme (0,0) ellipse (6cm and 3cm);

\nicepalecolourscheme (0,0) circle (1);
\draw[
	decoration={markings, mark=at position 0.25 with {\arrow{<}}}, postaction={decorate}
	] 
	(0,0) circle (1);
\draw (0,1.6) node {{\scriptsize $X$}};

\fill (-2.85,-0.75) circle (3.3pt) node[below] {{\scriptsize$\phi_1$}};
\fill (2.85,0.75) circle (3.3pt) node[above] {{\scriptsize$\phi_2$}};

\draw (0,0) [black] node {{\scriptsize $a$}};
\draw (1,-2) [black] node {{\scriptsize $b$}};

\end{tikzpicture}
\right\rangle
\ee
are equal up to a nonzero factor for all correlators. After an appropriate rescaling we may assume that the two correlators in~\eqref{eq:islandins} are actually equal. Of course we can also insert more than one ``island'' of theory~$a$ in the ``sea'' of theory~$b$, bounded by copies of the defect~$X$. 
Since the defects are topological we may let the islands expand until their boundaries nearly meet. What once were $a$-islands in $b$-sea is now $a$-land partitioned by $b$-rivers, and the correlators in~\eqref{eq:islandins} are equal to
\be
\left\langle
\begin{tikzpicture}[thick,scale=0.4,color=blue!50!black, baseline]
\clip (0,0) ellipse (6cm and 3cm);
\nicedashedcolourscheme (0,0) ellipse (6cm and 3cm);

\coordinate (circ1) at ($ (-2,1) $);
\nicepalecolourscheme (circ1) circle (1);
\draw[
	decoration={markings, mark=at position 0.25 with {\arrow{<}}}, postaction={decorate}
	] 
	(circ1) circle (1);
\coordinate (circ1) at ($ (2,-1) $);
\nicepalecolourscheme (circ1) circle (1);
\draw[
	decoration={markings, mark=at position 0.25 with {\arrow{<}}}, postaction={decorate}
	] 
	(circ1) circle (1);
\coordinate (circ2) at ($ (1.3,1.4) $);
\nicepalecolourscheme (circ2) circle (0.8);
\draw[
	decoration={markings, mark=at position 0.25 with {\arrow{<}}}, postaction={decorate}
	] 
	(circ2) circle (0.8);
\coordinate (circ2) at ($ (-1.3,-1.4) $);
\nicepalecolourscheme (circ2) circle (0.8);
\draw[
	decoration={markings, mark=at position 0.25 with {\arrow{<}}}, postaction={decorate}
	] 
	(circ2) circle (0.8);

\fill (-2.85,-0.75) circle (3.3pt) node[left] {};
\fill (2.85,0.75) circle (3.3pt) node[right] {};

\draw (0,0) [black] node {};
\draw (0,-2.25) [black] node {};

\end{tikzpicture}
\right\rangle
=
\left\langle
\begin{tikzpicture}[thick,scale=0.4,color=blue!50!black, baseline, rounded corners=0.1pt]
\clip (0,0) ellipse (6cm and 3cm);
\nicedashedpalecolourscheme (0,0) ellipse (6cm and 3cm);

\coordinate (circ1) at ($ (-150:0.5) + (-2.85,-0.75) $);
\coordinate (circ2) at ($ (-150:0.5) + (-5.85,-0.75) $);
\coordinate (circ3) at ($ (150:0.5) + (-2.85,-0.75) $);
\coordinate (circ4) at ($ (150:0.5) + (-5.85,-0.75) $);
\filldraw
	[
	fill=blue!22, 
	decoration={markings, mark=at position 0.23 with {\arrow{>}},
					mark=at position 0.745 with {\arrow{>}}
					}, postaction={decorate}
	] 
	(circ2) -- (circ1) arc (-150:150:0.5) -- (circ3) -- (circ4) -- (circ2);
\coordinate (circ1) at ($ (30:0.5) + (2.85,0.75) $);
\coordinate (circ2) at ($ (30:0.5) + (5.85,0.75) $);
\coordinate (circ3) at ($ (330:0.5) + (2.85,0.75) $);
\coordinate (circ4) at ($ (330:0.5) + (5.85,0.75) $);
\filldraw
	[
	fill=blue!22, 
	decoration={markings, mark=at position 0.23 with {\arrow{>}},
					mark=at position 0.745 with {\arrow{>}}
					}, postaction={decorate}
	] 
	(circ2) -- (circ1) arc (30:330:0.5) -- (circ3) -- (circ4) -- (circ2);
\filldraw
	[
	fill=blue!22, 
	decoration={markings, mark=at position 0.03 with {\arrow{>}},
					mark=at position 0.13 with {\arrow{>}},
					mark=at position 0.31 with {\arrow{>}},
					mark=at position 0.40 with {\arrow{>}},
					mark=at position 0.46 with {\arrow{>}},
					mark=at position 0.534 with {\arrow{>}},
					mark=at position 0.63 with {\arrow{>}},
					mark=at position 0.81 with {\arrow{>}},
					mark=at position 0.906 with {\arrow{>}},
					mark=at position 0.9635 with {\arrow{>}}
					}, postaction={decorate}
	] 
	(-0.25,3.2) -- (-0.25,1.25) -- (-6,1.25) -- (-6,0.75) -- (-0.25,0.75) -- (-0.25,-3.2) -- (0.25,-3.2) -- (0.25,-1.25) -- (6,-1.25) -- (6,-0.75) -- (0.25,-0.75) -- (0.25,3.2) -- (-0.25,3.2);

\fill (-2.85,-0.75) circle (3.3pt) node[left] {};
\fill (2.85,0.75) circle (3.3pt) node[right] {};

\draw (0,0) [black] node {};
\draw (0,-2.25) [black] node {};

\end{tikzpicture}
\right\rangle
. 
\ee
Note that whenever two parallel defect lines are close to each other, they have opposite orientation. Denoting the image of the field~$\phi_i$ under the action of the defect~$X$ by~$\Phi_i$, the orientation-flipped defect by~$X^\dagger$ and the fusion product of defects by~$\otimes$, we find that correlators in theory~$b$ can be computed as correlators in theory~$a$, together with a network of defects $A := X^\dagger \otimes X$ (that we draw in green) and trivalent junction fields: 
\be\label{eq:bfromaA}
\left\langle
\begin{tikzpicture}[thick,scale=0.4,color=blue!50!black, baseline]
\clip (0,0) ellipse (6cm and 3cm);
\nicedashedcolourscheme (0,0) ellipse (6cm and 3cm);

\draw (0,-2.25) [black] node {};

\fill (-2.85,-0.75) circle (3.3pt) node[below] {{\scriptsize$\phi_1$}};
\fill (2.85,0.75) circle (3.3pt) node[above] {{\scriptsize$\phi_2$}};

\draw (1,-2) [black] node {{\scriptsize $b$}};

\end{tikzpicture}
\right\rangle
=
\left\langle
\begin{tikzpicture}[very thick,scale=0.4,color=green!50!black, baseline, rounded corners=0.1pt]
\clip (0,0) ellipse (6cm and 3cm);
\nicedashedpalecolourscheme (0,0) ellipse (6cm and 3cm);

\draw
	[
	decoration={markings, mark=at position 0.45 with {\arrow{>}}}, postaction={decorate}
	] 
	(-2.85,-0.75) -- (-5.85,-0.75);
\draw
	[
	decoration={markings, mark=at position 0.45 with {\arrow{>}}}, postaction={decorate}
	] 
	(2.85,0.75) -- (5.85,0.75);
\draw
	[
	decoration={markings, mark=at position 0.2 with {\arrow{>}},
					mark=at position 0.5 with {\arrow{>}},
					mark=at position 0.84 with {\arrow{>}}
					}, postaction={decorate}
	] 
	(0,-3.2) -- (0,3.2);
\draw
	[
	decoration={markings, mark=at position 0.5 with {\arrow{<}}}, postaction={decorate}
	] 
	(0,-1) -- (6,-1);
\draw
	[
	decoration={markings, mark=at position 0.5 with {\arrow{>}}}, postaction={decorate}
	] 
	(0,1) -- (-6,1);

\fill (-2.85,-0.75) circle (4.3pt) node[below] {{\scriptsize$\Phi_1$}};
\fill (2.85,0.75) circle (4.3pt) node[above] {{\scriptsize$\Phi_2$}};
\fill (0,1) circle (4.3pt) node[below] {};
\fill (0,-1) circle (4.3pt) node[below] {};

\draw (1,-2) [black] node {{\scriptsize $a$}};

\end{tikzpicture}
\right\rangle \, .
\ee

This construction can also be turned around \cite{ffrs0909.5013}: one can start from a defect~$A$ together with two junctions, subject to certain properties detailed in Section~\ref{sec:TFT-orbifold}, and \textsl{define} the correlator on the left of~\eqref{eq:bfromaA} by the correlator on the right. The collection of correlators obtained in this way will be called the \textsl{generalised orbifold} of theory~$a$ by the defect~$A$ (with junctions). 

If a group of symmetries of theory $a$ is implemented by the action of defects, these can be assembled into a ``symmetry defect'' $A$. Together with a choice of junctions one recovers in this way ordinary orbifolds (see Section~\ref{subsec:equiMF} for an example). But in general~$A$ does not have to arise from a group, thus indeed generalising the concept of orbifolds; concrete examples of this phenomenon will be discussed in Sections~\ref{subsec:Knoerrer} and~\ref{subsec:minimalmodels}. 

\medskip

In the present paper we mould the above ideas into precise terms and study some of their consequences. To set the stage for a summary, we first organise theories, defects, and fields as the objects, 1-morphisms, and 2-morphisms of a bicategory~$\B$. Orientation reversal endows this bicategory with adjoints for all 1-morphisms as well as a pivotal structure. In Section~\ref{subsec:bicatadj} we recall the relevant definitions, and in Section~\ref{subsec:worldsheetbicategory} we review how to extract such a bicategory from the data of the functorial description of a 2d TFT with defects. 

Let now~$\B$ be any pivotal bicategory whose 1-morphism categories are idempotent complete (a technical assumption we need). In the categorical language the relevant properties of the defect $A: a\rightarrow a$ above will lead us to consider a certain kind of \textsl{algebra objects} $A\in\B(a,a)$, namely separable symmetric Frobenius algebras (see Section~\ref{subsec:algebrasbimodules}). 

In the motivational paragraphs above we considered the special case $A=X^\dagger \otimes X$ for a defect $X: a\rightarrow b$ with invertible quantum dimension (cf.~\eqref{eq:traces} below). This allowed us to obtain theory~$b$ from the pair $(a,A)$. Generalising even further, we construct a new bicategory $\Borb$ whose objects are pairs $(a,A)$ with $a\in\B$ and $A\in\B(a,a)$ a separable symmetric Frobenius algebra. 1- and  2-morphisms in $\Borb$ are defined to be 1- and 2-morphisms in~$\B$ with suitable extra structure (namely bimodules and bimodule maps, see Section~\ref{subsec:algebrasbimodules}). As shown in Section~\ref{subsec:Beq}, $\Borb$ is again pivotal. 

We can think of $\Borb$ as the theory of generalised orbifolds of~$\B$. As expected~$\B$ fully embeds into $\Borb$ since unit 1-morphisms~$I_a$ are naturally endowed with the structure to make $(a,I_a)$ an object in $\Borb$ for each $a\in\B$. Typically~$\B$ is not equivalent to $\Borb$, but in Proposition~\ref{prop:eqeqeq} and~\eqref{eq:orborborb} we will show that the full embedding $\Borb \subset (\Borb)_{\mathrm{orb}}$ gives an equivalence
\be
\Borb \cong (\Borb)_{\mathrm{orb}} \, . 
\ee
Thus $\Borb$ deserves the name \textsl{orbifold completion}: while the set of objects (=~theories) in~$\B$ may not be large enough to close under taking generalised orbifolds, the bicategory $\Borb$ is complete in this sense. 

We can now state one of the central results in the theory of generalised orbifolds (Propositions~\ref{prop:algebraXdaggerX}, \ref{prop:XXdaggerInverseInBeq} and Theorem~\ref{thm:algebraXdaggerX-summary}). 
Let~$\B$ be a pivotal bicategory as before, and let $X\in\B(a,b)$ have invertible quantum dimension. Then
\be\label{eq:aXXbI}
X: (a, X^\dagger \otimes X) \lra (b,I_b)
\ee
is an isomorphism in $\Borb$. Put differently and in terms of our TFT interpretation, theory~$b$ is equivalent to the $(X^\dagger \otimes X)$-orbifold of theory~$a$ -- just as we argued in~\eqref{eq:bfromaA}. This result is a defect-inspired variant of the monadicity theorem. 

The equivalence~\eqref{eq:aXXbI} holds in an even bigger bicategory $\Beq$ which is obtained from $\Borb$ by relaxing the conditions on the objects $(a,A)$; to wit, $A$ does not necessarily have to be symmetric. We call $\Beq$ the \textsl{equivariant completion} of~$\B$ since in the examples discussed later, $\Beq$ is already sufficient to recover ordinary equivariant constructions. In fact this construction works for \textsl{any} bicategory~$\B$ with idempotent complete 1-morphism categories (but without assuming adjunctions or pivotality). Furthermore, if~$\B$ is pivotal, then in Proposition~\ref{prop:Beq-has-adjoints} we will show that $\Beq$ has adjoints. 

A result that only holds in $\Borb$ concerns the existence of nondegenerate pairings. This is a structure that has to be present in the original bicategory~$\B$ if it is to describe a 2d TFT with defects. More precisely, let us assume that there are linear maps $\langle-\rangle_a: \End_{\B}(I_a) \rightarrow \C$ (the ``one-point correlators on a sphere''). They induce pairings on $\End_{\B}(I_a)$ which we interpret as two-point bulk correlators of theory~$a$. Furthermore, for any $X\in\B(a,b)$ we define the ``defect pairing''
\be
\langle \Psi_1, \Psi_2 \rangle _X = 
\left\langle \, 
\begin{tikzpicture}[very thick,scale=0.6,color=blue!50!black, baseline]
\draw (0,0) circle (1.5);
\draw[<-, very thick] (0.100,-1.5) -- (-0.101,-1.5) node[above] {}; 
\draw[<-, very thick] (-0.100,1.5) -- (0.101,1.5) node[below] {}; 
\fill (-22.5:1.5) circle (3.3pt) node[left] {{\small$\Psi_2$}};
\fill (22.5:1.5) circle (3.3pt) node[left] {{\small$\Psi_1$}};
\fill (270:1.5) circle (0pt) node[above] {{\small$X$}};
\end{tikzpicture} 
\,  \right\rangle_{\raisemath{10pt}{\!\!\!a}} 
\ee
where we employ standard string diagram notation as reviewed in Section~\ref{subsec:bicatadj}. 
In Corollary~\ref{prop:traceflipgivesnondeg} we will prove that if the symmetry property
\be\label{eq:klappe}
\left\langle \, 
\begin{tikzpicture}[very thick,scale=0.6,color=blue!50!black, baseline]
\draw (0,0) circle (1.5);
\draw[<-, very thick] (0.100,-1.5) -- (-0.101,-1.5) node[above] {}; 
\draw[<-, very thick] (-0.100,1.5) -- (0.101,1.5) node[below] {}; 
\fill (0:1.5) circle (3.3pt) node[left] {{\small$\Psi$}};
\fill (270:1.5) circle (0pt) node[above] {{\small$X$}};
\end{tikzpicture} 
\, \right\rangle_{\raisemath{10pt}{\!\!\!a}}
=
\left\langle \, 
\begin{tikzpicture}[very thick,scale=0.6,color=blue!50!black, baseline]
\draw (0,0) circle (1.5);
\draw[->, very thick] (0.100,-1.5) -- (-0.101,-1.5) node[above] {}; 
\draw[->, very thick] (-0.100,1.5) -- (0.101,1.5) node[below] {}; 
\fill (180:1.5) circle (3.3pt) node[right] {{\small$\Psi$}};
\fill (270:1.5) circle (0pt) node[above] {{\small$X$}};
\end{tikzpicture} 
\, \right\rangle_{\raisemath{10pt}{\!\!\!b}}
\ee
holds for all 2-morphisms $\Psi:X \rightarrow X$ in~$\B$, then nondegeneracy of $\langle-,-\rangle_X$ in~$\B$ implies nondegeneracy of the induced pairing in $\Borb$. 

The condition~\eqref{eq:klappe} appears naturally in the setting of topological field theory. In particular, we will see that if $a\in\B$ gives rise to an open/closed TFT in the way explained in \cite[Sect.\,9]{cm1208.1481} and Section~\ref{subsec:ocTFT}, then $(a,X^\dagger \otimes X)\in \Borb$ also gives rise to an open/closed TFT; this in particular entails a Calabi-Yau category of boundary conditions, and that the Cardy condition is satisfied. 

\medskip

Let us turn to a brief discussion of applications of the general theory outlined so far. This means that we have to identify interesting pivotal bicategories~$\B$ with idempotent complete 1-morphism categories. As already mentioned one obvious class of examples can be constructed from functors defining 2d TFTs with defects. More generally this construction also works for topological defects in non-topological 2d QFTs \cite{dkr1107.0495}, or, for that matter, in QFTs of any dimension by inflating defect bubbles until the worldvolume is filled with a defect foam. 

The examples that we will study in some detail are \textsl{Landau-Ginzburg models}. They form a bicategory $\LG$ whose objects are potentials~$W$ (i.\,e.~certain polynomials), and 1-morphisms are matrix factorisations of potential differences \cite{br0707.0922, bfk1105.3177, krs0810.5415, cm1208.1481}. In \cite{cm1208.1481} it was established that $\LG$ has all the properties we need, including in particular a simple residue formula to easily compute quantum dimensions (even by hand if need be). 

Given a finite group~$G$ that acts on polynomials and leaves~$W$ invariant, one can try to gauge this symmetry. This is the conventional theory of orbifold Landau-Ginzburg models and equivariant matrix factorisations. We will show in Section~\ref{subsec:equiMF} that one naturally recovers this theory by considering a particular orbifold $(W,A_G)\in \LGeq$, where $A_G$ is the sum of all $G$-twists of the identity defect~$I_W$. 

Assume now that $(W,A_G)$ is in $\LGorb$ and not only in $\LGeq$, i.\,e.~$A_G$ is symmetric. In addition to reformulating ordinary Landau-Ginzburg orbifolds in terms of defects, we also present a general proof that equivariant matrix factorisations form a Calabi-Yau category in this framework. Even better, by applying the general result (Theorem~\ref{thm:LGorbiTFT}) that every $(W,A)\in\LGorb$ gives rise to an open/closed TFT, we find (Theorem~\ref{thm:WAGocTFT}) that the unorbifolded Kapustin-Li pairing \cite{kl0305, hl0404} induces a nondegenerate pairing on $G$-equivariant matrix factorisations. Similarly, we give a conceptual, non-technical proof of the $G$-equivariant Cardy condition, independent of the proof in \cite[Thm.\,4.2.1]{pv1002.2116}. 

It is clear from our general discussion that the procedure of orbifold completion goes beyond ordinary orbifolds. In the case of Landau-Ginzburg models we will illustrate this by giving two examples of equivalences of type~\eqref{eq:aXXbI}: in Section~\ref{subsec:Knoerrer} we explain how to prove Kn\"orrer periodicity as a generalised orbifold equivalence, and in Section~\ref{subsec:minimalmodels} we discuss defects between the categories of matrix factorisations of A- and D-type singularities. In particular, we construct a matrix factorisation $A_d\in \LG(W^{(\mathrm{A}_{2d-1})}, W^{(\mathrm{A}_{2d-1})})$ where $W^{(\mathrm{A}_{2d-1})} = x^{2d} - y^2$ such that its modules are equivalent to matrix factorisations of $W^{(\mathrm{D}_{d+1})} = x^d - xy^2$: 
\be
\hmf(\C[x,y], W^{(\mathrm{D}_{d+1})})^\omega \cong \modu(A_d) \, . 
\ee
We expect that many other such equivalences can be found as a generalised orbifold construction. 

Another class of examples to which our orbifold theory can be immediately applied are B-twisted sigma models. The relevant bicategory is that of spaces and Fourier-Mukai kernels, which by the work of \cite{cw1007.2679} has all the properties we need. Similarly, one would expect A-twisted sigma models to provide another manifestation of orbifold completion. The relevant bicategories are studied in 
\cite{ww0708.2851, MWWunpublished}, 
but it is presently not known if or how they are pivotal.\footnote{We thank Sheel Ganatra for a guide to the A-model literature.} On the other hand, by including defects in the discussion of homological mirror symmetry, one would expect an equivalence of (the orbifold completion of) A- and B-models as monoidal pivotal bicategories with additional enrichments generalising the Calabi-Yau $A_\infty$-structure. 

Finally, a class of non-supersymmetric theories to which orbifold completion is applicable are bosonic sigma models with symmetry defects. The classical action on a worldsheet with defect network can be defined in terms of gerbes and 1- and 2-morphisms between them \cite{s0004117, w0702652, fsw0703145, rs0808.1419}. For invertible defects one can formulate defect fusion via composition of (invertible) 1-morphisms. In this way one obtains a 2-groupoid which can serve as the input for our orbifold construction (after completion with respect to direct sums).

\subsubsection*{Acknowledgements}

We are grateful for helpful discussions with Ilka Brunner, Daniel Plencner and Rafa\l\ Suszek, and it is a particular pleasure to thank Daniel Murfet for his strong interest in and valuable feedback on the contents of this paper. 
Furthermore, we are most grateful to the two referees, whose numerous suggestions and additional explanations helped to improve the paper significantly.
Both authors were supported by the German Science Foundation (DFG) within the Collaborative Research Center 676 ``Particles, Strings and the Early Universe'', and by ITGP (Interactions of Low-Dimensional Topology and Geometry with Mathematical Physics), an ESF RNP.

\section{Algebraic Background}

In this section we review several types of algebra objects and their (bi)modules, in the setting of bicategories with adjoints. Throughout we employ the efficient language of string diagrams, which make manifest the natural interpretation of modules as boundary conditions, and algebras and bimodules as defect lines. 

\subsection{Bicategories with adjoints}\label{subsec:bicatadj}

We begin by recalling the basic definitions and fix our notation. The data of a \textsl{bicategory}~$\B$ is as follows. There is a class of objects~$a$, for which we write $a\in\B$. For all pairs $a,b\in\B$ there is a category $\B(a,b)$ whose objects and arrows are called \textsl{1-morphisms} and \textsl{2-morphisms}, respectively. 1-morphisms can be composed using the functors
\be\label{eq:cabc}
\kappa_{abc} : \B(b,c) \times \B(a,b) \lra \B(a,c) 
\ee
for every $a,b,c\in\B$. For $X,X'\in\B(a,b)$, $Y,Y'\in\B(b,c)$ and 2-morphisms $\phi:X\rightarrow X'$, $\psi: Y\rightarrow Y'$ we write $Y\otimes X = \kappa_{abc}(Y,X)$ and $\psi\otimes\phi = \kappa_{abc}(\psi,\phi)$. This product is associative and unital in the following sense: for any triple of composable 1-morphisms $X,Y,Z$ there is a 2-isomorphism $\alpha_{XYZ}: (X\otimes Y)\otimes Z \rightarrow X\otimes(Y\otimes Z)$, called the \textsl{associator}, which is natural with respect to 2-morphisms in all three arguments. Furthermore, for every $a\in\B$ there is the \textsl{unit 1-morphism} $I_a\in\B(a,a)$ together with natural isomorphisms
\be
\lambda_X: I_b \otimes X \lra X
\, , \quad
\rho_X: X\otimes I_a \lra X
\ee
for every $X\in\B(a,b)$, called (left and right) \textsl{unit actions}. These data satisfy two coherence axioms which are e.\,g.~written out in~\cite[(7.18),\,(7.19)]{bor94}. 

We are interested in bicategories which have additional structure such as duals on the level of 1-morphisms. More precisely, an \textsl{adjunction} $Y \dashv X$ in a bicategory~$\B$ is a pair of 1-morphisms $X\in \B(a,b)$ and $Y \in \B(b,a)$ together with 2-morphisms (called \textsl{adjunction maps}) $\varepsilon: Y\otimes X \rightarrow I_a$ and $\eta: I_b \rightarrow X \otimes Y$ which satisfy the constraints 
\begin{align}
& \rho_X \circ (1_X \otimes \varepsilon) \circ \alpha_{X Y X} \circ (\eta \otimes 1_X) \circ \lambda_X^{-1} = 1_X \, , \label{eq:adjZ1} \\
& \lambda_{Y} \circ( \varepsilon \otimes 1_{Y}) \circ \alpha_{Y X Y}^{-1} \circ (1_{Y} \otimes \eta) \circ \rho_{Y}^{-1} = 1_{Y} \, . \label{eq:adjZ2}
\end{align}
In this situation we say that~$Y$ is \textsl{left adjoint} to~$X$, and~$X$ is \textsl{right adjoint} to~$Y$. 
Note that if~$\B$ is the bicategory of categories 
one recovers the usual notion of adjoint functors.\footnote{We stress that the choice of adjunction maps $\varepsilon, \eta$ is part of the data of an adjunction $Y\dashv X$; `twisting' $\varepsilon, \eta$ to $\varepsilon \circ (\psi \otimes \varphi), (\varphi^{-1} \otimes \psi^{-1}) \circ \eta$ by any 2-automorphisms~$\phi$ of~$X$ and~$\psi$ of~$Y$ produces another adjunction.\label{fn:twisting}} 

We say that~$\B$ is a \textsl{bicategory with left adjoints} if every $X\in\B(a,b)$ comes with $\dX \in\B(b,a)$ and a choice of adjunction $\dX \dashv X$. 
In this case we reserve special names for the adjunction maps $\varepsilon, \eta$, respectively, 
\be
\ev_X: \dX  \otimes X \lra I_a
\, , \quad
\coev_X: I_b \lra X \otimes \dX 
\ee
such that the constraints~\eqref{eq:adjZ1}, \eqref{eq:adjZ2} now read
\begin{align}
& \rho_X \circ (1_X \otimes \ev_X) \circ \alpha_{X {}^\dagger\!X X} \circ (\coev_X \otimes 1_X) \circ \lambda_X^{-1} = 1_X \, , \label{eq:Z1} \\
& \lambda_{\dX} \circ( \ev_X \otimes 1_{\dX}) \circ \alpha_{\dX X \dX}^{-1} \circ (1_{\dX} \otimes \coev_X) \circ \rho_{\dX}^{-1} = 1_{\dX} \, . \label{eq:Z2}
\end{align}
Similarly, $\B$ \textsl{is a bicategory with right adjoints} if every $X\in \B(a,b)$ comes with a choice of $X^\dagger\in \B(b,a)$ and an adjunction $X \dashv X^\dagger$. We denote the adjunction maps by
\be
\tev_X: X \otimes X^\dagger \lra I_b
\, , \quad
\tcoev_X: I_a \lra X^\dagger \otimes X \, . 
\ee
If $\B$ has left and right adjoints, we say it is a \textsl{bicategory with adjoints}.


The conditions imposed on the evaluation and coevaluation maps are conveniently presentable in the diagrammatic notation introduced in \cite{JSGoTCII}. We recall that (in obvious analogy to punctured worldsheets with defects) for this purpose objects in~$\B$ are associated to two-dimensional regions on the plane, 1-morphisms label lines separating these regions, and 2-morphisms correspond to vertices in the resulting network of lines. In this way any 2-morphism can be represented by such a \textsl{string diagram}, for which we adopt the convention that composition and tensoring are denoted vertically and horizontally, respectively, and we always read diagrams from bottom to top and from right to left. For a detailed discussion of string diagrams we refer e.\,g.~to \cite{ladia}. 

Using the diagrammatic language, the adjunction maps are given by 
\be
\ev_X = \!
\begin{tikzpicture}[very thick,scale=1.0,color=blue!50!black, baseline=.4cm]
\draw[line width=0pt] 
(3,0) node[line width=0pt] (D) {{\small$X\vphantom{X^\dagger}$}}
(2,0) node[line width=0pt] (s) {{\small$\dX$}}; 
\draw[directed] (D) .. controls +(0,1) and +(0,1) .. (s);
\end{tikzpicture}
, \quad
\coev_X =  \!
\begin{tikzpicture}[very thick,scale=1.0,color=blue!50!black, baseline=-.4cm,rotate=180]
\draw[line width=0pt] 
(3,0) node[line width=0pt] (D) {{\small$X\vphantom{X^\dagger}$}}
(2,0) node[line width=0pt] (s) {{\small$\dX$}}; 
\draw[redirected] (D) .. controls +(0,1) and +(0,1) .. (s);
\end{tikzpicture}
, \quad
\tev_X =  \!
\begin{tikzpicture}[very thick,scale=1.0,color=blue!50!black, baseline=.4cm]
\draw[line width=0pt] 
(3,0) node[line width=0pt] (D) {{\small$X^\dagger$}}
(2,0) node[line width=0pt] (s) {{\small$X\vphantom{X^\dagger}$}}; 
\draw[redirected] (D) .. controls +(0,1) and +(0,1) .. (s);
\end{tikzpicture}
, \quad
\tcoev_X =  \!
\begin{tikzpicture}[very thick,scale=1.0,color=blue!50!black, baseline=-.4cm,rotate=180]
\draw[line width=0pt] 
(3,0) node[line width=0pt] (D) {{\small$X^\dagger$}}
(2,0) node[line width=0pt] (s) {{\small$X\vphantom{X^\dagger}$}}; 
\draw[directed] (D) .. controls +(0,1) and +(0,1) .. (s);
\end{tikzpicture}
\ee
where we follow the rule to typically not display the units $I_a, I_b$. The defining properties~\eqref{eq:Z1}, \eqref{eq:Z2} for $\ev, \coev$ translate to
\be\label{eq:LeftZorroMoves}
\begin{tikzpicture}[very thick,scale=0.85,color=blue!50!black, baseline=0cm]
\draw[line width=0] 
(-1,1.25) node[line width=0pt] (A) {{\small $X$}}
(1,-1.25) node[line width=0pt] (A2) {{\small $X$}}; 
\draw[directed] (0,0) .. controls +(0,-1) and +(0,-1) .. (-1,0);
\draw[directed] (1,0) .. controls +(0,1) and +(0,1) .. (0,0);
\draw (-1,0) -- (A); 
\draw (1,0) -- (A2); 
\end{tikzpicture}
=
\begin{tikzpicture}[very thick,scale=0.85,color=blue!50!black, baseline=0cm]
\draw[line width=0] 
(0,1.25) node[line width=0pt] (A) {{\small $X$}}
(0,-1.25) node[line width=0pt] (A2) {{\small $X$}}; 
\draw (A2) -- (A); 
\end{tikzpicture}
\, , \quad
\begin{tikzpicture}[very thick,scale=0.85,color=blue!50!black, baseline=0cm]
\draw[line width=0] 
(1,1.25) node[line width=0pt] (A) {{\small $\dX$}}
(-1,-1.25) node[line width=0pt] (A2) {{\small $\dX$}}; 
\draw[directed] (0,0) .. controls +(0,1) and +(0,1) .. (-1,0);
\draw[directed] (1,0) .. controls +(0,-1) and +(0,-1) .. (0,0);
\draw (-1,0) -- (A2); 
\draw (1,0) -- (A); 
\end{tikzpicture}
=
\begin{tikzpicture}[very thick,scale=0.85,color=blue!50!black, baseline=0cm]
\draw[line width=0] 
(0,1.25) node[line width=0pt] (A) {{\small $\dX$}}
(0,-1.25) node[line width=0pt] (A2) {{\small $\dX$}}; 
\draw (A2) -- (A); 
\end{tikzpicture} 
\ee
and their analogues for $\tev, \tcoev$ read
\be\label{eq:otherZorro}
\begin{tikzpicture}[very thick,scale=0.85,color=blue!50!black, baseline=0cm]
\draw[line width=0] 
(1,1.25) node[line width=0pt] (A) {{\small $X$}}
(-1,-1.25) node[line width=0pt] (A2) {{\small $X$}}; 
\draw[redirected] (0,0) .. controls +(0,1) and +(0,1) .. (-1,0);
\draw[redirected] (1,0) .. controls +(0,-1) and +(0,-1) .. (0,0);
\draw (-1,0) -- (A2); 
\draw (1,0) -- (A); 
\end{tikzpicture}
=
\begin{tikzpicture}[very thick,scale=0.85,color=blue!50!black, baseline=0cm]
\draw[line width=0] 
(0,1.25) node[line width=0pt] (A) {{\small $X$}}
(0,-1.25) node[line width=0pt] (A2) {{\small $X$}}; 
\draw (A2) -- (A); 
\end{tikzpicture}
\, , \qquad
\begin{tikzpicture}[very thick,scale=0.85,color=blue!50!black, baseline=0cm]
\draw[line width=0] 
(-1,1.25) node[line width=0pt] (A) {{\small $X^\dagger$}}
(1,-1.25) node[line width=0pt] (A2) {{\small $X^\dagger$}}; 
\draw[redirected] (0,0) .. controls +(0,-1) and +(0,-1) .. (-1,0);
\draw[redirected] (1,0) .. controls +(0,1) and +(0,1) .. (0,0);
\draw (-1,0) -- (A); 
\draw (1,0) -- (A2); 
\end{tikzpicture}
=
\begin{tikzpicture}[very thick,scale=0.85,color=blue!50!black, baseline=0cm]
\draw[line width=0] 
(0,1.25) node[line width=0pt] (A) {{\small $X^\dagger$}}
(0,-1.25) node[line width=0pt] (A2) {{\small $X^\dagger$}}; 
\draw (A2) -- (A); 
\end{tikzpicture}
\, . 
\ee
Note that in these \textsl{Zorro moves} \cite{ZorroStory} we do not label the cups and caps; rather, which adjunction map they depict must be read off from the labels~$X$,~$\dX$ or $X^\dagger$ of the arc, and the orientation of the associated arrow. We will follow this convention for most string diagrams; the only deviation that we allow is the case of closed loops in string diagrams, which as in~\eqref{eq:traces} below we simply label by the 1-morphism associated to their upward-oriented part. 

It is natural to ask for the relation between left and right adjoints. One case of interest is when they coincide, i.\,e.~$\dX  = X^\dagger$. Under this assumption we say that a bicategory~$\B$
is \textsl{pivotal} if the chosen adjunctions satisfy
\be\label{eq:pivotal}
\begin{tikzpicture}[very thick,scale=0.85,color=blue!50!black, baseline=0cm]
\draw[line width=0] 
(-1,1.25) node[line width=0pt] (A) {{\small $Z^\dagger$}}
(1,-1.25) node[line width=0pt] (A2) {{\small $X^\dagger$}}; 
\fill (0,0) circle (2.5pt) node[left] {{\small $\phi$}};
\draw[redirected] (0,0) .. controls +(0,-1) and +(0,-1) .. (-1,0);
\draw[redirected] (1,0) .. controls +(0,1) and +(0,1) .. (0,0);
\draw (-1,0) -- (A); 
\draw (1,0) -- (A2); 
\end{tikzpicture}
\! = \!
\begin{tikzpicture}[very thick,scale=0.85,color=blue!50!black, baseline=0cm]
\draw[line width=0] 
(1,1.25) node[line width=0pt] (A) {{\small $Z^\dagger$}}
(-1,-1.25) node[line width=0pt] (A2) {{\small $X^\dagger$}}; 
\draw[directed] (0,0) .. controls +(0,1) and +(0,1) .. (-1,0);
\draw[directed] (1,0) .. controls +(0,-1) and +(0,-1) .. (0,0);
\fill (0,0) circle (2.5pt) node[right] {{\small $\phi$}};
\draw (-1,0) -- (A2); 
\draw (1,0) -- (A); 
\end{tikzpicture}
, \quad
\begin{tikzpicture}[very thick,scale=0.65,color=blue!50!black, baseline=-0.2cm, rotate=180]
\draw[line width=1pt] 
(2,-1.5) node[line width=0pt] (Y) {{\small $Y^\dagger$}}
(3,-1.5) node[line width=0pt] (X) {{\small $X^\dagger$}}
(-1,2) node[line width=0pt] (XY) {{\small $(Y\otimes X)^\dagger$}}; 
\draw[redirected] (1,0) .. controls +(0,1) and +(0,1) .. (2,0);
\draw[redirected] (0,0) .. controls +(0,2) and +(0,2) .. (3,0);
\draw[redirected] (-1,0) .. controls +(0,-1) and +(0,-1) .. (0.5,0);
\draw (2,0) -- (Y);
\draw (3,0) -- (X);
\draw[dotted] (0,0) -- (1,0);
\draw (-1,0) -- (XY);
\end{tikzpicture}
\! =  \!
\begin{tikzpicture}[very thick,scale=0.65,color=blue!50!black, baseline=-0.2cm, rotate=180]
\draw[line width=1pt] 
(-2,-1.5) node[line width=0pt] (X) {{\small $Y^\dagger$}}
(-1,-1.5) node[line width=0pt] (Y) {{\small $X^\dagger$}}
(2,2) node[line width=0pt] (XY) {{\small $(Y\otimes X)^\dagger$}}; 
\draw[redirected] (0,0) .. controls +(0,1) and +(0,1) .. (-1,0);
\draw[redirected] (1,0) .. controls +(0,2) and +(0,2) .. (-2,0);
\draw[redirected] (2,0) .. controls +(0,-1) and +(0,-1) .. (0.5,0);
\draw (-1,0) -- (Y);
\draw (-2,0) -- (X);
\draw[dotted] (0,0) -- (1,0);
\draw (2,0) -- (XY);
\end{tikzpicture}
\ee
whenever these diagrams make sense. One can show that in a pivotal bicategory the adjunctions determine natural monoidal isomorphisms $\{ \delta_X \}$ between the functor $(-)^{\dagger\dagger}$ and the identity on $\B(a,b)$, see e.\,g.~\cite[Sect.\,2.3]{cr1006.5609}.\footnote{We presented pivotality as a property of a bicategory with adjoints. An equivalent way to define pivotal bicategories is to start from a bicategory with only right adjoints, say, and to endow it with the extra structure of natural monoidal isomorphisms $\{ \delta_X \}$ as above (which may or may not exist).}
Given a 1-morphism $X\in\B(a,b)$ with $\dX  = X^\dagger$ and a 2-morphism $\phi: X\rightarrow X$, we define the latter's \textsl{left} and \textsl{right trace} to be the 2-morphisms
\be\label{eq:traces}
\tr_{\mathrm{l}}(\phi) = 
\begin{tikzpicture}[very thick,scale=0.45,color=blue!50!black, baseline]
\draw (0,0) circle (2);
\draw[<-, very thick] (0.100,-2) -- (-0.101,-2) node[above] {}; 
\draw[<-, very thick] (-0.100,2) -- (0.101,2) node[below] {}; 
\fill (0:2) circle (3.9pt) node[left] {{\small$\phi$}};
\fill (270:2) circle (0pt) node[above] {{\small$X$}};
\end{tikzpicture} 
\, , \quad 
\tr_{\mathrm{r}}(\phi) = 
\begin{tikzpicture}[very thick,scale=0.45,color=blue!50!black, baseline]
\draw (0,0) circle (2);
\draw[->, very thick] (0.100,-2) -- (-0.101,-2) node[above] {}; 
\draw[->, very thick] (0.100,2) -- (0.101,2) node[below] {}; 
\fill (180:2) circle (3.9pt) node[right] {{\small$\phi$}};
\fill (270:2) circle (0pt) node[above] {{\small$X$}};
\end{tikzpicture} 
\, , 
\ee
which are elements of $\End(I_a)$ and $\End(I_b)$, respectively. The special cases $\diml(X) := \tr_{\mathrm{l}}(1_X)$ and $\dimr(X) := \tr_{\mathrm{r}}(1_X)$ are the \textsl{left} and \textsl{right quantum dimensions} of~$X$.

\subsection{Algebras and bimodules}\label{subsec:algebrasbimodules}

Let $\mathcal C$ be a monoidal category, whose unit we denote~$I$. In our later discussions~$\mathcal C$ will be $\B(a,a)$ for some bicategory~$\B$ and some $a\in \B$. 

An object $A\in \mathcal C$ is an \textsl{algebra} if it comes with an associative product and a unit, i.\,e.~with maps
\be
\mu = 
\begin{tikzpicture}[very thick,scale=0.75,color=green!50!black, baseline=0.4cm]
\draw[-dot-] (3,0) .. controls +(0,1) and +(0,1) .. (2,0);
\draw (2.5,0.75) -- (2.5,1.5); 
\end{tikzpicture} 
: A\otimes A \lra A
\, , \quad 
\eta = 
\begin{tikzpicture}[very thick,scale=0.75,color=green!50!black, baseline]
\draw (0,-0.5) node[Odot] (D) {}; 
\draw (D) -- (0,0.6); 
\end{tikzpicture} 
: I \lra A
\ee
which satisfy
\be\label{eq:algcond}
\begin{tikzpicture}[very thick,scale=0.75,color=green!50!black, baseline=0.6cm]
\draw[-dot-] (3,0) .. controls +(0,1) and +(0,1) .. (2,0);
\draw[-dot-] (2.5,0.75) .. controls +(0,1) and +(0,1) .. (3.5,0.75);
\draw (3.5,0.75) -- (3.5,0); 
\draw (3,1.5) -- (3,2.25); 
\end{tikzpicture} 
=
\begin{tikzpicture}[very thick,scale=0.75,color=green!50!black, baseline=0.6cm]
\draw[-dot-] (3,0) .. controls +(0,1) and +(0,1) .. (2,0);
\draw[-dot-] (2.5,0.75) .. controls +(0,1) and +(0,1) .. (1.5,0.75);
\draw (1.5,0.75) -- (1.5,0); 
\draw (2,1.5) -- (2,2.25); 
\end{tikzpicture} 
\, , \quad
\begin{tikzpicture}[very thick,scale=0.75,color=green!50!black, baseline=-0.2cm]
\draw (-0.5,-0.5) node[Odot] (unit) {}; 
\fill (0,0.6) circle (2pt) node (meet) {};
\draw (unit) .. controls +(0,0.5) and +(-0.5,-0.5) .. (0,0.6);
\draw (0,-1) -- (0,1); 
\end{tikzpicture} 
=
\begin{tikzpicture}[very thick,scale=0.75,color=green!50!black, baseline=-0.2cm]
\draw (0,-1) -- (0,1); 
\end{tikzpicture} 
=
\begin{tikzpicture}[very thick,scale=0.75,color=green!50!black, baseline=-0.2cm]
\draw (0.5,-0.5) node[Odot] (unit) {}; 
\fill (0,0.6) circle (2pt) node (meet) {};
\draw (unit) .. controls +(0,0.5) and +(0.5,-0.5) .. (0,0.6);
\draw (0,-1) -- (0,1); 
\end{tikzpicture} 
\, . 
\ee
Note that we reserve a distinguished appereance for algebras in string diagrams. This allows us to refrain from displaying labels for arcs. Since we will never have to display more than one algebra per object $a\in \B$ at a time, this will be no source of confusion. 

Dually, we call~$A$ a \textsl{coalgebra} if it comes with maps
\be
\Delta = 
\begin{tikzpicture}[very thick,scale=0.75,color=green!50!black, baseline=-0.6cm, rotate=180]
\draw[-dot-] (3,0) .. controls +(0,1) and +(0,1) .. (2,0);
\draw (2.5,0.75) -- (2.5,1.5); 
\end{tikzpicture} 
: A \lra A \otimes A
\, , \quad
\varepsilon = 
\begin{tikzpicture}[very thick,scale=0.75,color=green!50!black, baseline, rotate=180]
\draw (0,-0.5) node[Odot] (D) {}; 
\draw (D) -- (0,0.6); 
\end{tikzpicture} 
: A \lra I
\ee
that satisfy the conditions~\eqref{eq:algcond} turned upside-down. 

\begin{definition}\label{def:algebra-properties}
Let $A \in \mathcal C$ have both an algebra and a coalgebra structure. 
\begin{enumerate}
\item 
$A$ is \textsl{Frobenius} if
\be\label{eq:Frob}
\begin{tikzpicture}[very thick,scale=0.85,color=green!50!black, baseline=0cm]
\draw[-dot-] (0,0) .. controls +(0,-1) and +(0,-1) .. (-1,0);
\draw[-dot-] (1,0) .. controls +(0,1) and +(0,1) .. (0,0);
\draw (-1,0) -- (-1,1.5); 
\draw (1,0) -- (1,-1.5); 
\draw (0.5,0.8) -- (0.5,1.5); 
\draw (-0.5,-0.8) -- (-0.5,-1.5); 
\end{tikzpicture}
=
\begin{tikzpicture}[very thick,scale=0.85,color=green!50!black, baseline=0cm]
\draw[-dot-] (0,1.5) .. controls +(0,-1) and +(0,-1) .. (1,1.5);
\draw[-dot-] (0,-1.5) .. controls +(0,1) and +(0,1) .. (1,-1.5);
\draw (0.5,-0.8) -- (0.5,0.8); 
\end{tikzpicture}
=
\begin{tikzpicture}[very thick,scale=0.85,color=green!50!black, baseline=0cm]
\draw[-dot-] (0,0) .. controls +(0,1) and +(0,1) .. (-1,0);
\draw[-dot-] (1,0) .. controls +(0,-1) and +(0,-1) .. (0,0);
\draw (-1,0) -- (-1,-1.5); 
\draw (1,0) -- (1,1.5); 
\draw (0.5,-0.8) -- (0.5,-1.5); 
\draw (-0.5,0.8) -- (-0.5,1.5); 
\end{tikzpicture}
\, . 
\ee
\item 
$A$ is \textsl{$\Delta$-separable} if 
\be\label{eq:special}
\begin{tikzpicture}[very thick,scale=0.85,color=green!50!black, baseline=0cm]
\draw[-dot-] (0,0) .. controls +(0,-1) and +(0,-1) .. (1,0);
\draw[-dot-] (0,0) .. controls +(0,1) and +(0,1) .. (1,0);
\draw (0.5,-0.8) -- (0.5,-1.2); 
\draw (0.5,0.8) -- (0.5,1.2); 
\end{tikzpicture}
= 
\begin{tikzpicture}[very thick,scale=0.85,color=green!50!black, baseline=0cm]
\draw (0.5,-1.2) -- (0.5,1.2); 
\end{tikzpicture}
\, . 
\ee
By slight misuse of language, in the following we will refer to this property simply as \textsl{separable}. 
\item 
Suppose~$\mathcal C$ is pivotal. Then we call~$A$ \textsl{symmetric} if
\be\label{eq:Asymmetric}
\begin{tikzpicture}[very thick,scale=0.85,color=green!50!black, baseline=0cm]
\draw[-dot-] (0,0) .. controls +(0,1) and +(0,1) .. (-1,0);
\draw[directedgreen, color=green!50!black] (1,0) .. controls +(0,-1) and +(0,-1) .. (0,0);
\draw (-1,0) -- (-1,-1.5); 
\draw (1,0) -- (1,1.5); 
\draw (-0.5,1.2) node[Odot] (end) {}; 
\draw (-0.5,0.8) -- (end); 
\end{tikzpicture}
= 
\begin{tikzpicture}[very thick,scale=0.85,color=green!50!black, baseline=0cm]
\draw[redirectedgreen, color=green!50!black] (0,0) .. controls +(0,-1) and +(0,-1) .. (-1,0);
\draw[-dot-] (1,0) .. controls +(0,1) and +(0,1) .. (0,0);
\draw (-1,0) -- (-1,1.5); 
\draw (1,0) -- (1,-1.5); 
\draw (0.5,1.2) node[Odot] (end) {}; 
\draw (0.5,0.8) -- (end); 
\end{tikzpicture}
\ee
as maps $A\rightarrow A^\dagger$. 
\end{enumerate}
\end{definition}

From now on we assume that we are given a bicategory~$\B$ and an algebra object~$A$ in $\mathcal C=\B(a,a)$ for some $a\in\B$. A \textsl{left $A$-module} is a 1-morphism $X\in \B(b,a)$ for some $b\in \B$, together with a left action of~$A$ compatible with multiplication: 
\be\label{eq:leftcomp}
\begin{tikzpicture}[very thick,scale=0.75,color=blue!50!black, baseline]

\draw (0,-1) node[right] (X) {{\small$X$}};
\draw (0,1) node[right] (Xu) {{\small$X$}};

\draw (0,-1) -- (0,1); 

\fill[color=green!50!black] (0,0.0) circle (2pt) node (meet) {};
\draw[color=green!50!black] (-0.5,-1) .. controls +(0,0.5) and +(-0.5,-0.5) .. (0,0.0);
\end{tikzpicture} 
: A \otimes X \lra X 
\, , \quad 
\begin{tikzpicture}[very thick,scale=0.75,color=blue!50!black, baseline]

\draw (0,-1) node[right] (X) {{\small$X$}};
\draw (0,1) node[right] (Xu) {{\small$X$}};

\draw (0,-1) -- (0,1); 

\fill[color=green!50!black] (0,-0.25) circle (2pt) node (meet) {};
\fill[color=green!50!black] (0,0.75) circle (2pt) node (meet2) {};
\draw[color=green!50!black] (-0.5,-1) .. controls +(0,0.25) and +(-0.25,-0.25) .. (0,-0.25);
\draw[color=green!50!black] (-1,-1) .. controls +(0,0.5) and +(-0.5,-0.5) .. (0,0.75);
\end{tikzpicture}
=
\begin{tikzpicture}[very thick,scale=0.75,color=blue!50!black, baseline]

\draw (0,-1) node[right] (X) {{\small$X$}};
\draw (0,1) node[right] (Xu) {{\small$X$}};

\draw (0,-1) -- (0,1); 

\fill[color=green!50!black] (0,0.75) circle (2pt) node (meet2) {};

\draw[-dot-, color=green!50!black] (-0.5,-1) .. controls +(0,1) and +(0,1) .. (-1,-1);

\draw[color=green!50!black] (-0.75,-0.2) .. controls +(0,0.5) and +(-0.5,-0.5) .. (0,0.75);
\end{tikzpicture} 
\, , \quad
\begin{tikzpicture}[very thick,scale=0.75,color=blue!50!black, baseline]
\draw (0,-1) -- (0,1); 
\draw (0,-1) node[right] (X) {{\small$X$}};
\draw (0,1) node[right] (Xu) {{\small$X$}};

\draw[color=green!50!black]  (-0.5,-0.5) node[Odot] (unit) {}; 
\fill[color=green!50!black]  (0,0.6) circle (2pt) node (meet) {};
\draw[color=green!50!black]  (unit) .. controls +(0,0.5) and +(-0.5,-0.5) .. (0,0.6);
\end{tikzpicture} 
=
\begin{tikzpicture}[very thick,scale=0.75,color=blue!50!black, baseline]
\draw (0,-1) -- (0,1); 
\draw (0,-1) node[right] (X) {{\small$X$}};
\draw (0,1) node[right] (Xu) {{\small$X$}};
\end{tikzpicture} 
. 
\ee
A 2-morphism $\phi:  X \rightarrow Y$ between left $A$-modules is called a \textsl{module map} if it satisfies
\be\label{eq:modmap}
\begin{tikzpicture}[very thick,scale=0.75,color=blue!50!black, baseline]

\draw (0,-1) node[right] (X) {{\small$X$}};
\draw (0,1) node[right] (Xu) {{\small$Y$}};

\draw (0,-1) -- (0,1); 

\fill (0,-0.3) circle (2pt) node[right] (phi) {{\small $\phi$}};

\fill[color=green!50!black] (0,0.3) circle (2pt) node (meet) {};
\draw[color=green!50!black] (-0.5,-1) .. controls +(0,0.5) and +(-0.5,-0.5) .. (0,0.3);
\end{tikzpicture} 
=
\begin{tikzpicture}[very thick,scale=0.75,color=blue!50!black, baseline]

\draw (0,-1) node[right] (X) {{\small$X$}};
\draw (0,1) node[right] (Xu) {{\small$Y$}};

\draw (0,-1) -- (0,1); 

\fill (0,0.3) circle (2pt) node[right] (phi) {{\small $\phi$}};

\fill[color=green!50!black] (0,-0.3) circle (2pt) node (meet) {};
\draw[color=green!50!black] (-0.5,-1) .. controls +(0,0.25) and +(-0.25,-0.25) .. (0,-0.3);
\end{tikzpicture} 
. 
\ee
We denote the subset in $\Hom_{\B(b,a)}(X,Y)$ of all module maps by $\Hom_A(X,Y)$. 

If~$A$ is also a coalgebra we can consider the map 
\be\label{eq:projA}
\pi_A : \phi \lmt
\begin{tikzpicture}[very thick,scale=0.75,color=blue!50!black, baseline=0cm]

\draw (1.5,-1.75) -- (1.5,1.75); 

\draw (1.5,-1.75) node[right] (X) {{\small$X$}};
\draw (1.5,1.75) node[right] (Xu) {{\small$Y$}};

\fill (1.5,0.7) circle (2pt) node[right] (phi) {{\small $\phi$}};

\fill[color=green!50!black] (1.5,0.3) circle (2pt) node (meet) {};
\fill[color=green!50!black] (1.5,1.1) circle (2pt) node (meet) {};

\draw[color=green!50!black] (1,-0.25) .. controls +(0.0,0.25) and +(-0.25,-0.25) .. (1.5,0.3);
\draw[color=green!50!black] (0.5,-0.25) .. controls +(0.0,0.5) and +(-0.25,-0.25) .. (1.5,1.1);

\draw[-dot-, color=green!50!black] (0.5,-0.25) .. controls +(0,-0.5) and +(0,-0.5) .. (1,-0.25);

\draw[color=green!50!black] (0.75,-1.2) node[Odot] (unit) {}; 
\draw[color=green!50!black] (0.75,-0.6) -- (unit);

\end{tikzpicture}
\ee
which acts on all 2-morphisms $\phi: X\rightarrow Y$ between left $A$-modules. Under the right circumstances this map projects to the set of module maps: 

\begin{lemma}\label{lem:piAproj}
If~$A$ is a separable Frobenius algebra then $\pi_A^2 = \pi_A$ and $\operatorname{im}(\pi_A) = \Hom_A(X,Y)$. 
\end{lemma}

\begin{proof}
$\pi_A$ acts as the identity on module maps. Indeed, for such a map~$\phi$ we have
\be
\pi_A(\phi) =
\begin{tikzpicture}[very thick,scale=0.75,color=blue!50!black, baseline=0cm]

\draw (1.5,-1.75) -- (1.5,1.75); 

\draw (1.5,-1.75) node[right] (X) {{\small$X$}};
\draw (1.5,1.75) node[right] (Xu) {{\small$Y$}};

\fill (1.5,0.7) circle (2pt) node[right] (phi) {{\small $\phi$}};

\fill[color=green!50!black] (1.5,0.3) circle (2pt) node (meet) {};
\fill[color=green!50!black] (1.5,1.1) circle (2pt) node (meet) {};

\draw[color=green!50!black] (1,-0.25) .. controls +(0.0,0.25) and +(-0.25,-0.25) .. (1.5,0.3);
\draw[color=green!50!black] (0.5,-0.25) .. controls +(0.0,0.5) and +(-0.25,-0.25) .. (1.5,1.1);

\draw[-dot-, color=green!50!black] (0.5,-0.25) .. controls +(0,-0.5) and +(0,-0.5) .. (1,-0.25);

\draw[color=green!50!black] (0.75,-1.2) node[Odot] (unit) {}; 
\draw[color=green!50!black] (0.75,-0.6) -- (unit);

\end{tikzpicture}
\stackrel{\eqref{eq:modmap}}{=}
\begin{tikzpicture}[very thick,scale=0.75,color=blue!50!black, baseline=0cm]

\draw (1.5,-1.75) -- (1.5,1.75); 

\draw (1.5,-1.75) node[right] (X) {{\small$X$}};
\draw (1.5,1.75) node[right] (Xu) {{\small$Y$}};

\fill (1.5,1.1) circle (2pt) node[right] (phi) {{\small $\phi$}};

\fill[color=green!50!black] (1.5,0.3) circle (2pt) node (meet) {};
\fill[color=green!50!black] (1.5,0.7) circle (2pt) node (meet) {};

\draw[color=green!50!black] (1,-0.25) .. controls +(0.0,0.25) and +(-0.25,-0.25) .. (1.5,0.3);
\draw[color=green!50!black] (0.5,-0.25) .. controls +(0.0,0.5) and +(-0.25,-0.25) .. (1.5,0.7);

\draw[-dot-, color=green!50!black] (0.5,-0.25) .. controls +(0,-0.5) and +(0,-0.5) .. (1,-0.25);

\draw[color=green!50!black] (0.75,-1.2) node[Odot] (unit) {}; 
\draw[color=green!50!black] (0.75,-0.6) -- (unit);

\end{tikzpicture}
\stackrel{\eqref{eq:leftcomp}}{=}
\begin{tikzpicture}[very thick,scale=0.75,color=blue!50!black, baseline=0cm]

\draw (1.5,-1.75) -- (1.5,1.75); 

\draw (1.5,-1.75) node[right] (X) {{\small$X$}};
\draw (1.5,1.75) node[right] (Xu) {{\small$Y$}};

\fill (1.5,1.1) circle (2pt) node[right] (phi) {{\small $\phi$}};

\fill[color=green!50!black] (1.5,0.7) circle (2pt) node (meet) {};

\draw[color=green!50!black] (0.75, 0.2) .. controls +(0.0,0.25) and +(-0.25,-0.25) .. (1.5,0.7);

\draw[-dot-, color=green!50!black] (0.5,-0.25) .. controls +(0,0.5) and +(0,0.5) .. (1,-0.25);
\draw[-dot-, color=green!50!black] (0.5,-0.25) .. controls +(0,-0.5) and +(0,-0.5) .. (1,-0.25);

\draw[color=green!50!black] (0.75,-1.2) node[Odot] (unit) {}; 
\draw[color=green!50!black] (0.75,-0.6) -- (unit);

\end{tikzpicture}
\stackrel{\eqref{eq:special}}{=}
\begin{tikzpicture}[very thick,scale=0.75,color=blue!50!black, baseline]

\draw (0,-1.75) node[right] (X) {{\small$X$}};
\draw (0,1.75) node[right] (Xu) {{\small$Y$}};

\draw (0,-1.75) -- (0,1.75); 

\fill[color=green!50!black] (0,0.7) circle (2pt) node (meet) {};

\fill (0,1.1) circle (2pt) node[right] (phi) {{\small $\phi$}};

\draw[color=green!50!black] (-0.5,-1) node[Odot] (unit) {}; 
\draw[color=green!50!black] (unit) .. controls +(0,0.25) and +(-0.25,-0.25) .. (0,0.7);

\end{tikzpicture} 
\stackrel{\eqref{eq:leftcomp}}{=}
\phi \, . 
\ee

It remains to show that every image under $\pi_A$ is a module map: 
\be
\begin{tikzpicture}[very thick,scale=0.75,color=blue!50!black, baseline=0cm]

\draw (1.5,-1.75) -- (1.5,2); 

\draw (1.5,-1.75) node[right] (X) {{\small$X$}};
\draw (1.5,2) node[right] (Xu) {{\small$Y$}};

\fill (1.5,0.7) circle (2pt) node[right] (phi) {{\small $\phi$}};

\fill[color=green!50!black] (1.5,0.3) circle (2pt) node (meet) {};
\fill[color=green!50!black] (1.5,1.1) circle (2pt) node (meet) {};
\fill[color=green!50!black] (1.5,1.5) circle (2pt) node (meet) {};

\draw[color=green!50!black] (1,-0.25) .. controls +(0.0,0.25) and +(-0.25,-0.25) .. (1.5,0.3);
\draw[color=green!50!black] (0.5,-0.25) .. controls +(0.0,0.5) and +(-0.25,-0.25) .. (1.5,1.1);

\draw[-dot-, color=green!50!black] (0.5,-0.25) .. controls +(0,-0.5) and +(0,-0.5) .. (1,-0.25);

\draw[color=green!50!black] (0.75,-1.2) node[Odot] (unit) {}; 
\draw[color=green!50!black] (0.75,-0.6) -- (unit);

\draw[color=green!50!black] (0,-1.75) .. controls +(0.0,2.75) and +(-0.25,-0.25) .. (1.5,1.5);

\end{tikzpicture}
\stackrel{\eqref{eq:leftcomp}}{=}
\begin{tikzpicture}[very thick,scale=0.75,color=blue!50!black, baseline=0cm]

\draw (1.5,-1.75) -- (1.5,2); 

\draw (1.5,-1.75) node[right] (X) {{\small$X$}};
\draw (1.5,2) node[right] (Xu) {{\small$Y$}};

\fill (1.5,0.7) circle (2pt) node[right] (phi) {{\small $\phi$}};

\fill[color=green!50!black] (1.5,0.3) circle (2pt) node (meet) {};
\fill[color=green!50!black] (1.5,1.1) circle (2pt) node (meet) {};
\fill[color=green!50!black] (0.6,0.12) circle (2pt) node (meet) {};

\draw[color=green!50!black] (1,-0.25) .. controls +(0.0,0.25) and +(-0.25,-0.25) .. (1.5,0.3);
\draw[color=green!50!black] (0.5,-0.25) .. controls +(0.0,0.5) and +(-0.25,-0.25) .. (1.5,1.1);

\draw[-dot-, color=green!50!black] (0.5,-0.25) .. controls +(0,-0.5) and +(0,-0.5) .. (1,-0.25);

\draw[color=green!50!black] (0.75,-1.2) node[Odot] (unit) {}; 
\draw[color=green!50!black] (0.75,-0.6) -- (unit);

\draw[color=green!50!black] (0,-1.75) .. controls +(0.0,1.75) and +(-0.25,-0.25) .. (0.6,0.12);

\end{tikzpicture}
\stackrel{\eqref{eq:Frob}}{=}
\begin{tikzpicture}[very thick,scale=0.75,color=blue!50!black, baseline=0cm]

\draw (1.5,-1.75) -- (1.5,2); 

\draw (1.5,-1.75) node[right] (X) {{\small$X$}};
\draw (1.5,2) node[right] (Xu) {{\small$Y$}};

\fill (1.5,1.1) circle (2pt) node[right] (phi) {{\small $\phi$}};

\fill[color=green!50!black] (1.5,0.7) circle (2pt) node (meet) {};
\fill[color=green!50!black] (1.5,1.5) circle (2pt) node (meet) {};

\draw[color=green!50!black] (1,0.25) .. controls +(0.0,0.25) and +(-0.25,-0.25) .. (1.5,0.7);
\draw[color=green!50!black] (0.5,0.25) .. controls +(0.0,0.5) and +(-0.25,-0.25) .. (1.5,1.5);

\draw[-dot-, color=green!50!black] (0.5,0.25) .. controls +(0,-0.5) and +(0,-0.5) .. (1,0.25);

\draw[-dot-, color=green!50!black] (0.5,-1) .. controls +(0,0.5) and +(0,0.5) .. (1,-1);

\draw[color=green!50!black] (1,-1.2) node[Odot] (unit) {}; 
\draw[color=green!50!black] (1,-1) -- (unit);
\draw[color=green!50!black] (0.5,-1) -- (0.5,-1.75);

\draw[color=green!50!black] (0.75,-0.6) -- (0.75,-0.1);

\end{tikzpicture}
=
\begin{tikzpicture}[very thick,scale=0.75,color=blue!50!black, baseline=0cm]

\draw (1.5,-1.75) -- (1.5,2); 

\draw (1.5,-1.75) node[right] (X) {{\small$X$}};
\draw (1.5,2) node[right] (Xu) {{\small$Y$}};

\fill (1.5,1.1) circle (2pt) node[right] (phi) {{\small $\phi$}};

\fill[color=green!50!black] (1.5,0.7) circle (2pt) node (meet) {};
\fill[color=green!50!black] (1.5,1.5) circle (2pt) node (meet) {};
\fill[color=green!50!black] (1.5,-0.5) circle (2pt) node (meet) {};

\draw[color=green!50!black] (1,0.25) .. controls +(0.0,0.25) and +(-0.25,-0.25) .. (1.5,0.7);
\draw[color=green!50!black] (0.5,0.25) .. controls +(0.0,0.5) and +(-0.25,-0.25) .. (1.5,1.5);

\draw[-dot-, color=green!50!black] (0.5,0.25) .. controls +(0,-0.5) and +(0,-0.5) .. (1,0.25);

\draw[color=green!50!black] (0.75,-0.6) node[Odot] (unit) {}; 
\draw[color=green!50!black] (0.75,-0.1) -- (unit);

\draw[color=green!50!black] (1,-1.75) .. controls +(0.0,0.25) and +(-0.25,-0.25) .. (1.5,-0.5);

\end{tikzpicture}
\ee
where we used \eqref{eq:algcond}, \eqref{eq:Frob} and \eqref{eq:leftcomp} in the last step. 
\end{proof}

Similarly one can work with \textsl{right $A$-modules} and their module maps. We do not spell out the details as they are obtained by simply reflecting all of the above diagrams at the line labelled by the module~$X$. 

A $B$-$A$-\textsl{bimodule} over two algebras $A\in\B(a,a)$ and $B\in\B(b,b)$ is a 1-morphism $X\in\B(a,b)$ that is simultaneously a right $A$-module and left $B$-module, together with the compatibility condition
\be
\begin{tikzpicture}[very thick,scale=0.75,color=blue!50!black, baseline]

\draw (0,-1) node[right] (X) {{\small$X$}};
\draw (0,1) node[right] (Xu) {{\small$X$}};

\draw (1.0,-1) node[right] (A) {{\small$A$}};
\draw (-0.5,-1) node[left] (B) {{\small$B$}};

\draw (0,-1) -- (0,1); 

\fill[color=green!50!black] (0,-0.3) circle (2pt) node (meet) {};
\fill[color=green!50!black] (0,0.3) circle (2pt) node (meet) {};

\draw[color=green!50!black] (-0.5,-1) .. controls +(0,0.25) and +(-0.5,-0.5) .. (0,-0.3);
\draw[color=green!50!black] (1.0,-1) .. controls +(0,0.5) and +(0.5,-0.5) .. (0,0.3);
\end{tikzpicture} 
=
\begin{tikzpicture}[very thick,scale=0.75,color=blue!50!black, baseline]

\draw (0,-1) node[left] (X) {{\small$X$}};
\draw (0,1) node[left] (Xu) {{\small$X$}};

\draw (0.5,-1) node[right] (A) {{\small$A$}};
\draw (-1,-1) node[left] (B) {{\small$B$}};

\draw (0,-1) -- (0,1); 

\fill[color=green!50!black] (0,-0.3) circle (2pt) node (meet) {};
\fill[color=green!50!black] (0,0.3) circle (2pt) node (meet) {};

\draw[color=green!50!black] (-1,-1) .. controls +(0,0.5) and +(-0.5,-0.5) .. (0,0.3);
\draw[color=green!50!black] (0.5,-1) .. controls +(0,0.25) and +(0.5,-0.5) .. (0,-0.3);
\end{tikzpicture} 
. 
\ee
Given two $B$-$A$-bimodules $X,Y$, a 2-morphism $\phi:X \rightarrow Y$ is called a \textsl{bimodule map} if it is both a map of left and right modules. We denote the subset in $\Hom_{\B(a,b)}(X,Y)$ of all bimodule maps by $\Hom_{BA}(X,Y)$. Analogously to~\eqref{eq:projA}, if $A,B$ are separable Frobenius there is a canonical projection to $\Hom_{BA}(X,Y)$ given by
\be\label{eq:bimodulemapprojector}
\pi_{BA}  : \phi \lmt
\begin{tikzpicture}[very thick,scale=0.75,color=blue!50!black, baseline=0cm]

\draw (1.5,-1.75) -- (1.5,1.95); 

\draw (1.5,-1.75) node[right] (X) {{\small$X$}};
\draw (1.5,1.95) node[right] (Xu) {{\small$Y$}};

\fill (1.5,0.7) circle (2pt) node[right] (phi) {{\small $\phi$}};

\fill[color=green!50!black] (1.5,0.3) circle (2pt) node (meet) {};
\fill[color=green!50!black] (1.5,1.1) circle (2pt) node (meet) {};
\fill[color=green!50!black] (1.5,1.5) circle (2pt) node (meet) {};
\fill[color=green!50!black] (1.5,-0.1) circle (2pt) node (meet) {};

\draw[color=green!50!black] (1,-0.25) .. controls +(0.0,0.25) and +(-0.25,-0.25) .. (1.5,0.3);
\draw[color=green!50!black] (0.5,-0.25) .. controls +(0.0,0.5) and +(-0.25,-0.25) .. (1.5,1.1);
\draw[-dot-, color=green!50!black] (0.5,-0.25) .. controls +(0,-0.5) and +(0,-0.5) .. (1,-0.25);

\draw[color=green!50!black] (0.75,-1.2) node[Odot] (unit) {}; 
\draw[color=green!50!black] (0.75,-0.6) -- (unit);

\draw[color=green!50!black] (2.25,-0.45) .. controls +(0.0,0.25) and +(0.25,-0.25) .. (1.5,-0.1);
\draw[color=green!50!black] (2.75,-0.45) .. controls +(0.0,1) and +(0.25,-0.25) .. (1.5,1.5);
\draw[-dot-, color=green!50!black] (2.25,-0.45) .. controls +(0,-0.5) and +(0,-0.5) .. (2.75,-0.45);

\draw[color=green!50!black] (2.5,-1.4) node[Odot] (unit) {}; 
\draw[color=green!50!black] (2.5,-0.8) -- (unit);

\end{tikzpicture}
\, . 
\ee

\subsection{Tensor products}

Let $A\in \B(a,a)$ be an algebra as before, and let $X\in\B(a,b)$, $Y\in \B(c,a)$ be right and left $A$-modules, respectively. We denote the actions of~$A$ by $\rho_X: X\otimes A\rightarrow X$ and $\rho_Y: A\otimes Y \rightarrow Y$. The \textsl{tensor product of~$X$ and~$Y$ over~$A$}, $X\otimes_A Y \in \B(c,b)$, is defined to be the coequaliser of $r = \rho_X \otimes 1_Y$ and $l = (1_X \otimes \rho_Y) \circ \alpha_{XAY}$. Recall that this means that $X\otimes_A Y$ is equipped with a map $\vartheta: X\otimes Y \rightarrow X\otimes_A Y$ with $\vartheta\circ l = \vartheta \circ r$ such that for all $\phi: X\otimes Y \rightarrow Z$ with $\phi \circ l = \phi \circ r$ there is a unique map $\zeta: X\otimes_A Y \rightarrow Z$ with $\zeta \circ \vartheta = \phi$: 
\be\label{eq:coequalisertensor}
\begin{tikzpicture}[
			     baseline=(current bounding box.base), 
			     >=stealth,
			     descr/.style={fill=white,inner sep=2.5pt}, 
			     normal line/.style={->}
			     ] 
\matrix (m) [matrix of math nodes, row sep=3em, column sep=2.5em, text height=1.5ex, text depth=0.25ex] {%
(X \otimes A) \otimes Y & & X \otimes Y & & X \otimes_A Y \\
&& && Z \\ };
\path[font=\scriptsize] (m-1-1) edge[->, transform canvas={yshift=0.5ex}] node[auto] {$ r $} (m-1-3)
				  (m-1-1) edge[->, transform canvas={yshift=-0.5ex}] node[auto, swap] {$ l $} (m-1-3); 
\path[font=\scriptsize] 
				 (m-1-3) edge[->] node[auto] {$ \vartheta $} (m-1-5)
				 (m-1-3) edge[->] node[auto, swap] {$ \phi $} (m-2-5);
\path[font=\scriptsize, dotted] (m-1-5) edge[->] node[auto] {$ \zeta $} (m-2-5);
\end{tikzpicture}
\ee

In general, the tensor product over a given algebra~$A$ may not exist. The following lemma provides a simple existence criterion, which will be sufficient for our purposes. 

\begin{lemma}\label{lem:tensorprojector}
Suppose that idempotent 2-morphisms split in~$\B$ and that~$A$ is separable Frobenius. Then $X\otimes_A Y$ exists for all modules $X,Y$ and can be written as the image of the idempotent
\be\label{eq:piA}
\pi_A^{X,Y} = 
\begin{tikzpicture}[very thick,scale=0.75,color=blue!50!black, baseline]

\draw (-1,-1) node[left] (X) {{\small$X$}};
\draw (1,-1) node[right] (Xu) {{\small$Y$}};

\draw (-1,-1) -- (-1,1); 
\draw (1,-1) -- (1,1); 

\fill[color=green!50!black] (-1,0.6) circle (2pt) node (meet) {};
\fill[color=green!50!black] (1,0.6) circle (2pt) node (meet) {};

\draw[-dot-, color=green!50!black] (0.35,-0.0) .. controls +(0,-0.5) and +(0,-0.5) .. (-0.35,-0.0);

\draw[color=green!50!black] (0.35,-0.0) .. controls +(0,0.25) and +(-0.25,-0.25) .. (1,0.6);
\draw[color=green!50!black] (-0.35,-0.0) .. controls +(0,0.25) and +(0.25,-0.25) .. (-1,0.6);

\draw[color=green!50!black] (0,-0.75) node[Odot] (down) {}; 
\draw[color=green!50!black] (down) -- (0,-0.35); 

\end{tikzpicture} 
. 
\ee
\end{lemma}

\begin{proof}
We compute
\be
\begin{tikzpicture}[very thick,scale=0.75,color=blue!50!black, baseline]

\draw (-1,-1.6) node[left] (X) {{\small$X$}};
\draw (1,-1.6) node[right] (Xu) {{\small$Y$}};

\draw (-1,-1.6) -- (-1,1); 
\draw (1,-1.6) -- (1,1); 

\fill[color=green!50!black] (-1,0.8) circle (2pt) node (meet) {};
\fill[color=green!50!black] (1,0.8) circle (2pt) node (meet) {};

\draw[-dot-, color=green!50!black] (0.35,0.3) .. controls +(0,-0.5) and +(0,-0.5) .. (-0.35,0.3);

\draw[color=green!50!black] (0.35,0.3) .. controls +(0,0.25) and +(-0.25,-0.25) .. (1,0.8);
\draw[color=green!50!black] (-0.35,0.3) .. controls +(0,0.25) and +(0.25,-0.25) .. (-1,0.8);

\draw[color=green!50!black] (0,-0.45) node[Odot] (down) {}; 
\draw[color=green!50!black] (down) -- (0,-0.05);

\fill[color=green!50!black] (-1,-0.2) circle (2pt) node (meet) {};
\fill[color=green!50!black] (1,-0.2) circle (2pt) node (meet) {};

\draw[-dot-, color=green!50!black] (0.35,-0.7) .. controls +(0,-0.5) and +(0,-0.5) .. (-0.35,-0.7);

\draw[color=green!50!black] (0.35,-0.7) .. controls +(0,0.25) and +(-0.25,-0.25) .. (1,-0.2);
\draw[color=green!50!black] (-0.35,-0.7) .. controls +(0,0.25) and +(0.25,-0.25) .. (-1,-0.2);

\draw[color=green!50!black] (0,-1.45) node[Odot] (down) {}; 
\draw[color=green!50!black] (down) -- (0,-1.05); 

\end{tikzpicture} 
\stackrel{\eqref{eq:leftcomp}}{=}
\begin{tikzpicture}[very thick,scale=0.75,color=blue!50!black, baseline]

\draw (-1,-1.6) node[left] (X) {{\small$X$}};
\draw (1,-1.6) node[right] (Xu) {{\small$Y$}};

\draw (-1,-1.6) -- (-1,1); 
\draw (1,-1.6) -- (1,1); 

\fill[color=green!50!black] (-1,0.8) circle (2pt) node (meet) {};
\fill[color=green!50!black] (1,0.8) circle (2pt) node (meet) {};

\draw[-dot-, color=green!50!black] (0.35,0.3) .. controls +(0,-0.5) and +(0,-0.5) .. (-0.35,0.3);

\draw[color=green!50!black] (0.35,0.3) .. controls +(0,0.25) and +(-0.25,-0.25) .. (1,0.8);
\draw[color=green!50!black] (-0.35,0.3) .. controls +(0,0.25) and +(0.25,-0.25) .. (-1,0.8);

\draw[color=green!50!black] (0,-0.45) node[Odot] (down) {}; 
\draw[color=green!50!black] (down) -- (0,-0.05);

\fill[color=green!50!black] (-0.7,0.61) circle (2pt) node (meet) {};
\fill[color=green!50!black] (0.7,0.61) circle (2pt) node (meet) {};

\draw[-dot-, color=green!50!black] (0.35,-0.7) .. controls +(0,-0.5) and +(0,-0.5) .. (-0.35,-0.7);

\draw[color=green!50!black] (0.35,-0.7) .. controls +(0,0.25) and +(-0.25,-0.25) .. (0.7,0.61);
\draw[color=green!50!black] (-0.35,-0.7) .. controls +(0,0.25) and +(0.25,-0.25) .. (-0.7,0.61);

\draw[color=green!50!black] (0,-1.45) node[Odot] (down) {}; 
\draw[color=green!50!black] (down) -- (0,-1.05); 

\end{tikzpicture} 
\stackrel{\eqref{eq:Frob}}{=}
\begin{tikzpicture}[very thick,scale=0.75,color=blue!50!black, baseline]

\draw (-1,-1.6) node[left] (X) {{\small$X$}};
\draw (1,-1.6) node[right] (Xu) {{\small$Y$}};

\draw (-1,-1.6) -- (-1,1); 
\draw (1,-1.6) -- (1,1); 

\fill[color=green!50!black] (-1,0.8) circle (2pt) node (meet) {};
\fill[color=green!50!black] (1,0.8) circle (2pt) node (meet) {};

\draw[-dot-, color=green!50!black] (0.35,0.3) .. controls +(0,-0.5) and +(0,-0.5) .. (-0.35,0.3);

\draw[color=green!50!black] (0.35,0.3) .. controls +(0,0.25) and +(-0.25,-0.25) .. (1,0.8);
\draw[color=green!50!black] (-0.35,0.3) .. controls +(0,0.25) and +(0.25,-0.25) .. (-1,0.8);

\draw[color=green!50!black] (0,-0.35) -- (0,-0.05);

\draw[-dot-, color=green!50!black] (0.35,-0.7) .. controls +(0,0.5) and +(0,0.5) .. (-0.35,-0.7);
\draw[-dot-, color=green!50!black] (0.35,-0.7) .. controls +(0,-0.5) and +(0,-0.5) .. (-0.35,-0.7);

\draw[color=green!50!black] (0,-1.45) node[Odot] (down) {}; 
\draw[color=green!50!black] (down) -- (0,-1.05); 

\end{tikzpicture} 
\stackrel{\eqref{eq:special}}{=}
\begin{tikzpicture}[very thick,scale=0.75,color=blue!50!black, baseline]

\draw (-1,-1.6) node[left] (X) {{\small$X$}};
\draw (1,-1.6) node[right] (Xu) {{\small$Y$}};

\draw (-1,-1.6) -- (-1,1); 
\draw (1,-1.6) -- (1,1); 

\fill[color=green!50!black] (-1,0.8) circle (2pt) node (meet) {};
\fill[color=green!50!black] (1,0.8) circle (2pt) node (meet) {};

\draw[-dot-, color=green!50!black] (0.35,0.3) .. controls +(0,-0.5) and +(0,-0.5) .. (-0.35,0.3);

\draw[color=green!50!black] (0.35,0.3) .. controls +(0,0.25) and +(-0.25,-0.25) .. (1,0.8);
\draw[color=green!50!black] (-0.35,0.3) .. controls +(0,0.25) and +(0.25,-0.25) .. (-1,0.8);

\draw[color=green!50!black] (0,-1.45) node[Odot] (down) {}; 
\draw[color=green!50!black] (down) -- (0,-0.05); 

\end{tikzpicture} 
\ee
and thus find $(\pi_A^{X,Y})^2 = \pi_A^{X,Y}$. Hence there are splitting maps $\xi: X \otimes_A Y \rightarrow X\otimes Y$ and $\vartheta: X\otimes Y \rightarrow X\otimes_A Y$ with $\vartheta \xi = 1$ and $\xi\vartheta = \pi_A^{X,Y}$, and the epimorphism~$\vartheta$ satisfies the universal coequaliser property: for~$\phi$ as in~\eqref{eq:coequalisertensor} we set $\zeta = \phi\xi$. 
\end{proof}

\begin{remark}\label{rem:split-coequaliser}
For a separable Frobenius algebra the pair $l,r: (X\otimes A) \otimes Y \rightrightarrows X \otimes Y$ is \textsl{contractible}, i.\,e.~there exists $t: X\otimes Y \rightarrow (X\otimes A) \otimes Y$ such that $lt=1$ and $rtl = rtr$. Explicitly, 
\be
t = 
\begin{tikzpicture}[very thick,scale=1.0,color=blue!50!black, baseline=.8cm]
\draw[line width=0pt] 
(1.75,1.75) node[line width=0pt, color=green!50!black] (A) {}
(0.86,1.75) node[line width=0pt] (X1) {}
(0.86,-0.5) node[line width=0pt] (X2) {}
(2,1.75) node[line width=0pt] (Y1) {}
(2,-0.5) node[line width=0pt] (Y2) {}; 

\draw[color=green!50!black] (1.25,0.55) .. controls +(0.0,0.25) and +(0.25,-0.15) .. (0.86,0.95);
\draw[-dot-, color=green!50!black] (1.25,0.55) .. controls +(0,-0.5) and +(0,-0.5) .. (1.75,0.55);

\draw[color=green!50!black] (1.5,-0.1) node[Odot] (unit) {}; 
\draw[color=green!50!black] (1.5,0.15) -- (unit);

\fill[color=green!50!black] (0.86,0.95) circle (2pt) node (meet) {};

\draw[color=green!50!black] (1.75,0.55) -- (A);
\draw[color=blue!50!black] (X1) -- (X2);
\draw[color=blue!50!black] (Y1) -- (Y2);

\end{tikzpicture}
\, . 
\ee
(If we only assume that~$A$ is an algebra we can only deduce that $l,r$ is \textsl{reflexive}, i.\,e.~there is a map $s: X \otimes Y \rightarrow (X\otimes A) \otimes Y$ with $ls=1=rs$.) 
Contractability implies that every coequaliser is split, i.\,e.~that in addition to $t$ there is $\xi: X \otimes_A Y \rightarrow X \otimes Y$ with $\vartheta \xi=1$, $\xi \vartheta = rt$ and $lt=1$.
\end{remark}

\section{Two-dimensional topological field theory with defects}\label{sec:2dTFTwithDefects}

In this section we briefly review the functorial approach to two-dimensional topological field theories in the presence of defects \cite{dkr1107.0495} and the application of defects to the construction of orbifold models. Once formulated in terms of defects, the orbifold construction immediately generalises beyond the group case \cite{ffrs0909.5013}.\footnote{In fact, the only place where we will meet groups is the example of matrix factorisations which are equivariant with respect to a group action, to be discussed in Section~\ref{subsec:equiMF}.} The present section is mainly meant to describe the conceptual origins of our constructions in later sections. The rest of this paper can be read independently of the material presented here, though knowing the original motivations and intuition is surely useful.

\subsection{TFTs with defects as symmetric monoidal functors}\label{subsec:TFTwdefects}

We assume  that the reader has some familiarity with the formulation of a closed 2d TFT as a symmetric monoidal functor from two-dimensional bordisms to vector spaces \cite{Dijkgraaf-thesis, Abrams2dTQFTaFA, Kockbook}. Enlarging the bordism category to include surfaces with unparametrised (``free'') boundaries leads to open/closed 2d TFTs \cite{l0010269, an0202164, lp0510664, ms0609042}. Here we discuss a different enlargement of the bordism category in terms of defects \cite{rs0808.1419, dkr1107.0495}. The description of bordisms with defects is a bit lengthy, but we will need these details to explain the orbifold construction in Section \ref{sec:TFT-orbifold}.

\medskip

A typical patch of worldsheet with phases and domain walls is shown in~\eqref{eq:worldsheetwithdefects}. To describe the bordism category for such worldsheets precisely, we first introduce two label sets, then the objects and morphisms of the bordism category, and finally two maps $s,t$ on the label sets that constrain the allowed assignments of labels to different components of a worldsheet.

\medskip
\noindent
{\sl Sets of defect conditions:} Fix two sets $D_2$ and $D_1$. We refer to elements of $D_2$ as {\sl phases}, and to those of $D_1$ as {\sl domain wall types} or {\sl defect conditions}.

\medskip
\noindent
{\sl Objects of the bordism category:} Objects are one-dimensional, oriented, compact manifolds without boundary and with extra decoration. Concretely, an object~$U$ has underlying manifold $\emptyset$ or $S^1 \times \{1,\dots,n\}$ for some $n \geqslant 1$, i.\,e.~an ordered disjoint union of unit circles in $\R^2$. On each $S^1 \setminus \{1\}$ there is a finite number of marked points, each labelled by a pair $(x,\varepsilon)$, where $x \in D_1$ and $\varepsilon \in \{ \pm 1\}$. The open intervals between two marked points are labelled by elements $a \in D_2$. We write $|U|=n$ for the number of copies of $S^1$ contained in $U$ and $U(k)$ for the $k$-th copy together with its decoration. An example for an object with $n=1$ is
\be\label{eq:objectinBord}
\begin{tikzpicture}[scale=0.55,color=blue!50!black, baseline]
\draw[very thick, color=black] (0,0) circle (2);
\fill (50:2) circle (3.0pt) node[right] {{\scriptsize$(x_1,+)$}};
\fill (-8:2) circle (3.0pt) node[right] {{\scriptsize$(x_2,+)$}};
\fill (-50:2) circle (3.0pt) node[right] {{\scriptsize$(x_3,-)$}};
\fill[color=black] (18:1.9) circle (0pt) node[right] {{\scriptsize$a_1$}};
\fill[color=black] (-30:1.9) circle (0pt) node[right] {{\scriptsize$a_2$}};
\fill[color=black] (-180:1.9) circle (0pt) node[left] {{\scriptsize$a_3$}};
\end{tikzpicture} 
.
\ee

\medskip
\noindent
{\sl Morphisms of the bordism category:} Let $U,V$ be objects as above. A morphism $M : U \to V$ is either a permutation or a bordism:
\begin{itemize}
\item Permutation: Suppose $|U|=|V|=n$. Then $M$ can be a permutation $\sigma \in S_n$ such that $V(\sigma(k)) = U(k)$ for all $k\in \{1,\dots,n \}$.
\item Bordism: $M$ can be (the equivalence class of) a two-dimensional, oriented, compact manifold together with a parametrisation of its boundary\footnote{To be precise, $\phi$ and $\psi$ are germs of smooth injections. On each component $S^1$ of $U$ (resp.\ of $V$), the ``incoming (resp.\ outgoing) parametrisation'' $\phi$ (resp.\ $\psi$) is defined on some open neighbourhood of $S^1$ intersected with $|z| \geqslant 1$ (resp.\ $|z| \leqslant 1$).} 
by maps $\phi : U \to M$ and $\psi : V \to M$ and a defect graph. The defect graph consists of a one-dimensional oriented submanifold $M_1$ build from non-intersecting defect lines (which are circles or closed intervals), each labelled by an element of $D_1$. 
$M_1$ must meet the boundary of~$M$ transversally, and the boundary points of~$M_1$ must be precisely the marked points on the boundary of~$M$; the $D_1$-label of a defect line ending on $(x,\varepsilon)$ must be~$x$. 
If $\varepsilon = 1$, the defect line is oriented away from the boundary for in-going boundary components, and towards the boundary for out-going boundary components; for $\varepsilon = -1$ the situation is reversed, see \cite[Fig.\,3]{dkr1107.0495}.
The complement $M_2 = M \setminus M_1$ consists of two-dimensional connected patches, each of which is labelled by an element from~$D_2$; 
if such a patch intersects the boundary, the $D_2$-labels have to match those of the corresponding intervals in~$U$ and~$V$ (mapped to $\partial M$ via $\phi$ and $\psi$, respectively). 
\end{itemize}

\medskip
\noindent
{\sl The maps} $s,t${\sl :} Since the bordisms and the defect lines are oriented, we can speak of a region immediately to the left and to the right of a segment of defect line. The maps $s,t : D_1 \to D_2$ (``source'' and ``target'') describe which worldsheet phase is allowed to the left and right of a given defect type; our orientation convention is 
\be
\begin{tikzpicture}[very thick,scale=0.3,color=blue!50!black, baseline]
\clip (0,0) ellipse (6cm and 3cm);
\nicedashedcolourscheme (0,0) ellipse (6cm and 3cm);

\draw (2.5,0) [black] node {{\scriptsize $s(x)$}};
\draw (-2.5,0) [black] node {{\scriptsize $t(x)$}};

\draw (0.6,0) node {{\scriptsize $x$}};

\draw[->, very thick] (0,-3) to (0,0);
\draw[very thick] (0,-0.2) to (0,3);

\end{tikzpicture}
\, .
\ee
The labelling of objects and morphisms has to be compatible with $s,t$. 

Since for objects the marked points are labelled by pairs $(x,\varepsilon)$, where $\varepsilon$ encodes the orientation of the intersecting defect line, it is convenient to define
\begin{align}
  s(x,+) = s(x) \, ,~~
  t(x,+) = t(x) \quad \text{and} \quad
  s(x,-) = t(x) \, ,~~
  t(x,-) = s(x) \, .
\end{align}
Using this, we call a sequence of defect types $(x_1,\varepsilon_1),\dots,(x_n,\varepsilon_n)$ {\sl composable} if $s(x_k,\varepsilon_k) = t(x_{k+1},\varepsilon_{k+1})$ and {\sl cyclically composable} if in addition $s(x_n,\varepsilon_n) = t(x_1, \varepsilon_1)$. 
The labelling of the marked points and intervals in an object is now such that the interval in clockwise direction of a marked point $(x,\varepsilon)$ is labelled by $s(x,\varepsilon)$ and that in counter-clockwise direction by $t(x,\varepsilon)$. For example, in~\eqref{eq:objectinBord} this means that $a_1 = s(x_1,+) = s(x_1)$, $a_2 = t(x_3,-) = s(x_3)$, etc.

As usual, composition is given by gluing or composition with the permutation, the tensor product by disjoint union and the symmetric structure by the permutation morphisms. 
This completes the description of the bordism category $\mathrm{Bord}_{2,1}^\mathrm{def}(D_2,D_1)$ of bordisms with defects.\footnote{In \cite{dkr1107.0495} the bordism category includes 0-dimensional defects called junctions and labelled by a set $D_0$. We recover the present setting from \cite{dkr1107.0495} by choosing $D_0$ to be the empty set.}

\medskip

We can now state that a {\sl two-dimensional oriented topological field theory with defects} is a symmetric monoidal functor
\begin{align}
  \tau : \mathrm{Bord}_{2,1}^\mathrm{def}(D_2,D_1) \lra \operatorname{Vect} 
\end{align}
which depends on objects and morphisms only up to isotopy; for objects, the isotopy is restricted not to move marked points across the point $-1$ of each unit circle.

In a maybe more familiar variant of the bordism category one would not include permutations into the sets of morphisms. A technical subtlety in the present definition is that the identity morphism on an object $U$ is the identity permutation and {\sl not} the cylinder over $U$. As a consequence, $\tau$ maps such a cylinder to an idempotent on the state space $\tau(U)$ and not necessarily to the identity. We say that $\tau$ is {\sl nondegenerate} if this idempotent is the identity for all $U$. The distinction between degenerate and nondegenerate TFTs (rather than just excluding the former case) is useful in the description of orbifolds.

\begin{remark}\label{rem:trivialandproperdefectTFTs}
\begin{enumerate}
\item
A closed 2d TFT is a special case of a 2d TFT with defects in which $D_1 = \emptyset$ (no domain walls) and $D_2 = \{ \bullet \}$ has just a single element. Similarly,
an open/closed 2d TFT is a special case of a 2d TFT with defects where this time $D_1$ is the set of boundary conditions, and $D_2 = \{ \bullet, \circ \}$, where the phase $\circ$ stands for the trivial theory. The map $s$ maps all of $D_1$ to $\circ$ (say) and $t$ maps $D_1$ to $\bullet$. In Section~\ref{subsec:ocTFT} we will discuss the example of open/closed TFTs from Landau-Ginzburg models, 
see in particular Remark~\ref{rem:LGocTFT}. 
\item ``Proper'' examples of 2d TFTs with defects, i.\,e.~examples not of the form in part (i), can be obtained via a lattice construction, see  \cite[Sect.\,3]{dkr1107.0495}. There, the worldsheet phases are described by certain Frobenius algebras and the domain walls by bimodules. It is clear that the lattice construction does not cover all defect TFTs since it even fails to do so for closed or open/closed TFTs \cite{bp9205031, fhk9212154, lp0602047}. Conjecturally, Landau-Ginzburg models and matrix factorisations (see Section~\ref{sec:LGmodels}) give rise to a defect TFT; this defect TFT does in general not have a lattice description.
\end{enumerate}
\end{remark}

\subsection{Bicategory of worldsheet phases}\label{subsec:worldsheetbicategory}

A 2d TFT with defects $\tau$ gives rise to a strict bicategory (i.\,e.~a 2-category) with adjoints \cite[Sect.\,2.4]{dkr1107.0495}. Its objects and 1-morphisms are build from the sets $D_2$, $D_1$ and from the maps $s,t : D_1 \to D_2$, while the functor $\tau$ defines the 2-morphism spaces, compositions, and the adjunction maps. In detail, this bicategory $\mathcal{D}_\tau$ is defined as follows.

\medskip\noindent
{\sl Objects:} The objects of $\mathcal{D}_\tau$ are simply given by the set of worldsheet phases $D_2$.

\medskip\noindent
{\sl 1-morphisms:} 
Let $a,b \in \mathcal{D}_\tau$. The set of 1-morphisms from $a$ to $b$ consists of formal sums of composable sequences of defect conditions,
\begin{align}
  \mathcal{D}_\tau(a,b) &= \big\langle \, X =\big(  (x_1,\varepsilon_1),\dots,(x_n,\varepsilon_n)\big) \text{ composable} \,  \big|\,  \nonumber \\
   & \hspace{10em}n \geqslant 0 \,,\, s(x_n,\varepsilon_n) = a \,,\, t(x_1,\varepsilon_1) = b \, \big\rangle_\oplus \, ,
\end{align}
that is, elements of $\mathcal{D}_\tau(a,b)$ are finite formal sums $X_1 \oplus \cdots \oplus X_l$, where each $X_i$ is a list as above (of possibly varying length $l$, but with fixed source $a$ and target $b$).
Evaluating $\tau$ for such sums is defined by taking direct sums of state spaces for objects and by adding the values of $\tau$ for morphisms. The operations below (composition, adjoints, assigment of 2-morphism spaces, \dots) distribute similarly over direct sums; we will not make this explicit and treat only the case of a single summand.

The identity 1-morphism $I_a$ is the empty sequence ($n=0$). Horizontal composition is concatenation of sequences and will be written as~$\otimes$.

\medskip\noindent
{\sl 2-morphisms:} 
As above we write
$X \equiv (X,+)= \big( (x_1,\varepsilon_1),\dots,(x_n,\varepsilon_n)\big)$ for composable sequences. Define the {\sl adjoint} $X^\dagger$ of such a sequence as
\begin{align}
  X^\dagger  \equiv (X,-) = \big( (x_n,-\varepsilon_n),\dots,(x_1,-\varepsilon_1)\big) \, .
\end{align}
Let $Z$ be cyclically composable. By $O(Z)$ we mean the object of $\mathrm{Bord}_{2,1}^\mathrm{def}(D_2,D_1)$ which consists of a single $S^1$ with $n$ marked points labelled $(z_1,\varepsilon_1),\dots,(z_n,\varepsilon_n)$ starting in clockwise direction after $-1 \in S^1$.

Given $X, Y \in \mathcal{D}_\tau(a,b)$, the space of 2-morphisms $X \to Y$ is given by
\be \label{eq:2-morph-via-tau}
\tau(O(Y \otimes X^\dagger)) = 
\tau \left(
\begin{tikzpicture}[scale=0.65,color=blue!50!black, baseline]
\draw[very thick, color=black] (0,0) circle (2);
\fill (165:2) circle (3.0pt) node[left] {{\scriptsize$(y_1, \nu_1)$}};
\fill (140:2) circle (3.0pt) node[left] {{\scriptsize$(y_2, \nu_2)$}};
\fill (110:2) circle (3.0pt) node[left] {};
\fill (90:2) circle (0pt) node[above] {{\scriptsize$\dots$}};
\fill (70:2) circle (3.0pt) node[left] {};
\fill (40:2) circle (3.0pt) node[right] {{\scriptsize$(y_{m-1}, \nu_{m-1})$}};
\fill (15:2) circle (3.0pt) node[right] {{\scriptsize$(y_m, \nu_m)$}};
\fill (-150:2) circle (3.0pt) node[left] {{\scriptsize$(x_1, -\varepsilon_1)$}};
\fill (-130:2) circle (3.0pt) node[left] {{\scriptsize$(x_2, -\varepsilon_2)$}};
\fill (-105:2) circle (3.0pt) node[left] {};
\fill (-90:2) circle (0pt) node[below] {{\scriptsize$\dots$}};
\fill (-75:2) circle (3.0pt) node[left] {};
\fill (-50:2) circle (3.0pt) node[right] {{\scriptsize$(x_{n-1}, -\varepsilon_{n-1})$}};
\fill (-30:2) circle (3.0pt) node[right] {{\scriptsize$(x_n, -\varepsilon_n)$}};
\end{tikzpicture} 
\right) ,
\ee
i.\,e.~the state space for an $S^1$ labelled in clockwise direction starting after $-1 \in S^1$ by $(y_1,\nu_1),\dots,(y_m,\nu_m), (x_n,-\varepsilon_n),\dots,(x_1,-\varepsilon_1)$. The identity 2-morphism, and the horizontal and vertical composition of 2-morphisms are obtained by applying $\tau$ to the (expected) special bordisms given in \cite[Fig.\,6]{dkr1107.0495}. Using invariance of $\tau$ under isotopy it is straightforward to verify the properties of a strict bicategory.

\medskip\noindent
{\sl Adjunctions:} The four adjunction maps $\ev_X$, $\tev_X$, $\coev_X$, $\tcoev_X$ described in Section \ref{subsec:bicatadj} (where here $\dX  = X^\dagger$) are given by evaluating $\tau$ on the bordisms in \cite[Fig.\,7]{dkr1107.0495}. The Zorro moves hold by isotopy invariance of $\tau$. It is equally immediate that $\mathcal{D}_\tau$ is pivotal, in fact strictly pivotal as $X^{\dagger\dagger} = X$ and we choose $\delta_{X} = 1_{X}$. The identities in \eqref{eq:pivotal} again amount to isotopy invariance.

\medskip

The above construction is summarised in the following theorem.

\begin{theorem}
A 2d TFT with defects $\tau : \mathrm{Bord}_{2,1}^\mathrm{def}(D_2,D_1) \rightarrow \operatorname{Vect}$ gives rise to a strictly pivotal 2-category $\mathcal{D}_\tau$ with objects and morphism categories as above.
\end{theorem}

\begin{remark}
An analogous theorem holds for non-topological two-dimensional field theories \cite[Sect.\,2.4]{dkr1107.0495}. In this case one has to restrict one's attention to topological defect types and the 2-morphism spaces are formed by families of translation and scale invariant states.
\end{remark}

Let us come back to the (patch of) worldsheet $\Sigma$ with phases and domain walls shown in~\eqref{eq:worldsheetwithdefects}. This worldsheet also involves points and defect junctions to be marked by fields. In the functorial formulation, a point marked by a field $\phi$ is described by cutting out a small disc around the point, resulting in an (incoming) boundary circle $O(X)$ for some sequence $X$. The field $\phi$ is an element of the state space $\tau(O(X))$ and is inserted in the corresponding tensor factor when evaluating $\tau$ on $\Sigma \setminus (\text{discs})$. In the orbifold construction we make plentiful use of such defect junctions labelled by elements of the corresponding state space.

\subsection{Orbifold TFTs}\label{sec:TFT-orbifold}

In the introduction we illustrated the procedure of blowing up bubbles filled with a phase~$a$ inside a worldsheet in phase~$b$. The result was a worldsheet filled with phase~$a$, together with a network of defect lines. Here we will mimic this procedure without a priori knowledge of the phase~$b$ and the domain wall separating~$b$ from~$a$. We will do so by describing a new closed 2d TFT $\tau^\mathrm{orb}_A : \mathrm{Bord}_{2,1} \rightarrow \operatorname{Vect}$ in terms of a tuple $(a,A,\mu,\Delta)$ where $a \in \mathcal{D}_\tau$, $A \in \mathcal{D}_\tau(a,a)$, and $\mu : A \otimes A \to A$ and $\Delta : A \to A \otimes A$, subject to certain conditions. This is the generalised orbifold construction of \cite{ffrs0909.5013}.

\medskip

By definition the two maps~$\mu$ and~$\Delta$ are elements in $\tau(O((A,+),(A,-),(A,-)))$ and $\tau(O((A,+),(A,+),(A,-)))$, respectively. They therefore label three-fold junctions of the defect~$A$ with two incoming and one outgoing line (for $\mu$) or one incoming and two outgoing lines (for $\Delta$): 
\be
\begin{tikzpicture}[very thick,scale=0.3,color=blue!50!black, baseline]
\clip (0,0) ellipse (6cm and 3cm);
\nicedashedcolourscheme (0,0) ellipse (6cm and 3cm);

\draw[
	decoration={markings, mark=at position 0.5 with {\arrow{>}}}, postaction={decorate},%
	color=green!50!black
	]
	 (-2.7,-2.7) -- (0,0);
\draw[
	decoration={markings, mark=at position 0.5 with {\arrow{>}}}, postaction={decorate},%
	color=green!50!black
	]
	 (2.7,-2.7) -- (0,0);
\draw[
	decoration={markings, mark=at position 0.5 with {\arrow{>}}}, postaction={decorate},%
	color=green!50!black
	]
	 (0,0) -- (0,3);

\fill[color=green!50!black] (0,0) circle (6pt) node[right] {{\scriptsize$\mu$}};

\fill[color=green!50!black] (0.7,1.25) circle (0pt) node {{\scriptsize$A$}};
\fill[color=green!50!black] (-2.3,-1.3) circle (0pt) node {{\scriptsize$A$}};
\fill[color=green!50!black] (2.3,-1.3) circle (0pt) node {{\scriptsize$A$}};

\draw (2.5,0.5) [black] node {{\scriptsize $a$}};
\draw (-2.5,0.5) [black] node {{\scriptsize $a$}};
\draw (0,-2.0) [black] node {{\scriptsize $a$}};

\end{tikzpicture}
\, , \quad
\begin{tikzpicture}[very thick,scale=0.3,color=blue!50!black, baseline]
\clip (0,0) ellipse (6cm and 3cm);
\nicedashedcolourscheme (0,0) ellipse (6cm and 3cm);

\draw[
	decoration={markings, mark=at position 0.5 with {\arrow{<}}}, postaction={decorate},%
	color=green!50!black
	]
	 (2.7,2.7) -- (0,0);
\draw[
	decoration={markings, mark=at position 0.5 with {\arrow{<}}}, postaction={decorate},%
	color=green!50!black
	]
	 (-2.7,2.7) -- (0,0);
\draw[
	decoration={markings, mark=at position 0.5 with {\arrow{<}}}, postaction={decorate},%
	color=green!50!black
	]
	 (0,0) -- (0,-3);

\fill[color=green!50!black] (0,0) circle (6pt) node[right] {{\scriptsize$\Delta$}};

\fill[color=green!50!black] (0.7,-1.25) circle (0pt) node {{\scriptsize$A$}};
\fill[color=green!50!black] (-2.3,1.3) circle (0pt) node {{\scriptsize$A$}};
\fill[color=green!50!black] (2.3,1.3) circle (0pt) node {{\scriptsize$A$}};

\draw (2.5,-0.5) [black] node {{\scriptsize $a$}};
\draw (-2.5,-0.5) [black] node {{\scriptsize $a$}};
\draw (0,2.0) [black] node {{\scriptsize $a$}};

\end{tikzpicture}
\, . 
\ee
We require~$\mu$ and~$\Delta$ to satisfy two sets of conditions of type ``bubble omission'' and ``crossing'': 
\be\label{eq:crossing-and-bubble}
\tau \Biggl(
\begin{tikzpicture}[very thick, scale=0.5,color=green!50!black, baseline]
\nicecolourscheme (0,0) circle (2);
\draw[very thick, color=blue!22] (0,0) circle (2);
\fill (90:2) circle (3.0pt) node[left] {};
\fill (-90:2) circle (3.0pt) node[left] {};
\draw (0,0.8) -- (90:2);
\draw[-dot-] (-1,0) .. controls +(0,1) and +(0,1) .. (1,0);
\draw[-dot-] (-1,0) .. controls +(0,-1) and +(0,-1) .. (1,0);
\draw (0,-0.8) -- (-90:2);
\end{tikzpicture}
\Biggl)
=
\tau \Biggl(
\begin{tikzpicture}[very thick, scale=0.5,color=green!50!black, baseline]
\nicecolourscheme (0,0) circle (2);
\draw[very thick, color=blue!22] (0,0) circle (2);
\fill (90:2) circle (3.0pt) node[left] {};
\fill (-90:2) circle (3.0pt) node[left] {};
\draw (90:2) -- (-90:2);
\end{tikzpicture}
\Biggl) ,
\quad 
\tau \Biggl(
\begin{tikzpicture}[very thick, scale=0.5,color=green!50!black, baseline]
\nicecolourscheme (0,0) circle (2);
\draw[very thick, color=blue!22] (0,0) circle (2);
\fill (70:2) circle (3.5pt) node[left] {};
\fill (110:2) circle (3.5pt) node[left] {};
\fill (-70:2) circle (3.5pt) node[left] {};
\fill (-110:2) circle (3.5pt) node[left] {};
\fill (0,1) circle (3.5pt) node[left] {};
\fill (0,-1) circle (3.5pt) node[left] {};
\draw (70:2) -- (0,1);
\draw (110:2) -- (0,1);
\draw (-70:2) -- (0,-1);
\draw (-110:2) -- (0,-1);
\draw (0,1) -- (0,-1);
\end{tikzpicture}
\Biggl)
=
\tau \Biggl( \!\!
\begin{tikzpicture}[very thick, scale=0.5,color=green!50!black, baseline]
\nicecolourscheme (0,0) circle (2);
\draw[very thick, color=blue!22] (0,0) circle (2);
\fill (150:2) circle (3.5pt) node[left] {};
\fill (-150:2) circle (3.5pt) node[left] {};
\fill (-30:2) circle (3.5pt) node[left] {};
\fill (30:2) circle (3.5pt) node[left] {};
\fill (-1,0) circle (3.5pt) node[left] {};
\fill (1,0) circle (3.5pt) node[left] {};
\draw (150:2) -- (-1,0);
\draw (-150:2) -- (-1,0);
\draw (-30:2) -- (1,0);
\draw (30:2) -- (1,0);
\draw (-1,0) -- (1,0);
\end{tikzpicture}
\Biggl) .
\ee
These identities are shorthand for the set of conditions obtained by putting arrows on the defect lines in all ways which allow the two junctions to be labelled by~$\mu$ or~$\Delta$. For example, this includes 
\begin{align}
& 
\tau \Biggl(
\begin{tikzpicture}[very thick, scale=0.5,color=green!50!black, baseline]
\nicecolourscheme (0,0) circle (2);
\draw[very thick, color=blue!22] (0,0) circle (2);
\fill (90:2) circle (3.0pt) node[left] {};
\fill (-90:2) circle (3.0pt) node[left] {};
\draw[
	decoration={markings, mark=at position 0.5 with {\arrow{>}}}, postaction={decorate},%
	color=green!50!black
	] 
	(0,0.8) -- (90:2);
\draw[-dot-] (-1,0) .. controls +(0,1) and +(0,1) .. (1,0);
\draw[-dot-] (-1,0) .. controls +(0,-1) and +(0,-1) .. (1,0);
\draw[
	decoration={markings, mark=at position 0.5 with {\arrow{>}}}, postaction={decorate},%
	color=green!50!black
	] 
	(-90:2) -- (0,-0.8);
\draw[
	decoration={markings, mark=at position 0.5 with {\arrow{>}}}, postaction={decorate},%
	color=green!50!black
	] 
	(-1,-0.01) -- (-1,0);
\draw[
	decoration={markings, mark=at position 0.5 with {\arrow{>}}}, postaction={decorate},%
	color=green!50!black
	] 
	(1,-0.01) -- (1,0);
\fill[color=green!50!black] (0,0.9) circle (0pt) node[below] {{\scriptsize$\mu$}};
\fill[color=green!50!black] (0,-0.9) circle (0pt) node[above] {{\scriptsize$\Delta$}};
\end{tikzpicture}
\Biggl)
= 
\tau \Biggl(
\begin{tikzpicture}[very thick, scale=0.5,color=green!50!black, baseline]
\nicecolourscheme (0,0) circle (2);
\draw[very thick, color=blue!22] (0,0) circle (2);
\fill (90:2) circle (3.0pt) node[left] {};
\fill (-90:2) circle (3.0pt) node[left] {};
\draw[
	decoration={markings, mark=at position 0.5 with {\arrow{>}}}, postaction={decorate},%
	color=green!50!black
	]
	(0,0.8) -- (90:2);
\draw[-dot-] (0,0.8) .. controls +(0,-0.5) and +(0,-0.5) .. (-0.75,0.8);
\draw[directedgreen, color=green!50!black] (-0.75,0.8) .. controls +(0,0.5) and +(0,0.5) .. (-1.5,0.8);
\draw[-dot-] (0,-0.8) .. controls +(0,0.5) and +(0,0.5) .. (-0.75,-0.8);
\draw[redirectedgreen, color=green!50!black] (-0.75,-0.8) .. controls +(0,-0.5) and +(0,-0.5) .. (-1.5,-0.8);
\draw[
	decoration={markings, mark=at position 0.5 with {\arrow{>}}}, postaction={decorate},%
	color=green!50!black
	] 
	(-90:2) -- (0,-0.8);
\draw (-1.5,0.8) -- (-1.5,-0.8);
\draw[
	decoration={markings, mark=at position 0.6 with {\arrow{>}}}, postaction={decorate},%
	color=green!50!black
	] 
	(-0.375,-0.4) -- (-0.375,0.4);
\fill[color=green!50!black] (-0.375,-0.35) circle (0pt) node[below] {{\scriptsize$\mu$}};
\fill[color=green!50!black] (-0.375,0.35) circle (0pt) node[above] {{\scriptsize$\Delta$}};
\end{tikzpicture}
\Biggl)
=
\tau \Biggl(
\begin{tikzpicture}[very thick, scale=0.5,color=green!50!black, baseline]
\nicecolourscheme (0,0) circle (2);
\draw[very thick, color=blue!22] (0,0) circle (2);
\fill (90:2) circle (3.0pt) node[left] {};
\fill (-90:2) circle (3.0pt) node[left] {};
\draw[
	decoration={markings, mark=at position 0.5 with {\arrow{>}}}, postaction={decorate},%
	color=green!50!black
	]
	(0,0.8) -- (90:2);
\draw[-dot-] (0,0.8) .. controls +(0,-0.5) and +(0,-0.5) .. (0.75,0.8);
\draw[directedgreen, color=green!50!black] (0.75,0.8) .. controls +(0,0.5) and +(0,0.5) .. (1.5,0.8);
\draw[-dot-] (0,-0.8) .. controls +(0,0.5) and +(0,0.5) .. (0.75,-0.8);
\draw[redirectedgreen, color=green!50!black] (0.75,-0.8) .. controls +(0,-0.5) and +(0,-0.5) .. (1.5,-0.8);
\draw[
	decoration={markings, mark=at position 0.5 with {\arrow{>}}}, postaction={decorate},%
	color=green!50!black
	]
	(-90:2) -- (0,-0.8);
\draw (1.5,0.8) -- (1.5,-0.8);
\draw[
	decoration={markings, mark=at position 0.6 with {\arrow{>}}}, postaction={decorate},%
	color=green!50!black
	]
	(0.375,-0.4) -- (0.375,0.4);
\fill[color=green!50!black] (0.375,-0.35) circle (0pt) node[below] {{\scriptsize$\mu$}};
\fill[color=green!50!black] (0.375,0.35) circle (0pt) node[above] {{\scriptsize$\Delta$}};
\end{tikzpicture}
\Biggl)
=
\tau \Biggl(
\begin{tikzpicture}[very thick, scale=0.5,color=green!50!black, baseline]
\nicecolourscheme (0,0) circle (2);
\draw[very thick, color=blue!22] (0,0) circle (2);
\fill (90:2) circle (3.0pt) node[left] {};
\fill (-90:2) circle (3.0pt) node[left] {};
\draw[
	decoration={markings, mark=at position 0.6 with {\arrow{>}}}, postaction={decorate},%
	color=green!50!black
	]
	(-90:2) -- (90:2);
\end{tikzpicture}
\Biggl) ,
 \label{eq:bubbleommision} \\
& 
\tau \Biggl(
\begin{tikzpicture}[very thick, scale=0.5,color=green!50!black, baseline]
\nicecolourscheme (0,0) circle (2);
\draw[very thick, color=blue!22] (0,0) circle (2);
\fill (70:2) circle (3.5pt) node[left] {};
\fill (110:2) circle (3.5pt) node[left] {};
\fill (-70:2) circle (3.5pt) node[left] {};
\fill (-110:2) circle (3.5pt) node[left] {};
\fill (0,1) circle (3.5pt) node[left] {{\scriptsize$\Delta$}};
\fill (0,-1) circle (3.5pt) node[left] {{\scriptsize$\mu$}};
\draw[
	decoration={markings, mark=at position 0.6 with {\arrow{>}}}, postaction={decorate},%
	color=green!50!black
	] 
	(0,1) -- (70:2);
\draw[
	decoration={markings, mark=at position 0.6 with {\arrow{>}}}, postaction={decorate},%
	color=green!50!black
	]
	(0,1) -- (110:2);
\draw[
	decoration={markings, mark=at position 0.6 with {\arrow{>}}}, postaction={decorate},%
	color=green!50!black
	]
	(-70:2) -- (0,-1);
\draw[
	decoration={markings, mark=at position 0.6 with {\arrow{>}}}, postaction={decorate},%
	color=green!50!black
	]
	(-110:2) -- (0,-1);
\draw[
	decoration={markings, mark=at position 0.6 with {\arrow{>}}}, postaction={decorate},%
	color=green!50!black
	]
	(0,-1) -- (0,1);
\end{tikzpicture}
\Biggl)
=
\tau \Biggl( \!\!
\begin{tikzpicture}[very thick, scale=0.5,color=green!50!black, baseline]
\nicecolourscheme (0,0) circle (2);
\draw[very thick, color=blue!22] (0,0) circle (2);
\fill (150:2) circle (3.5pt) node[left] {};
\fill (-150:2) circle (3.5pt) node[left] {};
\fill (-30:2) circle (3.5pt) node[left] {};
\fill (30:2) circle (3.5pt) node[left] {};
\fill (-1,0.2) circle (3.5pt) node[left] {};
\fill (1,-0.2) circle (3.5pt) node[left] {};
\fill (-0.85,0.6) circle (0pt) node {{\scriptsize$\mu$}};
\fill (0.9,-0.6) circle (0pt) node {{\scriptsize$\Delta$}};
\draw[
	decoration={markings, mark=at position 0.6 with {\arrow{>}}}, postaction={decorate},%
	color=green!50!black
	]
	(-1,0.2) -- (150:2);
\draw[
	decoration={markings, mark=at position 0.6 with {\arrow{>}}}, postaction={decorate},%
	color=green!50!black
	]
	(-150:2) -- (-1,0.2);
\draw[
	decoration={markings, mark=at position 0.6 with {\arrow{>}}}, postaction={decorate},%
	color=green!50!black
	]
	(-30:2) -- (1,-0.2);
\draw[
	decoration={markings, mark=at position 0.6 with {\arrow{>}}}, postaction={decorate},%
	color=green!50!black
	]
	(1,-0.2) -- (30:2);
\draw[
	decoration={markings, mark=at position 0.6 with {\arrow{>}}}, postaction={decorate},%
	color=green!50!black
	]
	(1,-0.2) -- (-1,0.2);
\end{tikzpicture}
\Biggl)
=
\tau \Biggl( \!\!
\begin{tikzpicture}[very thick, scale=0.5,color=green!50!black, baseline]
\nicecolourscheme (0,0) circle (2);
\draw[very thick, color=blue!22] (0,0) circle (2);
\fill (150:2) circle (3.5pt) node[left] {};
\fill (-150:2) circle (3.5pt) node[left] {};
\fill (-30:2) circle (3.5pt) node[left] {};
\fill (30:2) circle (3.5pt) node[left] {};
\fill (-1,-0.2) circle (3.5pt) node[left] {};
\fill (1,0.2) circle (3.5pt) node[left] {};
\fill (-0.95,-0.6) circle (0pt) node {{\scriptsize$\Delta$}};
\fill (0.9,0.55) circle (0pt) node {{\scriptsize$\mu$}};
\draw[
	decoration={markings, mark=at position 0.6 with {\arrow{>}}}, postaction={decorate},%
	color=green!50!black
	]
	(-1,-0.2) -- (150:2);
\draw[
	decoration={markings, mark=at position 0.6 with {\arrow{>}}}, postaction={decorate},%
	color=green!50!black
	]
	(-150:2) -- (-1,-0.2);
\draw[
	decoration={markings, mark=at position 0.6 with {\arrow{>}}}, postaction={decorate},%
	color=green!50!black
	]
	(-30:2) -- (1,0.2);
\draw[
	decoration={markings, mark=at position 0.6 with {\arrow{>}}}, postaction={decorate},%
	color=green!50!black
	]
	(1,0.2) -- (30:2);
\draw[
	decoration={markings, mark=at position 0.6 with {\arrow{>}}}, postaction={decorate},%
	color=green!50!black
	]
	(-1,-0.2) -- (1,0.2);
\end{tikzpicture}
\Biggl) 
\, . 
 \label{eq:crossing}
\end{align}

Define the two morphisms $\eta : I_a \to A$ and $\varepsilon : A \to I_a$ as
\begin{align}
  \eta & = 
  \tau \Biggl(
\begin{tikzpicture}[very thick, scale=0.5,color=green!50!black, baseline]
\nicecolourscheme (0,0) circle (2);
\draw[very thick, color=blue!22] (0,0) circle (2);
\fill (90:2) circle (3.0pt) node[left] {};
\draw (0,0.8) -- (90:2);
\draw[-dot-] (0,0.8) .. controls +(0,-0.5) and +(0,-0.5) .. (-0.75,0.8);
\draw[directedgreen, color=green!50!black] (-0.75,0.8) .. controls +(0,0.5) and +(0,0.5) .. (-1.5,0.8);
\draw[redirectedgreen, color=green!50!black] (-0.375,-0.4) .. controls +(0,-0.75) and +(0,-0.75) .. (-1.5,-0.4);
\draw (-1.5,-0.4) -- (-1.5,0.8);
\draw (-0.375,-0.4) -- (-0.375,0.4);
\end{tikzpicture}
\Biggl)
=
\tau \Biggl(
\begin{tikzpicture}[very thick, scale=0.5,color=green!50!black, baseline]
\nicecolourscheme (0,0) circle (2);
\draw[very thick, color=blue!22] (0,0) circle (2);
\fill (90:2) circle (3.0pt) node[left] {};
\draw (0,0.8) -- (90:2);
\draw[-dot-] (0,0.8) .. controls +(0,-0.5) and +(0,-0.5) .. (0.75,0.8);
\draw[directedgreen, color=green!50!black] (0.75,0.8) .. controls +(0,0.5) and +(0,0.5) .. (1.5,0.8);
\draw[redirectedgreen, color=green!50!black] (0.375,-0.4) .. controls +(0,-0.75) and +(0,-0.75) .. (1.5,-0.4);
\draw (1.5,-0.4) -- (1.5,0.8);
\draw (0.375,-0.4) -- (0.375,0.4);
\end{tikzpicture}
\Biggl)
\label{eq:eta-from-mu-Delta}
   \, , \\
  \varepsilon & = 
  \tau \Biggl(
\begin{tikzpicture}[very thick, scale=0.5,color=green!50!black, baseline, rotate=180]
\nicecolourscheme (0,0) circle (2);
\draw[very thick, color=blue!22] (0,0) circle (2);
\fill (90:2) circle (3.0pt) node[left] {};
\draw (0,0.8) -- (90:2);
\draw[-dot-] (0,0.8) .. controls +(0,-0.5) and +(0,-0.5) .. (0.75,0.8);
\draw[redirectedgreen, color=green!50!black] (0.75,0.8) .. controls +(0,0.5) and +(0,0.5) .. (1.5,0.8);
\draw[directedgreen, color=green!50!black] (0.375,-0.4) .. controls +(0,-0.75) and +(0,-0.75) .. (1.5,-0.4);
\draw (1.5,-0.4) -- (1.5,0.8);
\draw (0.375,-0.4) -- (0.375,0.4);
\end{tikzpicture}
\Biggl)
=
\tau \Biggl(
\begin{tikzpicture}[very thick, scale=0.5,color=green!50!black, baseline, rotate=180]
\nicecolourscheme (0,0) circle (2);
\draw[very thick, color=blue!22] (0,0) circle (2);
\fill (90:2) circle (3.0pt) node[left] {};
\draw (0,0.8) -- (90:2);
\draw[-dot-] (0,0.8) .. controls +(0,-0.5) and +(0,-0.5) .. (-0.75,0.8);
\draw[redirectedgreen, color=green!50!black] (-0.75,0.8) .. controls +(0,0.5) and +(0,0.5) .. (-1.5,0.8);
\draw[directedgreen, color=green!50!black] (-0.375,-0.4) .. controls +(0,-0.75) and +(0,-0.75) .. (-1.5,-0.4);
\draw (-1.5,-0.4) -- (-1.5,0.8);
\draw (-0.375,-0.4) -- (-0.375,0.4);
\end{tikzpicture}
\Biggl)
   \, . 
  \label{eq:eps-from-mu-Delta}
\end{align}
That $\tau$ gives the same answer for both defining bordisms is due to the bubble omission property. In more detail, in the case of~$\eta$ one first verifies that each choice is a one-sided unit for $\mu$; e.\,g.~for the first bordism given for~$\eta$ one computes
\be
  \tau \Biggl(
\begin{tikzpicture}[very thick, scale=0.5,color=green!50!black, baseline]
\nicecolourscheme (0,0) circle (2);
\draw[very thick, color=blue!22] (0,0) circle (2);
\fill (60:2) circle (3.0pt) node[left] {};
\fill (-60:2) circle (3.0pt) node[left] {};
\draw(60:2) -- (-60:2);
\fill (1.0,1.2) circle (3.0pt) node[left] {};
\draw (0,0.8) .. controls +(0,0.3) and +(0,-0.1) .. (1.0,1.2);
\draw[-dot-] (0,0.8) .. controls +(0,-0.5) and +(0,-0.5) .. (-0.75,0.8);
\draw[directedgreen, color=green!50!black] (-0.75,0.8) .. controls +(0,0.5) and +(0,0.5) .. (-1.5,0.8);
\draw[redirectedgreen, color=green!50!black] (-0.375,-0.4) .. controls +(0,-0.75) and +(0,-0.75) .. (-1.5,-0.4);
\draw (-1.5,-0.4) -- (-1.5,0.8);
\draw (-0.375,-0.4) -- (-0.375,0.4);
\end{tikzpicture}
\Biggl)
\stackrel{\eqref{eq:crossing}}{=}
\tau \Biggl(
\begin{tikzpicture}[very thick, scale=0.5,color=green!50!black, baseline]
\nicecolourscheme (0,0) circle (2);
\draw[very thick, color=blue!22] (0,0) circle (2);
\fill (90:2) circle (3.0pt) node[left] {};
\fill (-90:2) circle (3.0pt) node[left] {};
\draw (0,0.8) -- (90:2);
\draw[-dot-] (0,0.8) .. controls +(0,-0.5) and +(0,-0.5) .. (-0.75,0.8);
\draw[directedgreen, color=green!50!black] (-0.75,0.8) .. controls +(0,0.5) and +(0,0.5) .. (-1.5,0.8);
\draw[-dot-] (0,-0.8) .. controls +(0,0.5) and +(0,0.5) .. (-0.75,-0.8);
\draw[redirectedgreen, color=green!50!black] (-0.75,-0.8) .. controls +(0,-0.5) and +(0,-0.5) .. (-1.5,-0.8);
\draw (0,-0.8) -- (-90:2);
\draw (-1.5,0.8) -- (-1.5,-0.8);
\draw (-0.375,0.4) -- (-0.375,-0.4);
\end{tikzpicture}
\Biggl)
\stackrel{\eqref{eq:bubbleommision}}{=}
\tau \Biggl(
\begin{tikzpicture}[very thick, scale=0.5,color=green!50!black, baseline]
\nicecolourscheme (0,0) circle (2);
\draw[very thick, color=blue!22] (0,0) circle (2);
\fill (90:2) circle (3.0pt) node[left] {};
\fill (-90:2) circle (3.0pt) node[left] {};
\draw (90:2) -- (-90:2);
\end{tikzpicture}
\Biggl) .
\ee
Analogously, the second bordism for $\eta$ is verified to be a right unit. The two one-sided units then have to be identical (and in particular $\eta$ is a two-sided unit) since
\be
\tau \Biggl(
\begin{tikzpicture}[very thick, scale=0.5,color=green!50!black, baseline]
\nicecolourscheme (0,0) circle (2);
\draw[very thick, color=blue!22] (0,0) circle (2);
\fill (90:2) circle (3.0pt) node[left] {};
\draw (0,0.8) -- (90:2);
\draw[-dot-] (0,0.8) .. controls +(0,-0.5) and +(0,-0.5) .. (-0.75,0.8);
\draw[directedgreen, color=green!50!black] (-0.75,0.8) .. controls +(0,0.5) and +(0,0.5) .. (-1.5,0.8);
\draw[redirectedgreen, color=green!50!black] (-0.375,-0.4) .. controls +(0,-0.75) and +(0,-0.75) .. (-1.5,-0.4);
\draw (-1.5,-0.4) -- (-1.5,0.8);
\draw (-0.375,-0.4) -- (-0.375,0.4);
\end{tikzpicture}
\Biggl)
=
\tau \Biggl(
\begin{tikzpicture}[very thick, scale=0.5,color=green!50!black, baseline]
\nicecolourscheme (0,0) circle (2);
\draw[very thick, color=blue!22] (0,0) circle (2);
\fill (90:2) circle (3.0pt) node[left] {};
\draw (0,0.8) -- (90:2);
%
\draw[-dot-] (-0.25,0.5) .. controls +(0,0.5) and +(0,0.5) .. (0.25,0.5);
\draw[-dot-] (-0.25,0.5) .. controls +(0,-0.5) and +(0,-0.5) .. (-1,0.5);
\draw[directedgreen, color=green!50!black] (-1,0.5) .. controls +(0,0.5) and +(0,0.5) .. (-1.75,0.5);
\draw[redirectedgreen, color=green!50!black] (-0.625,-0.4) .. controls +(0,-0.75) and +(0,-0.75) .. (-1.75,-0.4);
\draw (-1.75,-0.4) -- (-1.75,0.5);
\draw (-0.625,-0.4) -- (-0.625,0.1);
%
\draw[-dot-] (0.25,0.5) .. controls +(0,-0.5) and +(0,-0.5) .. (1,0.5);
\draw[directedgreen, color=green!50!black] (1,0.5) .. controls +(0,0.5) and +(0,0.5) .. (1.75,0.5);
\draw[redirectedgreen, color=green!50!black] (0.625,-0.4) .. controls +(0,-0.75) and +(0,-0.75) .. (1.75,-0.4);
\draw (1.75,-0.4) -- (1.75,0.5);
\draw (0.625,-0.4) -- (0.625,0.1);
\end{tikzpicture}
\Biggl)
=
\tau \Biggl(
\begin{tikzpicture}[very thick, scale=0.5,color=green!50!black, baseline]
\nicecolourscheme (0,0) circle (2);
\draw[very thick, color=blue!22] (0,0) circle (2);
\fill (90:2) circle (3.0pt) node[left] {};
\draw (0,0.8) -- (90:2);
\draw[-dot-] (0,0.8) .. controls +(0,-0.5) and +(0,-0.5) .. (0.75,0.8);
\draw[directedgreen, color=green!50!black] (0.75,0.8) .. controls +(0,0.5) and +(0,0.5) .. (1.5,0.8);
\draw[redirectedgreen, color=green!50!black] (0.375,-0.4) .. controls +(0,-0.75) and +(0,-0.75) .. (1.5,-0.4);
\draw (1.5,-0.4) -- (1.5,0.8);
\draw (0.375,-0.4) -- (0.375,0.4);
\end{tikzpicture}
\Biggl)
\, .
\ee
Along the same lines one checks that the two bordisms defining~$\varepsilon$ give the same 2-morphism and that $\varepsilon$ is a two-sided counit for $\Delta$. The precise relation between the properties in \eqref{eq:crossing-and-bubble} and Frobenius algebras is stated in the next proposition.

\begin{proposition}\label{prop:orb-defect}
Let $a \in \mathcal{D}_\tau$, $A \in \mathcal{D}_\tau(a,a)$ and $\mu : A \otimes A \to A$, $\Delta : A \to A \otimes A$ be given. The following are equivalent:
\begin{enumerate}
\item $A$, $\mu$, $\Delta$ together with $\eta,\varepsilon$ as in \eqref{eq:eta-from-mu-Delta}, \eqref{eq:eps-from-mu-Delta} form a separable symmetric Frobenius algebra (see Definition~\ref{def:algebra-properties}).
\item $A$, $\mu$, $\Delta$ satisfy the conditions in \eqref{eq:crossing-and-bubble}.
\end{enumerate}
\end{proposition}

\begin{proof}
(i)$\Rightarrow$(ii): The crossing conditions~\eqref{eq:crossing} are satisfied by (co)associativity and by the Frobenius property. The bubble omission with all arrows oriented to the top of the disc diagram amounts to separability. Bubble omission with arrows in the loop oriented clockwise follows from symmetry and separability:
\be
\begin{tikzpicture}[very thick, scale=0.5,color=green!50!black, baseline]
\draw (0,0.8) -- (90:2);
\draw[-dot-] (0,0.8) .. controls +(0,-0.5) and +(0,-0.5) .. (0.75,0.8);
\draw[directedgreen, color=green!50!black] (0.75,0.8) .. controls +(0,0.5) and +(0,0.5) .. (1.5,0.8);
\draw[-dot-] (0,-0.8) .. controls +(0,0.5) and +(0,0.5) .. (0.75,-0.8);
\draw[redirectedgreen, color=green!50!black] (0.75,-0.8) .. controls +(0,-0.5) and +(0,-0.5) .. (1.5,-0.8);
\draw (0,-0.8) -- (-90:2);
\draw (1.5,0.8) -- (1.5,-0.8);
\draw (0.375,0.4) -- (0.375,-0.4);
\end{tikzpicture}
\stackrel{\eqref{eq:Frob}}{=}
\begin{tikzpicture}[very thick, scale=0.5,color=green!50!black, baseline]
\draw (0,0.8) -- (90:2);
\draw[-dot-] (0,0.8) .. controls +(0,-0.5) and +(0,-0.5) .. (1,0.8);
\draw[directedgreen, color=green!50!black] (1,0.8) .. controls +(0,0.5) and +(0,0.5) .. (2,0.8);
\draw[-dot-] (0.5,-0.8) .. controls +(0,-0.5) and +(0,-0.5) .. (1,-0.8);
\draw[-dot-] (1.5,-0.8) .. controls +(0,0.5) and +(0,0.5) .. (1,-0.8);
\draw[,color=green!50!black] (1.5,-0.8) .. controls +(0,-0.5) and +(0,-0.5) .. (2,-0.8);
\draw[very thick, <-] (1.679,-1.17) -- (1.681,-1.17) node[above] {}; 
\draw (1.25,0.1) node[Odot] (D) {}; 
\draw (D) -- (1.25,-0.3); 
\draw (2,-0.8) -- (2,0.8);
\draw (0.5,-0.8) -- (0.5,0.5);
\draw (0.75,-1.1) -- (0.75,-2);
\end{tikzpicture}
\stackrel{\eqref{eq:Asymmetric}}{=}
\begin{tikzpicture}[very thick, scale=0.5,color=green!50!black, baseline]
\draw (0,0.8) -- (90:2);
\draw[-dot-] (0,0.8) .. controls +(0,-0.5) and +(0,-0.5) .. (0.75,0.8);
\draw[directedgreen, color=green!50!black] (0.75,0.8) .. controls +(0,0.5) and +(0,0.5) .. (1.5,0.8);
\draw[directedgreen, color=green!50!black] (1.5,0.8) .. controls +(0,-0.5) and +(0,-0.5) .. (2.25,0.8);
\draw[-dot-] (2.25,0.8) .. controls +(0,0.5) and +(0,0.5) .. (3,0.8);
\draw (2.625,1.7) node[Odot] (D) {}; 
\draw (D) -- (2.625,1.2); 
\draw (0.375,0.4) -- (0.375,-0.4);
\draw (3,0.8) -- (3,-0.4);
\draw[-dot-] (3,-0.4) .. controls +(0,-1.25) and +(0,-1.25) .. (0.375,-0.4);
\draw (1.6875,-1.3) -- (1.6875,-2);
\end{tikzpicture}
\stackrel{\eqref{eq:otherZorro}}{=}
\begin{tikzpicture}[very thick, scale=0.5,color=green!50!black, baseline]
\draw (0,0.8) -- (90:2);
\draw[-dot-] (0,0.8) .. controls +(0,-0.5) and +(0,-0.5) .. (0.75,0.8);
\draw[-dot-, color=green!50!black] (0.75,0.8) .. controls +(0,0.5) and +(0,0.5) .. (1.5,0.8);
\draw (1.125,1.7) node[Odot] (D) {}; 
\draw (D) -- (1.125,1.2); 
\draw (0.375,0.4) -- (0.375,-0.7);
\draw (1.5,0.8) -- (1.5,-0.7);
\draw[-dot-] (1.5,-0.7) .. controls +(0,-0.75) and +(0,-0.75) .. (0.375,-0.7);
\draw (0.9375,-1.2) -- (0.9375,-2);
\end{tikzpicture}
\stackrel{\eqref{eq:Frob}}{=}
\begin{tikzpicture}[very thick, scale=0.5,color=green!50!black, baseline]
\draw (0,1.1) -- (90:2);
\draw[-dot-] (-0.375,0.8) .. controls +(0,0.5) and +(0,0.5) .. (0.375,0.8);
\draw[-dot-] (1.125,0.8) .. controls +(0,-0.5) and +(0,-0.5) .. (0.375,0.8);
\draw (1.125,1.2) node[Odot] (D) {}; 
\draw (D) -- (1.125,0.8); 
\draw (-0.375,0.8) -- (-0.375,-0.7);
\draw (0.75,0.5) -- (0.75,-0.7);
\draw[-dot-] (0.75,-0.7) .. controls +(0,-0.75) and +(0,-0.75) .. (-0.375,-0.7);
\draw (0.1875,-1.2) -- (0.1875,-2);
\end{tikzpicture}
\stackrel{\eqref{eq:leftcomp}}{=}
\begin{tikzpicture}[very thick, scale=0.5,color=green!50!black, baseline]
\draw (0,0.8) -- (90:2);
\draw[-dot-] (-1,0) .. controls +(0,1) and +(0,1) .. (1,0);
\draw[-dot-] (-1,0) .. controls +(0,-1) and +(0,-1) .. (1,0);
\draw (0,-0.8) -- (-90:2);
\end{tikzpicture}
\stackrel{\eqref{eq:special}}{=}
\begin{tikzpicture}[very thick, scale=0.5,color=green!50!black, baseline]
\draw (90:2) -- (-90:2);
\end{tikzpicture}
\ee
where all equalities hold after application of~$\tau$. 
The argument for anti-clockwise oriented arrows is analogous.

(ii)$\Rightarrow$(i): The (co)unit property of $\eta$ and $\varepsilon$ was checked above. (Co)associativity, the Frobenius property and separability are immediate from crossing and bubble omission. For symmetry one computes
\be
\begin{tikzpicture}[very thick, scale=0.5,color=green!50!black, baseline]
\draw[-dot-] (0,-0.5) .. controls +(0,0.75) and +(0,0.75) .. (1,-0.5);
\draw[redirectedgreen] (1,-0.5) .. controls +(0,-0.75) and +(0,-0.75) .. (2,-0.5);
\draw (0.5,0.6) node[Odot] (D) {}; 
\draw (D) -- (0.5,0.1); 
\draw (0,-0.5) -- (0,-2);
\draw (2,-0.5) -- (2,2);
\end{tikzpicture}
\stackrel{\eqref{eq:eps-from-mu-Delta}}{=}
\begin{tikzpicture}[very thick, scale=0.5,color=green!50!black, baseline]
\draw[-dot-] (0,-0.5) .. controls +(0,0.75) and +(0,0.75) .. (1,-0.5);
\draw (0.5,0.6) -- (0.5,0.1); 
\draw[-dot-] (0.5,0.6) .. controls +(0,0.75) and +(0,0.75) .. (1.5,0.6);
\draw[redirectedgreen] (1.5,0.6) .. controls +(0,-0.75) and +(0,-0.75) .. (2.5,0.6);
\draw (1,1.3) -- (1,1.4);
\draw (2.5,0.6) -- (2.5,1.4);
\draw[directedgreen] (1,1.4) .. controls +(0,0.5) and +(0,0.5) .. (2.5,1.4);
\draw[redirectedgreen] (1,-0.5) .. controls +(0,-0.75) and +(0,-0.75) .. (3,-0.5);
\draw (0,-0.5) -- (0,-2);
\draw (3,-0.5) -- (3,2);
\end{tikzpicture}
\stackrel{\eqref{eq:algcond}}{=}
\begin{tikzpicture}[very thick, scale=0.5,color=green!50!black, baseline]
\draw[-dot-] (0.5,0.6) .. controls +(0,0.75) and +(0,0.75) .. (1.5,0.6);
\draw[-dot-] (1,-0.5) .. controls +(0,0.75) and +(0,0.75) .. (2,-0.5);
\draw (1.5,0.6) -- (1.5,0.1);
\draw (1.0,1.1) -- (1.0,1.4);
\draw (3,-0.5) -- (3.0,1.4);
\draw (0.5,0.6) -- (0.5,-2);
\draw (1,-0.5) -- (1,-1.0);
\draw[redirectedgreen] (2,-0.5) .. controls +(0,-0.75) and +(0,-0.75) .. (3,-0.5);
\draw[directedgreen] (1.0,1.4) .. controls +(0,0.5) and +(0,0.5) .. (3.0,1.4);
\draw[redirectedgreen] (1,-1.0) .. controls +(0,-0.75) and +(0,-0.75) .. (3.5,-1.0);
\draw (3.5,-1.0) -- (3.5,2);
\end{tikzpicture}
=
\begin{tikzpicture}[very thick, scale=0.5,color=green!50!black, baseline]
\draw[-dot-] (0.5,0.6) .. controls +(0,0.75) and +(0,0.75) .. (1.5,0.6);
\draw[-dot-] (1,-0.5) .. controls +(0,0.75) and +(0,0.75) .. (2,-0.5);
\draw (1.5,0.6) -- (1.5,0.1);
\draw (1.0,1.1) -- (1.0,1.4);
\draw (3,-0.5) -- (3.0,1.4);
\draw (1,-0.5) -- (1,-1.0);
\draw[redirectedgreen] (2,-0.5) .. controls +(0,-0.75) and +(0,-0.75) .. (3,-0.5);
\draw[directedgreen] (1.0,1.4) .. controls +(0,0.5) and +(0,0.5) .. (3.0,1.4);
\draw[redirectedgreen] (0.5,0.6) .. controls +(0,-0.75) and +(0,-0.75) .. (-0.5,0.6);
\draw (1,-1.0) -- (1,-2.0);
\draw (-0.5,0.6) -- (-0.5,2);
\end{tikzpicture}
\stackrel{\eqref{eq:algcond}}{=}
\begin{tikzpicture}[very thick, scale=0.5,color=green!50!black, baseline]
\draw[-dot-] (0.5,0.6) .. controls +(0,0.75) and +(0,0.75) .. (1.5,0.6);
\draw[-dot-] (1,-0.5) .. controls +(0,0.75) and +(0,0.75) .. (0,-0.5);
\draw (0.5,0.6) -- (0.5,0.1);
\draw (1,1.1) -- (1,1.4);
\draw (2.5,1.4) -- (2.5,0.6);
\draw[redirectedgreen] (1.5,0.6) .. controls +(0,-0.75) and +(0,-0.75) .. (2.5,0.6);
\draw[directedgreen] (1,1.4) .. controls +(0,0.5) and +(0,0.5) .. (2.5,1.4);
\draw[redirectedgreen] (0,-0.5) .. controls +(0,-0.75) and +(0,-0.75) .. (-1,-0.5);

\draw (-1,2) -- (-1,-0.5);
\draw (1,-0.5) -- (1,-2);
\end{tikzpicture}
\stackrel{\eqref{eq:eps-from-mu-Delta}}{=}
\begin{tikzpicture}[very thick, scale=0.5,color=green!50!black, baseline]
\draw[-dot-] (1,-0.5) .. controls +(0,0.75) and +(0,0.75) .. (0,-0.5);
\draw[redirectedgreen] (0,-0.5) .. controls +(0,-0.75) and +(0,-0.75) .. (-1,-0.5);
\draw (0.5,0.6) node[Odot] (D) {}; 
\draw (D) -- (0.5,0.1); 
\draw (-1,2) -- (-1,-0.5);
\draw (1,-0.5) -- (1,-2);
\end{tikzpicture}
\ee
where again application of~$\tau$ is implicit, and in the unmarked step we used isotopy invariance. 
\end{proof}

We will refer to a tuple $(a,A,\mu,\Delta)$ satisfying either condition in Proposition~\ref{prop:orb-defect} as an \textsl{orbifolding defect}, and we will abbreviate such a tuple by~$A$. 
Given an orbifolding defect $A \in \mathcal{D}_\tau(a,a)$, we can construct a nondegenerate closed 2d TFT without defects, i.\,e.~a functor
\begin{align}
  \tau^\mathrm{orb}_A : \mathrm{Bord}_{2,1} \longrightarrow \operatorname{Vect} \, ,
\end{align}
in two steps.
First, we define a possibly degenerate closed 2d TFT $\hat\tau^\mathrm{orb}_A$. For a general object in $\mathrm{Bord}_{2,1}$ we set
\begin{align}
  \hat\tau^\mathrm{orb}_A(S^1 \times \{1,\dots,n\}) = \tau(O(A,+)\times \{1,\dots,n\}) \, .
\end{align}
In words, the state space of the (possibly degenerate) orbifolded theory on a disjoint union of circles is given by evaluating the unorbifolded theory on the same set of circles, but with a single marked point $(A,+)$ placed on each circle, say at the point~$1$.
For a morphism $M : U \to V$ in $\mathrm{Bord}_{2,1}$ we define
\begin{align}
  \hat\tau^\mathrm{orb}_A(M) = \tau(M^{\text{$A$-network}}) \, .
  \label{eq:tau-orb-on-morph}
\end{align}
Here $M^{\text{$A$-network}}$ is $M$ together with a network of defect lines, all labelled by~$A$. The network is such that it only has three-valent junctions, and each junction has precisely one or two incoming lines so that they can be labelled by either~$\Delta$ or~$\mu$. Each boundary circle in the image of $U$ under the parametrisation map (i.\,e.~an incoming boundary circle) is the starting point of precisely one $A$-defect line, and each outgoing boundary circle has precisely one $A$-line ending on it. 
The network also has to be fine enough in the sense that the complement of the network in~$M$ consists of connected components homeomorphic to discs (if this is the case, any further refinements can be removed using $\Delta$-separability).

It is not too hard to convince oneself that the defining properties of an orbifolding defect~$A$ given in Proposition~\ref{prop:orb-defect} guarantee that $\tau(M^{\text{$A$-network}})$ is independent of the choice of defect network (thanks to the condition that it is fine enough), so that the assignment \eqref{eq:tau-orb-on-morph} is well-defined and that the following result holds true:

\begin{proposition}
Let $A \in \mathcal{D}_\tau(a,a)$ be an orbifolding defect. Then $\hat\tau^\mathrm{orb}_A : \mathrm{Bord}_{2,1} \rightarrow \operatorname{Vect}$ is a (possibly degenerate) closed 2d TFT.
\end{proposition}

This completes the first step in the construction of the nondegenerate orbifold TFT. The second step consists of making $\hat\tau^\mathrm{orb}_A$ nondegenerate, for which there is a simple general procedure. Namely, for each object $U \in \mathrm{Bord}_{2,1}$, the cylinder over $U$ gets mapped to an idempotent
\begin{align}
   P_U = \hat\tau^\mathrm{orb}_A(\text{cylinder over $U$}) \, .
\end{align}
Let $e_U : \mathrm{im}(P_U) \to  \hat\tau^\mathrm{orb}_A(U)$ and $r_U : \hat\tau^\mathrm{orb}_A(U) \to \mathrm{im}(P_U)$ be the embedding of and the restriction to the image, respectively. It is straightforward to check that
\begin{align}
  \tau^\mathrm{orb}_A(U) = \mathrm{im}(P_U) \, , \quad
  \tau^\mathrm{orb}_A(U \xrightarrow{M} V) = r_V \circ \hat\tau^\mathrm{orb}_A(U \xrightarrow{M} V) \circ e_U
\end{align}
defines a nondegenerate closed 2d TFT.

\begin{remark}
If we think of orbifolding as gauging a discrete symmetry, the above procedure has a natural interpretation. The orbifolding defect~$A$ describes the ``gauge symmetry'', which however no longer has to be given by a group. The state space $\hat\tau^\mathrm{orb}_A(S^1)$ is the sum of all untwisted and twisted states on a circle. The amplitude $\hat\tau^\mathrm{orb}_A(M)$ amounts to ``averaging over the gauge symmetry'' in the sense that any two disc-shaped regions in the complement of the defect network in $M^{\text{$A$-network}}$ can only communicate through $A$-defects, which we can think of as implementing the averaging. Finally, in passing to the nondegenerate theory one has restricted the state space to gauge invariant states.
\end{remark}

\subsection{Domain walls between orbifolded theories}

In introducing the orbifolding procedure we have concentrated on defining a closed 2d TFT $\tau^\mathrm{orb}_A$ without defects as an orbifold of a 2d TFT with defects. However, one can also easily describe the domain walls between two orbifolded theories in terms of the orbifolding defects. This gives rise to a new and `larger' TFT with defects $\tau^\mathrm{orb}$ whose worldsheet phases are labelled by orbifolding defects, as we will now explain. 

\medskip

As before, let $\tau$ be a 2d TFT with defects.
Let $a,b$ be two worldsheet phases and let $A \in \mathcal{D}_\tau(a,a)$ and $B \in \mathcal{D}_\tau(b,b)$ be orbifolding defects. Then each $B$-$A$-bimodule $X \in \mathcal{D}_\tau(a,b)$ describes a domain wall from the $A$-orbifold of $a$ to the $B$-orbifold of $b$. 
More generally, we can define $D_2^\mathrm{orb}$ to be the set of pairs $(a,A)$ where $a \in D_2$ is arbitrary and $A \in \mathcal{D}_\tau(a,a)$ is an orbifolding defect. $D_2^\mathrm{orb}$ describes the theories which can be reached from $\tau$ by the orbifolding procedure, and we will refer to elements of $D_2^\mathrm{orb}$ as \textsl{orbifold phases}. 

The set of domain walls $D_1^\mathrm{orb}$ consists of triples $\big((b,B),X,(a,A)\big)$ where $(a,A)$ and $(b,B)$ are orbifold phases in $D_2^\mathrm{orb}$ and $X \in \mathcal{D}_\tau(a,b)$ is a $B$-$A$-bimodule. The source map $s : D_1^\mathrm{orb} \to D_2^\mathrm{orb}$ produces $(a,A)$ and the target map $t$ returns $(b,B)$.

We can now define a new 2d TFT with defects in terms of $\tau$, the {\sl orbifold completion of} $\tau$. Namely, we construct a functor
\begin{align}
  \tau^\mathrm{orb} : \mathrm{Bord}_{2,1}^\mathrm{def}(D_2^\mathrm{orb},D_1^\mathrm{orb})\longrightarrow \operatorname{Vect} 
\end{align}
analogously to the purely closed case discussed in the previous section. To an object $U$ of $\mathrm{Bord}_{2,1}^\mathrm{def}(D_2^\mathrm{orb},D_1^\mathrm{orb})$ it assigns the image of an idempotent $P_U$ in $\tau(U)$. If~$U$ is a single circle decorated with a $B$-$A$-bimodule~$X$ and an $A$-$B$-bimodule~$Y$ we have
\be
P_U = 
\tau \Biggl(
\begin{tikzpicture}[very thick, scale=0.8,color=green!50!black, baseline]
\nicecolourscheme (0,0) circle (2);
\draw[very thick, color=blue!22] (0,0) circle (2);

\draw[very thick, color=white, fill=white] (0,0) circle (0.6);

\fill[color=blue!50!black] (50:2) circle (3.0pt) node[left] {};
\fill[color=blue!50!black] (130:2) circle (3.0pt) node[left] {};
\fill[color=blue!50!black] (50:0.6) circle (3.0pt) node[left] {};
\fill[color=blue!50!black] (130:0.6) circle (3.0pt) node[left] {};
\draw[color=blue!50!black] (50:0.6) -- (50:2);
\draw[color=blue!50!black] (130:0.6) -- (130:2);
\fill (130:1.7) circle (3.0pt) node[left] {};
\fill (130:1.01) circle (3.0pt) node[left] {};
\fill (50:1.7) circle (3.0pt) node[left] {};
\fill (50:1.01) circle (3.0pt) node[left] {};
\draw (50:1.01) arc (50:130:1.01);
\draw (130:1.7) arc (130:410:1.7);
\draw (-90:0.85) [black] node {{\scriptsize $a$}};
\draw (90:1.8) [black] node {{\scriptsize $b$}};
\draw (90:1.22) node {{\scriptsize $B$}};
\draw (-90:1.48) node {{\scriptsize $A$}};
\draw (42:1.35) [blue!50!black] node {{\scriptsize $X$}};
\draw (139:1.35) [blue!50!black] node {{\scriptsize $Y$}};
\end{tikzpicture}
\Biggl)
\, ,
\ee
where circle segments labelled by~$A$ and~$B$ correspond to idempotents of type~\eqref{eq:piA}. For example, if $X$ is oriented inwards and $Y$ outwards, $P_U$ implements the projection~\eqref{eq:bimodulemapprojector} to bimodule maps in the space of 2-morphisms $X \to Y$, cf.\ \eqref{eq:2-morph-via-tau}.
If~$U$ has a different number of circles and defect decorations $P_U$ is constructed analogously. 
When writing $\tau(U)$ we implicitly use the forgetful functor $\mathrm{Bord}_{2,1}^\mathrm{def}(D_2^\mathrm{orb},D_1^\mathrm{orb}) \to \mathrm{Bord}_{2,1}^\mathrm{def}(D_2,D_1)$ which forgets the orbifolding defects and the bimodule actions in the labelling of objects and morphisms.

Similar to the purely closed case of Section~\ref{sec:TFT-orbifold}, 
for a morphism $M : U \to V$ in  $\mathrm{Bord}_{2,1}^\mathrm{def}(D_2^\mathrm{orb},D_1^\mathrm{orb})$  we set 
$\tau^\mathrm{orb}(M) = r_V \circ \tau(M^{\text{network}}) \circ e_U$, where a fine enough $A_i$-network is placed inside each phase labelled $(a_i,A_i)$, and the $A_i$-defect lines can end on a bounding domain wall via a junction labelled by the bimodule action. Again it is not hard to check that $\tau(M^{\text{network}})$ is independent of the choice of network and that $\tau^\mathrm{orb}(M)$ defines a nondegenerate 2d TFT with defects.

We summarise this somewhat sketchy discussion as the following result. 

\begin{theorem}
Each 2d TFT with defects $\tau : \mathrm{Bord}_{2,1}^\mathrm{def}(D_2,D_1) \to  \operatorname{Vect}$ gives rise to a nondegenerate 2d TFT with defects $\tau^\mathrm{orb} : \mathrm{Bord}_{2,1}^\mathrm{def}(D_2^\mathrm{orb},D_1^\mathrm{orb}) \to  \operatorname{Vect}$, called the orbifold completion of $\tau$.
\end{theorem}

\begin{remark}
As an instance of such an orbifold completion, we note that the lattice construction of defect TFTs presented in \cite[Sect.\,3]{dkr1107.0495} can be understood as an orbifold completion of the trivial closed 2d TFT $\tau_{\mathrm{triv}}$. For the trivial theory we have $D_2 = \{ \circ \}$ and $D_1 = \emptyset$, and before including formal sums (see the beginning of Section~\ref{subsec:worldsheetbicategory}), 
$\tau_{{\mathrm{triv}}}$ maps all objects to $\C$ and all morphisms to the identity map. The inclusion of formal sums means we have in addition the defects $(I_\circ)^{\oplus n}$ at our disposal. For example, the state space of a circle with a single marked point labelled by $(I_\circ)^{\oplus n}$ is $\C^n$. The orbifolding defects now correspond to symmetric separable Frobenius algebras over $\C$ and the domain walls to bimodules thereof.
\end{remark}

\section{Equivariant bicategory}\label{sec:equibicat}

Motivated by the discussion of Sections~\ref{introduction} and~\ref{sec:2dTFTwithDefects} we now begin our orbifold construction in the framework of bicategories. It turns out that many results such as a completeness property and the construction of certain equivalences can already be obtained without demanding the algebras involved to satisfy all the conditions of orbifolding defects as defined in Section~\ref{sec:TFT-orbifold}. This will be explained in the present section. 

\subsection[Definition of the equivariant completion $\Beq$]{Definition of the equivariant completion $\boldsymbol{\Beq}$}\label{subsec:Beq}

Let~$\B$ be a bicategory whose categories of 1-morphisms are idempotent complete. Motivated by the discussion of Section~\ref{sec:2dTFTwithDefects} we construct a new bicategory out of $\B$:

\begin{definition}
The \textsl{equivariant completion} $\Beq$ of~$\B$ consists of the following data: 
\begin{itemize}
\item
Objects in $\Beq$ are pairs $(a,A)$ with $a\in\B$ and $A\in\B(a,a)$ a separable Frobenius algebra. 
\item 
1-morphisms $(a,A) \rightarrow (b,B)$ in $\Beq$ are $X \in \B(a,b)$ with the structure of an $B$-$A$-bimodule. 
\item 
2-morphisms in $\Beq$ are 2-morphisms in~$\B$ that are bimodule maps. 
\item 
The composition of 1-morphisms $X: (a,A) \rightarrow (b,B)$ and $Y: (b,B) \rightarrow (c,C)$ is the tensor product $Y \otimes_B X : (a,A) \rightarrow (c,C)$ (which exists by the assumption of idempotent completeness, see Lemma~\ref{lem:tensorprojector}). The composition of 2-morphisms in $\Beq$ is that of~$\B$. The associator in $\Beq$ is the one induced from~$\B$, since by Remark~\ref{rem:split-coequaliser} the coequaliser defining $Y \otimes_B X$ is split and hence preserved by any functor, in particular by horizontal composition with another 1-morphism. 
\item The unit 1-morphism for $(a,A) \in \Beq$ is~$A$. The left and right unit action on $X: (a,A) \rightarrow (b,B)$ is given by the left and right action on the corresponding bimodule, respectively, and the inverse unit actions are given by 
\be
\begin{tikzpicture}[very thick,scale=1.0,color=blue!50!black, baseline=.8cm]
\draw[line width=0pt] 
(-0.03,1.75) node[line width=0pt, color=green!50!black] (A) {{\small$B$}}
(0.86,1.75) node[line width=0pt] (X1) {{\small$X$}}
(0.86,-0.5) node[line width=0pt] (X2) {{\small$X$}}; 

\draw[color=green!50!black] (0.47,0.55) .. controls +(0.0,0.25) and +(-0.25,-0.15) .. (0.86,0.95);
\draw[-dot-, color=green!50!black] (0.47,0.55) .. controls +(0,-0.5) and +(0,-0.5) .. (-0.03,0.55);

\draw[color=green!50!black] (0.22,-0.1) node[Odot] (unit) {}; 
\draw[color=green!50!black] (0.22,0.15) -- (unit);

\fill[color=green!50!black] (0.86,0.95) circle (2pt) node (meet) {};

\draw[color=green!50!black] (-0.03,0.55) -- (A);
\draw[color=blue!50!black] (X1) -- (X2);

\end{tikzpicture}
\, , \quad 
\begin{tikzpicture}[very thick,scale=1.0,color=blue!50!black, baseline=.8cm]
\draw[line width=0pt] 
(1.75,1.75) node[line width=0pt, color=green!50!black] (A) {{\small$A$}}
(0.86,1.75) node[line width=0pt] (X1) {{\small$X$}}
(0.86,-0.5) node[line width=0pt] (X2) {{\small$X$}}; 

\draw[color=green!50!black] (1.25,0.55) .. controls +(0.0,0.25) and +(0.25,-0.15) .. (0.86,0.95);
\draw[-dot-, color=green!50!black] (1.25,0.55) .. controls +(0,-0.5) and +(0,-0.5) .. (1.75,0.55);

\draw[color=green!50!black] (1.5,-0.1) node[Odot] (unit) {}; 
\draw[color=green!50!black] (1.5,0.15) -- (unit);

\fill[color=green!50!black] (0.86,0.95) circle (2pt) node (meet) {};

\draw[color=green!50!black] (1.75,0.55) -- (A);
\draw[color=blue!50!black] (X1) -- (X2);

\end{tikzpicture}
\, ,
\ee
where here and below all string diagrams are drawn in~$\B$, not in $\Beq$. 
\end{itemize}
\end{definition}

A first observation about the equivariant completion is that  the original bicategory~$\B$ fully embeds in $\Beq$. Indeed, it is straightforward to see that the left or right actions of the units in~$\B$ make~$I_a$ into a separable Frobenius algebra for any $a\in\B$. Thus $a \mapsto (a,I_a)$ is a full embedding as every 1-morphism $a\rightarrow b$ in~$\B$ is an $I_b$-$I_a$-bimodule, and 2-morphisms are bimodule maps. 

\medskip

The term ``equivariant'' in the name equivariant completion will be motivated in Section~\ref{subsec:equiMF} where we will show how the standard theory of equivariant Landau-Ginzburg models embeds into the general framework developed here. 
The term ``completion'' is justified because $\Beq$ is invariant under the equivariantisation procedure. We will prove this in Proposition \ref{prop:eqeqeq} below, but before getting there, we briefly illustrate the intuition behind the technical proof (even if this intuition is rooted in the stronger assumptions of Section~\ref{sec:orbibicat}). 

\medskip

Let us fix a theory $\mathcal T$ and an orbifolding defect~$A$ in it (cf.~Section~\ref{sec:TFT-orbifold}). The idea is that correlators in the $A$-orbifold theory $\mathcal T_A$ are computed from correlators in~$\mathcal T$ with a fine enough $A$-defect network. 

Now let~$B$ be an orbifolding defect in $\mathcal T_A$ (and consequently also in~$\mathcal T$). Correlators in $(\mathcal T_A)_B$ are correlators in $\mathcal T_A$ with a fine enough $B$-network, and hence correlators in~$\mathcal T$ with fine enough $A$-networks inside all phases of a fine enough $B$-network. But since the $B$-network is already fine enough we can take the $A$-network to be trivial, thus arriving at
\be
\langle \ldots \rangle_{(\mathcal T_A)_B} = \langle \ldots \rangle_{\mathcal T_B}
\ee
for all correlators. Analogously we find that $\Beq$ is already ``complete'': 

\begin{proposition}\label{prop:eqeqeq}
$(\Beq)_{\mathrm{eq}} \cong \Beq$. 
\end{proposition}

\begin{proof}
We will show that the full embedding $\Beq \to (\Beq)_{\mathrm{eq}}$ is essentially surjective, i.\,e.~for every object in $(\Beq)_{\mathrm{eq}}$ there is a 1-isomorphism to an object in the image of $\Beq$. The proof boils down to the statement that for an algebra~$B$ one has $B \otimes_B B \cong B$ as $B$-$B$-bimodules; the main difficulty is to not get lost in notation along the way.

Fix an object $\big((a,A),\mathbb{B}\big)$ in $(\Beq)_{\mathrm{eq}}$. Then $A$ is a separable Frobenius algebra in $\B(a,a)$, and we will denote its unit and multiplication maps as $\eta_A : I_a \to A$ and $\mu_A : A \otimes A \to A$.
$\mathbb{B}$ is an $A$-$A$-bimodule in $\B(a,a)$, together with unit $\eta_{\mathbb{B}} : A \to \mathbb{B}$ and multiplication $\mu_{\mathbb{B}} :  \mathbb{B} \otimes_A  \mathbb{B} \to  \mathbb{B}$, both of which are $A$-$A$-bimodule maps, and analogously for coproduct and counit. Write $\mathbb{B} = (B, \rho^l, \rho^r)$, where $B \in\B(a,a)$ is the underlying object of the bimodule, $\rho^l : A \otimes B \to B$ is the left action, and $\rho^r$ the right action.

Denote the canonical projection $B \otimes B \to \mathbb{B} \otimes_A  \mathbb{B}$ by $r$. Then $\mu_B := \mu_{\mathbb{B}} \circ r : B \otimes B \to B$ and $\eta_B := \eta_{\mathbb{B}} \circ \eta_A : I_a \to B$ turn $B$ into a unital algebra in $\B(a,a)$. Analogously, the coproduct and counit of $\mathbb{B}$ turn $B$ into a separable Frobenius algebra. $B$ is the unit of the category of $B$-$B$-bimodules in $\B(a,a)$ and we will write $\mathbb{I}_B$ when it is used in this function. The embedding of the object $(a,B)\in\Beq$  into $(\Beq)_{\mathrm{eq}}$ is $\big((a,B),\mathbb{I}_B\big)$.
The proposition is proved once we have established the following claim: 
\begin{quote}
$\big((a,A),\mathbb{B}\big)$ and $\big((a,B),\mathbb{I}_B\big)$ are isomorphic in $(\Beq)_{\mathrm{eq}}$.
\end{quote}

We start with a 1-morphism $\mathbb{X} : \big((a,B),\mathbb{I}_B\big) \to \big((a,A),\mathbb{B}\big)$. This means, first of all, an underlying 1-morphism $X : (a,B) \to (a,A)$, i.\,e.~an $A$-$B$-bimodule $X$, which as a 1-morphism in~$\B$ we take to be~$B$. Secondly, $X$ has to be equipped with a left action $\chi^l : \mathbb{B} \otimes_A X \to X$ which is also an $A$-$B$-bimodule map. The right action $\chi^r : X \otimes_B \mathbb{I}_B \to X$ is by definition the unit isomorphism. Altogether, $\mathbb{X} = (X , \chi^l , \chi^r)$. We choose the $A$-$B$-bimodule $X$ to be $(B,\rho^l,\mu_B)$, i.\,e.~the left action comes from the $A$-$A$-bimodule structure of $B$ and the right action is the multiplication of $B$. For the map $\chi^l$ we take the multiplication $\mu_{\mathbb{B}}$ of $\mathbb{B}$.

Analogously, we construct a 1-morphism $\mathbb{Y} : \big((a,A),\mathbb{B}\big) \to \big((a,B),\mathbb{I}_B\big)$ with underlying object $B$, but with interchanged left and right actions. 

Next we establish that $\mathbb{X}$ and $\mathbb{Y}$ are inverse to each other. Consider first $\mathbb{X} \otimes_{\mathbb{I}_B} \mathbb{Y}$. The tensor product over the tensor unit just produces the tensor product in the underlying category which is $B \otimes_B B \cong B$. The $B$ on the right-hand side is equipped with left and right $\mathbb{B}$-action, and is in fact $\mathbb{B}$ as a bimodule over itself. Finally, $\mathbb{Y} \otimes_{\mathbb{B}} \mathbb{X}$ is again equal to $B \otimes_B B$ since the coequaliser of the left and right $\mathbb{B}$-action $\mathbb{Y} \otimes_A \mathbb{B} \otimes_A \mathbb{X} \to \mathbb{Y} \otimes_A \mathbb{X}$ inside $A$-$A$-bimodules in $\B(a,a)$ is the same as the coequaliser $Y \otimes B \otimes X \to Y \otimes X$ of the left and right $B$-action in $\B(a,a)$.
\end{proof}

\subsection[Equivalences in $\Beq$ from adjunctions]{Equivalences in $\boldsymbol{\Beq}$ from adjunctions}\label{subsec:eqBeq}

The following result is a slight generalisation of a known construction of Frobenius algebras, see e.\,g.~\cite[Lem.\,3.4]{m0111204}, which provides a way of explicitly constructing orbifolding defects. 

In this section we denote the adjunction maps for 
$Y \dashv X$ by $\varepsilon : Y \otimes X \to I$ and $\eta : I \to X \otimes Y$, and those for $X \dashv Y$ by $\tilde\varepsilon$ and $\tilde\eta$.

\begin{proposition}\label{prop:algebraXdaggerX}
Let~$\B$ be a bicategory (which for the purpose of this proposition need not have idempotent complete 1-morphism categories).	
\begin{enumerate}
\item 
An adjunction $X \dashv Y$ in~$\B$ gives rise to an algebra structure on $Y\otimes X$ and to a coalgebra structure of $X\otimes Y$. 
\item 
Adjunctions $Y \dashv X \dashv Y$ give a Frobenius algebra structure on $A := Y\otimes X$. 
If $\tilde\varepsilon \circ \eta$ is the identity, 
$A$ is separable.
\item 
If $\tilde\varepsilon \circ \eta$ is invertible, we can twist the adjunction $Y \dashv X$ in (ii) by the automorphism $\varphi := \lambda_X \circ ((\tilde\varepsilon\circ \eta)^{-1} \otimes 1_X) \circ \lambda_X^{-1}$. This does not affect the algebra structure on $A = Y\otimes X$ but does turn $A$ into a separable Frobenius algebra. 
\end{enumerate}
\end{proposition}

\begin{proof}
\begin{enumerate}
\item 
The algebra structure on $Y\otimes X$ is given by
\be\label{eq:algebra-on-YX}
\begin{tikzpicture}[very thick,scale=0.75,color=green!50!black, baseline=0.4cm]
\draw[-dot-] (3,0) .. controls +(0,1) and +(0,1) .. (2,0);
\draw (2.5,0.75) -- (2.5,1.5); 
\end{tikzpicture} 
=
\begin{tikzpicture}[very thick,scale=0.9,color=blue!50!black, baseline=.9cm]
\draw[line width=0pt] 
(3,0) node[line width=0pt] (D) {{\small$Y$}}
(2,0) node[line width=0pt] (s) {{\small$X\vphantom{Y}$}}; 
\draw[redirected] (D) .. controls +(0,1) and +(0,1) .. (s);
\draw[line width=0pt] (2.5,0.5) node[line width=0pt] (D) {{\footnotesize$\tilde\varepsilon$}};
\draw[line width=0pt] 
(3.45,0) node[line width=0pt] (re) {{\small$X\vphantom{Y}$}}
(1.55,0) node[line width=0pt] (li) {{\small$Y$}}; 
\draw[line width=0pt] 
(2.7,2) node[line width=0pt] (ore) {{\small$X\vphantom{Y}$}}
(2.3,2) node[line width=0pt] (oli) {{\small$Y$}}; 
\draw (li) .. controls +(0,0.75) and +(0,-0.25) .. (2.3,1.25);
\draw (2.3,1.25) -- (oli);
\draw (re) .. controls +(0,0.75) and +(0,-0.25) .. (2.7,1.25);
\draw (2.7,1.25) -- (ore);
\end{tikzpicture}
, \quad
\begin{tikzpicture}[very thick,scale=0.75,color=green!50!black, baseline]
\draw (0,-0.5) node[Odot] (D) {}; 
\draw (D) -- (0,0.6); 
\end{tikzpicture} 
=
\begin{tikzpicture}[very thick,scale=1.0,color=blue!50!black, baseline=-.4cm,rotate=180]
\draw[line width=0pt] 
(3,0) node[line width=0pt] (D) {{\small$Y$}}
(2,0) node[line width=0pt] (s) {{\small$X\vphantom{Y}$}}; 
\draw[directed] (D) .. controls +(0,1) and +(0,1) .. (s);
\draw[line width=0pt] (2.5,0.5) node[line width=0pt] (D) {{\footnotesize$\tilde\eta$}};
\end{tikzpicture}
\ee
while the coalgebra structure on $X \otimes Y$ is 
\be
\begin{tikzpicture}[very thick,scale=0.75,color=green!50!black, baseline=-0.6cm, rotate=180]
\draw[-dot-] (3,0) .. controls +(0,1) and +(0,1) .. (2,0);
\draw (2.5,0.75) -- (2.5,1.5); 
\end{tikzpicture} 
=
\begin{tikzpicture}[very thick,scale=0.9,color=blue!50!black, baseline=-0.9cm, rotate=180]
\draw[line width=0pt] 
(3,0) node[line width=0pt] (D) {{\small$Y\vphantom{Y}$}}
(2,0) node[line width=0pt] (s) {{\small${X}$}}; 
\draw[redirected] (s) .. controls +(0,1) and +(0,1) .. (D);
\draw[line width=0pt] (2.5,0.5) node[line width=0pt] (D) {{\footnotesize$\tilde\eta$}};
\draw[line width=0pt] 
(3.45,0) node[line width=0pt] (re) {{\small${X}$}}
(1.55,0) node[line width=0pt] (li) {{\small$Y\vphantom{Y}$}}; 
\draw[line width=0pt] 
(2.7,2) node[line width=0pt] (ore) {{\small${X}$}}
(2.3,2) node[line width=0pt] (oli) {{\small$Y\vphantom{Y}$}}; 
\draw (li) .. controls +(0,0.75) and +(0,-0.25) .. (2.3,1.25);
\draw (2.3,1.25) -- (oli);
\draw (re) .. controls +(0,0.75) and +(0,-0.25) .. (2.7,1.25);
\draw (2.7,1.25) -- (ore);
\end{tikzpicture}
, \quad
\begin{tikzpicture}[very thick,scale=0.75,color=green!50!black, baseline, rotate=180]
\draw (0,-0.5) node[Odot] (D) {}; 
\draw (D) -- (0,0.6); 
\end{tikzpicture} 
= 
\begin{tikzpicture}[very thick,scale=1.0,color=blue!50!black, baseline=.4cm]
\draw[line width=0pt] 
(3,0) node[line width=0pt] (D) {{\small$Y\vphantom{Y}$}}
(2,0) node[line width=0pt] (s) {{\small$X$}}; 
\draw[directed] (s) .. controls +(0,1) and +(0,1) .. (D);
\draw[line width=0pt] (2.5,0.55) node[line width=0pt] (D) {{\footnotesize$\tilde\varepsilon$}};
\end{tikzpicture}
\ee
where we write the adjunction maps of $X \dashv Y$ as
$
\begin{tikzpicture}[thick,scale=0.35,color=blue!50!black, baseline=0.05cm]
\draw[redirected] (3,0) .. controls +(0,1) and +(0,1) .. (2,0);
\draw[line width=0pt] (2.5,0.2) node[line width=0pt] (D) {{\scriptsize$\tilde\varepsilon$}};
\end{tikzpicture} 
$
and 
$
\begin{tikzpicture}[thick,scale=0.35,color=blue!50!black, baseline=-0.22cm]
\draw[redirected] (3,0) .. controls +(0,-1) and +(0,-1) .. (2,0);
\draw[line width=0pt] (2.5,-0.15) node[line width=0pt] (D) {{\scriptsize$\tilde\eta$}};
\end{tikzpicture} 
$. 
Checking the defining properties of  (co)algebras is straightforward; e.\,g.~for associativity one observes
\be
\begin{tikzpicture}[very thick,scale=0.75,color=green!50!black, baseline=0.6cm]
\draw[-dot-] (3,0) .. controls +(0,1) and +(0,1) .. (2,0);
\draw[-dot-] (2.5,0.75) .. controls +(0,1) and +(0,1) .. (3.5,0.75);
\draw (3.5,0.75) -- (3.5,0); 
\draw (3,1.5) -- (3,2.25); 
\end{tikzpicture} 
=
\begin{tikzpicture}[very thick, scale=0.8, color=blue!50!black, baseline=0.9cm]
\draw[line width=0pt] 
(3,0) node[line width=0pt] (D) {}
(2,0) node[line width=0pt] (s) {}; 
\draw[redirected] (D) .. controls +(0,1) and +(0,1) .. (s);
\draw[line width=0pt] (2.5,0.5) node[line width=0pt] (D) {{\scriptsize$\tilde\varepsilon$}};
\draw[line width=0pt] 
(3.25,0) node[line width=0pt] (re) {}
(1.75,0) node[line width=0pt] (li) {}; 
\draw (li) .. controls +(0,0.75) and +(0,-0.25) .. (2.4,1.15);
\draw (2.4,1.15) -- (2.4,1.4);
\draw (re) .. controls +(0,0.75) and +(0,-0.25) .. (2.6,1.15);
\draw (2.6,1.15) -- (2.6,1.4);
\draw[line width=0pt] 
(4.1,0) node[line width=0pt] (r1) {}
(4.35,0) node[line width=0pt] (r2) {};
\draw[directed] (2.6,1.4) .. controls +(0,1) and +(0,1) .. (4.1,1.4);
\draw (4.1,1.4) -- (r1);
\draw (4.35,1.4) -- (r2);
\draw[line width=0pt] (3.35,1.8) node[line width=0pt] (D) {{\scriptsize{$\tilde\varepsilon$}}};
\draw (2.4,1.4) .. controls +(0,0.75) and +(0,-0.35) .. (3.28,2.55);
\draw (4.35,1.4) .. controls +(0,0.75) and +(0,-0.35) .. (3.47,2.55);
\draw (3.28,2.55) -- (3.28,2.75);
\draw (3.47,2.55) -- (3.47,2.75);
\end{tikzpicture}
= \!\!
\reflectbox{%
\begin{tikzpicture}[very thick, scale=0.8, color=blue!50!black, baseline=0.9cm]
\draw[line width=0pt] 
(3,0) node[line width=0pt] (D) {}
(2,0) node[line width=0pt] (s) {}; 
\draw[directed] (D) .. controls +(0,1) and +(0,1) .. (s);
\draw[line width=0pt] (2.5,0.5) node[line width=0pt] (D) {{\scriptsize{\reflectbox{$\tilde\varepsilon$}}}};
\draw[line width=0pt] 
(3.25,0) node[line width=0pt] (re) {}
(1.75,0) node[line width=0pt] (li) {}; 
\draw (li) .. controls +(0,0.75) and +(0,-0.25) .. (2.4,1.15);
\draw (2.4,1.15) -- (2.4,1.4);
\draw (re) .. controls +(0,0.75) and +(0,-0.25) .. (2.6,1.15);
\draw (2.6,1.15) -- (2.6,1.4);
\draw[line width=0pt] 
(4.1,0) node[line width=0pt] (r1) {}
(4.35,0) node[line width=0pt] (r2) {};
\draw[redirected] (2.6,1.4) .. controls +(0,1) and +(0,1) .. (4.1,1.4);
\draw (4.1,1.4) -- (r1);
\draw (4.35,1.4) -- (r2);
\draw[line width=0pt] (3.35,1.8) node[line width=0pt] (D) {{\scriptsize{\reflectbox{$\tilde\varepsilon$}}}};
\draw (2.4,1.4) .. controls +(0,0.75) and +(0,-0.35) .. (3.28,2.55);
\draw (4.35,1.4) .. controls +(0,0.75) and +(0,-0.35) .. (3.47,2.55);
\draw (3.28,2.55) -- (3.28,2.75);
\draw (3.47,2.55) -- (3.47,2.75);
\end{tikzpicture}
}
=
\begin{tikzpicture}[very thick,scale=0.75,color=green!50!black, baseline=0.6cm]
\draw[-dot-] (3,0) .. controls +(0,1) and +(0,1) .. (2,0);
\draw[-dot-] (2.5,0.75) .. controls +(0,1) and +(0,1) .. (1.5,0.75);
\draw (1.5,0.75) -- (1.5,0); 
\draw (2,1.5) -- (2,2.25); 
\end{tikzpicture} 
.
\ee
\item 
Writing 
$
\begin{tikzpicture}[thick,scale=0.35,color=blue!50!black, baseline=0.05cm]
\draw[redirected] (2,0) .. controls +(0,1) and +(0,1) .. (3,0);
\draw[line width=0pt] (2.5,0.2) node[line width=0pt] (D) {{\scriptsize$\varepsilon$}};
\end{tikzpicture} 
$
and 
$
\begin{tikzpicture}[thick,scale=0.35,color=blue!50!black, baseline=-0.22cm]
\draw[redirected] (2,0) .. controls +(0,-1) and +(0,-1) .. (3,0);
\draw[line width=0pt] (2.5,-0.15) node[line width=0pt] (D) {{\scriptsize$\eta$}};
\end{tikzpicture} 
$ 
for the adjunction maps of $Y \dashv X$, the coalgebra structure on $A = Y \otimes X$ is 
\be
\begin{tikzpicture}[very thick,scale=0.75,color=green!50!black, baseline=-0.6cm, rotate=180]
\draw[-dot-] (3,0) .. controls +(0,1) and +(0,1) .. (2,0);
\draw (2.5,0.75) -- (2.5,1.5); 
\end{tikzpicture} 
=
\begin{tikzpicture}[very thick,scale=0.9,color=blue!50!black, baseline=-0.9cm, rotate=180]
\draw[line width=0pt] 
(3,0) node[line width=0pt] (D) {{\small$X\vphantom{Y}$}}
(2,0) node[line width=0pt] (s) {{\small${Y}$}}; 
\draw[redirected] (D) .. controls +(0,1) and +(0,1) .. (s);
\draw[line width=0pt] (2.5,0.5) node[line width=0pt] (D) {{\footnotesize$\eta$}};
\draw[line width=0pt] 
(3.45,0) node[line width=0pt] (re) {{\small${Y}$}}
(1.55,0) node[line width=0pt] (li) {{\small$X\vphantom{Y}$}}; 
\draw[line width=0pt] 
(2.7,2) node[line width=0pt] (ore) {{\small${Y}$}}
(2.3,2) node[line width=0pt] (oli) {{\small$X\vphantom{Y}$}}; 
\draw (li) .. controls +(0,0.75) and +(0,-0.25) .. (2.3,1.25);
\draw (2.3,1.25) -- (oli);
\draw (re) .. controls +(0,0.75) and +(0,-0.25) .. (2.7,1.25);
\draw (2.7,1.25) -- (ore);
\end{tikzpicture}
, \quad
\begin{tikzpicture}[very thick,scale=0.75,color=green!50!black, baseline, rotate=180]
\draw (0,-0.5) node[Odot] (D) {}; 
\draw (D) -- (0,0.6); 
\end{tikzpicture} 
= 
\begin{tikzpicture}[very thick,scale=1.0,color=blue!50!black, baseline=.4cm]
\draw[line width=0pt] 
(3,0) node[line width=0pt] (D) {{\small$X\vphantom{Y}$}}
(2,0) node[line width=0pt] (s) {{\small$Y$}}; 
\draw[directed] (D) .. controls +(0,1) and +(0,1) .. (s);
\draw[line width=0pt] (2.5,0.55) node[line width=0pt] (D) {{\footnotesize$\varepsilon$}};
\end{tikzpicture}
\ee
and the Frobenius condition is easily checked. If $\tilde\varepsilon \circ \eta = 1_I$
then~$A$ is separable: 
\be
\begin{tikzpicture}[very thick,scale=0.85,color=green!50!black, baseline=0cm]
\draw[-dot-] (0,0) .. controls +(0,-1) and +(0,-1) .. (1,0);
\draw[-dot-] (0,0) .. controls +(0,1) and +(0,1) .. (1,0);
\draw (0.5,-0.8) -- (0.5,-1.2); 
\draw (0.5,0.8) -- (0.5,1.2); 
\end{tikzpicture}
=
\begin{tikzpicture}[very thick,scale=0.45,color=blue!50!black, baseline]
\draw (0,0) circle (1);
\draw[->, very thick] (0.100,-1) -- (-0.101,-1) node[above] {}; 
\draw[->, very thick] (0.100,1) -- (0.101,1) node[below] {}; 
\draw[line width=0pt] 
(-1.5,-2.4) node[line width=0pt] (r1) {{\small $Y$}}
(1.5,-2.4) node[line width=0pt] (r2) {{\small $X\vphantom{Y}$}}
(-1.5,2.4) node[line width=0pt] (or1) {{\small $Y$}}
(1.5,2.4) node[line width=0pt] (or2) {{\small $X\vphantom{Y}$}};
\draw (r1) -- (or1); 
\draw (r2) -- (or2); 
\draw[line width=0pt] (0,0.4) node[line width=0pt] (D) {{\scriptsize$\tilde\varepsilon$}};
\draw[line width=0pt] (0,0-.4) node[line width=0pt] (D) {{\scriptsize$\eta$}};
\end{tikzpicture} 
= 
\begin{tikzpicture}[very thick,scale=0.85,color=green!50!black, baseline=0cm]
\draw (0.5,1.2) -- (0.5,-1.2); 
\end{tikzpicture} 
\, . 
\ee
\item 
Twisting by~$\varphi$ means replacing $\varepsilon$ and $\eta$ by 
$\varepsilon' := \varepsilon \circ (1_Y \otimes \varphi^{-1})$
and 
$\eta' := 
(\varphi \otimes 1_Y) \circ \eta$, see Footnote~\ref{fn:twisting}.
The algebra structure remains unchanged since it only depends $\tilde\varepsilon$ and $\tilde\eta$. The Frobenius algebra structure obtained from $\tilde\varepsilon$, $\tilde\eta$, $\varepsilon'$ and $\eta'$ is separable by (ii) since
\be
\tilde\varepsilon \circ \eta' = 
\begin{tikzpicture}[very thick,scale=0.45,color=blue!50!black, baseline]
\draw (0,0) circle (1);
\draw[->, very thick] (0.100,-1) -- (-0.101,-1) node[above] {}; 
\draw[->, very thick] (0.100,1) -- (0.101,1) node[below] {}; 
\fill (180:1) circle (3.9pt) node[left] {{\small$\varphi$}};
\draw[line width=0pt] (0,0.4) node[line width=0pt] (D) {{\scriptsize$\tilde\varepsilon$}};
\draw[line width=0pt] (0,0-.4) node[line width=0pt] (D) {{\scriptsize$\eta$}};
\end{tikzpicture} 
= (\tilde\varepsilon \circ \eta)^{-1} \circ (\tilde\varepsilon \circ \eta) = 1_I
\, . 
\ee
\end{enumerate}
\end{proof}

For the application to equivariant completions, part (iii) of the above proposition is the key point, since it states that the algebra $A$ is separable Frobenius and hence an object in $\Beq$. The next proposition describes this more precisely. Recall that we assumed $\B$ to have idempotent complete 1-morphism categories.

\begin{proposition}\label{prop:XXdaggerInverseInBeq}
Let $X \in \B(a,b)$ and $Y \in \B(b,a)$.
\begin{enumerate}
\item Suppose $X,Y$ form a biadjunction $Y \dashv X \dashv Y$ such that $\tilde\varepsilon \circ \eta$ is invertible. 
Then for $A = Y \otimes X$ there is an adjoint equivalence $X: (a,A) \rightleftarrows (b,I_b) : Y$ in $\Beq$. 
\item Conversely, if we are given an adjoint equivalence $X: (a,A) \rightleftarrows (b,I_b) : Y$ in $\Beq$, then $Y \dashv X \dashv Y$ in $\B$ with $\tilde\varepsilon \circ \eta = 1_{I_b}$, 
and $A \cong Y \otimes X$ as Frobenius algebras.
\end{enumerate}
\end{proposition}

\begin{proof}
\begin{enumerate}
\item
We show that $I_b$ satisfies the universal property of the coequaliser.
In the diagram~\eqref{eq:coequalisertensor} we set $\vartheta = \tilde\varepsilon$ and
\be
l = 
\begin{tikzpicture}[very thick, scale=1.0, color=blue!50!black, baseline=0.9cm]
\draw[line width=0pt] 
(3,0) node[line width=0pt] (D) {{\small $Y$}}
(2,0) node[line width=0pt] (s) {{\small $X\vphantom{Y}$}}; 
\draw[redirected] (D) .. controls +(0,1) and +(0,1) .. (s);
\draw[line width=0pt] (2.5,0.55) node[line width=0pt] (D) {{\scriptsize$\tilde\varepsilon$}};
\draw[line width=0pt] 
(1.0,0) node[line width=0pt] (r1) {{\small $X\vphantom{Y}$}}
(1.5,0) node[line width=0pt] (r2) {{\small $Y$}}
(1.0,2) node[line width=0pt] (or1) {{\small $X\vphantom{Y}$}}
(1.5,2) node[line width=0pt] (or2) {{\small $Y$}};
\draw (r1) -- (or1); 
\draw (r2) -- (or2); 
\end{tikzpicture}
, \quad
r = 
\begin{tikzpicture}[very thick, scale=1.0, color=blue!50!black, baseline=0.9cm]
\draw[line width=0pt] 
(3,0) node[line width=0pt] (D) {{\small $Y$}}
(2,0) node[line width=0pt] (s) {{\small $X\vphantom{Y}$}}; 
\draw[redirected] (D) .. controls +(0,1) and +(0,1) .. (s);
\draw[line width=0pt] (2.5,0.55) node[line width=0pt] (D) {{\scriptsize$\tilde\varepsilon$}};
\draw[line width=0pt] 
(3.5,0) node[line width=0pt] (r1) {{\small $X\vphantom{Y}$}}
(4.0,0) node[line width=0pt] (r2) {{\small $Y$}}
(3.5,2) node[line width=0pt] (or1) {{\small $X\vphantom{Y}$}}
(4.0,2) node[line width=0pt] (or2) {{\small $Y$}};
\draw (r1) -- (or1); 
\draw (r2) -- (or2); 
\end{tikzpicture}
. 
\ee
Then for any $\phi: X\otimes Y \rightarrow Z$ with $\phi\circ l = \phi \circ r$ we observe
\be
\begin{tikzpicture}[very thick, scale=1.0, color=blue!50!black, baseline=0.9cm]
\draw[line width=0pt] 
(3,0) node[line width=0pt] (D) {{\small $Y$}}
(2,0) node[line width=0pt] (s) {{\small $X\vphantom{Y}$}}; 
\draw[redirected] (D) .. controls +(0,1) and +(0,1) .. (s);
\draw[line width=0pt] (2.5,0.55) node[line width=0pt] (D) {{\scriptsize$\tilde\varepsilon$}};
\draw[line width=0pt] 
(1.0,0) node[line width=0pt] (r1) {{\small $X\vphantom{Y}$}}
(1.5,0) node[line width=0pt] (r2) {{\small $Y$}};
%
\draw (r1) -- (1.0, 1.25); 
\draw (r2) -- (1.5, 1.25); 
\draw
(1.25,1.25) node[rectangle, very thick, blue!50!black,draw,fill=white] 
(phi) {{\small$ \; \phi \; $}};
\draw[line width=0pt] 
(1.25, 2.25) node[line width=0pt] (Z) {{\small $Z$}};
\draw (phi) -- (Z);
\end{tikzpicture}
= 
\begin{tikzpicture}[very thick, scale=1.0, color=blue!50!black, baseline=0.9cm]
\draw[line width=0pt] 
(3,0) node[line width=0pt] (D) {{\small $Y$}}
(2,0) node[line width=0pt] (s) {{\small $X\vphantom{Y}$}}; 
\draw[redirected] (D) .. controls +(0,1) and +(0,1) .. (s);
\draw[line width=0pt] (2.5,0.55) node[line width=0pt] (D) {{\scriptsize$\tilde\varepsilon$}};
\draw[line width=0pt] 
(3.5,0) node[line width=0pt] (r1) {{\small $X\vphantom{Y}$}}
(4.0,0) node[line width=0pt] (r2) {{\small $Y$}};
%
\draw (r1) -- (3.5, 1.25); 
\draw (r2) -- (4.0, 1.25); 
\draw
(3.75,1.25) node[rectangle, very thick, blue!50!black,draw,fill=white] 
(phi) {{\small$ \; \phi \; $}};
\draw[line width=0pt] 
(3.75, 2.25) node[line width=0pt] (Z) {{\small $Z$}};
\draw (phi) -- (Z);
\end{tikzpicture}
\quad
\Longrightarrow
\quad
\begin{tikzpicture}[very thick, scale=1.0, color=blue!50!black, baseline=0.9cm]
\draw[redirected] (3,0) .. controls +(0,0.75) and +(0,0.75) .. (2,0);
\draw[directed] (3,0) .. controls +(0,-0.75) and +(0,-0.75) .. (2,0);
\draw[line width=0pt] 
(1.0,-1) node[line width=0pt] (r1) {{\small $X\vphantom{Y}$}}
(1.5,-1) node[line width=0pt] (r2) {{\small $Y$}};
\draw (r1) -- (1.0, 1.25); 
\draw (r2) -- (1.5, 1.25); 
\draw
(1.25,1.25) node[rectangle, very thick, blue!50!black,draw,fill=white] 
(phi) {{\small$ \; \phi \; $}};
\draw[line width=0pt] 
(1.25, 2.25) node[line width=0pt] (Z) {{\small $Z$}};
\draw (phi) -- (Z);
\draw[line width=0pt] (2.5,0.3) node[line width=0pt] (D) {{\scriptsize$\tilde\varepsilon$}};
\draw[line width=0pt] (2.5,-0.3) node[line width=0pt] (D) {{\scriptsize$\eta$}};
\end{tikzpicture}
= 
\begin{tikzpicture}[very thick, scale=1.0, color=blue!50!black, baseline=0.9cm]
\draw[redirected] (3,0) .. controls +(0,0.75) and +(0,0.75) .. (2,0);
\draw[line width=0pt] 
(3,-1) node[line width=0pt] (D) {{\small $Y$}}
(2,-1) node[line width=0pt] (s) {{\small $X\vphantom{Y}$}}; 
\draw (3, 0) -- (D); 
\draw (2, 0) -- (s); 
\draw (3.5, 0) -- (3.5, 1.25); 
\draw (4.0, 0) -- (4.0, 1.25); 
\draw[redirected] (3.5,0) .. controls +(0,-0.75) and +(0,-0.75) .. (4.0,0);
\draw
(3.75,1.25) node[rectangle, very thick, blue!50!black,draw,fill=white] 
(phi) {{\small$ \; \phi \; $}};
\draw[line width=0pt] 
(3.75, 2.25) node[line width=0pt] (Z) {{\small $Z$}};
\draw (phi) -- (Z);
\draw[line width=0pt] (2.5,0.3) node[line width=0pt] (D) {{\scriptsize$\tilde\varepsilon$}};
\draw[line width=0pt] (3.75,-0.3) node[line width=0pt] (D) {{\scriptsize$\eta$}};
\end{tikzpicture}
\ee
Composing with $(\tilde\varepsilon \circ \eta)^{-1}$ reveals that
\be
\zeta = 
\begin{tikzpicture}[very thick, scale=1.0, color=blue!50!black, baseline=1.1cm]
\draw (3.5, 0.75) -- (3.5, 1.25); 
\draw (4.0, 0.75) -- (4.0, 1.25); 
\draw[redirected] (3.5,0.75) .. controls +(0,-0.5) and +(0,-0.5) .. (4.0,0.75);
\draw
(3.75,1.25) node[rectangle, very thick, blue!50!black,draw,fill=white] 
(phi) {{\small$ \; \phi \; $}};
\draw[line width=0pt] 
(3.75, 2.25) node[line width=0pt] (Z) {{\small $Z$}};
\draw (phi) -- (Z);
\draw[line width=0pt] (3.75,0.6) node[line width=0pt] (D) {{\scriptsize$\eta$}};
\end{tikzpicture}
\, \circ 
(\tilde\varepsilon \circ \eta)^{-1}
\ee
makes~\eqref{eq:coequalisertensor} commute, thus proving $X \otimes_A Y \cong I_b$. 
\item
As we will not use this result explicitly, we will only sketch the proof. We are given an isomorphism $\varepsilon_A : Y \otimes X \to A$ of $A$-$A$-bimodules and an isomorphism $\eta_A : I_b \to X \otimes_A Y$ which form an adjunction $Y \dashv X$ in $\Beq$. This becomes a biadjunction via $\tilde\varepsilon_A := \eta_A^{-1}$ and $\tilde\eta_A := \varepsilon_A^{-1}$. Composing with the spitting maps for $\otimes_A$ produces a biadjunction $Y \dashv X \dashv Y$ in $\B$. For example, $\eta := \xi \circ \eta_A$, where $\xi : X \otimes_A Y \to X \otimes Y$. One now verifies that for the Frobenius algebra structure defined on $Y \otimes X$ via Proposition \ref{prop:algebraXdaggerX},  $\varepsilon_A$ is an isomorphism of Frobenius algebras. Finally, $\tilde\varepsilon \circ \eta = \tilde\varepsilon_A \circ \eta_A = 1_{I_b}$.
\end{enumerate}
\end{proof}

\begin{remark}
Recall that the motivation for Propositions~\ref{prop:algebraXdaggerX} and~\ref{prop:XXdaggerInverseInBeq} lies in the defect construction discussed in Section~\ref{sec:2dTFTwithDefects}. These results may be seen as a variant of the Eilenberg-Moore comparison functor and Beck's monadicity theorem, see e.\,g.~\cite[Sect.\,4.4]{Borceux2} or \cite[Sect.\,2]{BalmerStacksOfGroupRepresentations} for details.\footnote{We thank Paul Balmer and Alexei Davydov for making us aware of this parallel.} 

More precisely, from an adjunction $F \dashv G$ of functors $F: C \rightleftarrows D :G$ one obtains a monad $GF$ and the comparison functor $K: D \rightarrow C^{GF}$. Under the assumptions of the monadicity theorem (i.\,e.~if~$G$ creates coequalisers of $G$-split pairs) $K$ has a left adjoint which is an equivalence that on $M\in C^{GF}$ is the coequaliser of $FGFM \rightrightarrows M$. Hence in our setting (choosing~$\B$ to be the bicategory of categories and functors, where we view~$M$ as a left $GF$-module from the initial category to~$C$) the inverse of the comparison functor~$K$ is precisely $F \otimes_{GF}(-)$, completely analogous to the adjoint equivalence $X^\dagger: (b,I_b) \rightleftarrows (a,A) :X$ of Proposition~\ref{prop:XXdaggerInverseInBeq}. There Beck's coequaliser condition is automatically satisfied due to the existence of horizontal composition in $\Beq$ (cf.~Lemma~\ref{lem:tensorprojector}). 
\end{remark}

\begin{example}
An illustrative application of Propositions~\ref{prop:algebraXdaggerX} and~\ref{prop:XXdaggerInverseInBeq} is to obtain results of classical Galois theory in this setting.\footnote{We thank the anonymous referee for explaining this example to us.} 
Let~$\B$ be the bicategory of rings and bimodules, and let $L/K$ be a finite field extension. Then for $X := {}_K L_L \in \B(L,K)$ and $Y := {}_L L_K \in \B(K,L)$ there is an adjunction $Y \dashv X$ with $\varepsilon: L \otimes_K L \to L$, $\ell \otimes \ell' \mapsto \ell\ell'$ and $\eta: K \to L \otimes_L L$, $k \to k\otimes 1$. If $L/K$ is separable (so the pairing induced by $\tr_{L/K}(-)$ is nondegenerate)  we also have $X \dashv Y$ with adjunction maps 
$\tilde\varepsilon: L \otimes_L L \to K$, $\ell \otimes \ell' \mapsto \tr_{L/K}(\ell\ell')$ and
$\tilde\eta: L \to L \otimes_K L$, $\ell \mapsto \sum_i \ell e_i \otimes e'_i$, where $\{ e_i \}$, $\{ e'_i \}$ are dual $K$-bases of~$L$ with respect to the trace pairing. 

For this biadjunction $Y \dashv X \dashv Y$ we find $
\tilde\varepsilon \circ \eta = \tr_{L/K}(1) = [L:K]$ and $
\varepsilon \circ \tilde\eta = \sum_i e_i e'_i$. Hence after a twist by a 2-automorphism $\ell_0: X \to X$ with $\tr_{L/K}(\ell_0) = 1$, Proposition~\ref{prop:algebraXdaggerX} gives us the structure of a separable Frobenius algebra over~$K$ 
on $Y\otimes X = L\otimes_K L$, see~\eqref{eq:algebra-on-YX}. It is independent of the choice of~$\ell_0$, with 
multiplication $(\ell_1 \otimes \ell_2) \cdot (\ell'_1 \otimes \ell'_2) = \tr_{L/K}(\ell_2 \ell'_1) \ell_1 \otimes \ell'_2$ and unit $\sum_i e_i \otimes e'_i$. 

The separable Frobenius algebra $L \otimes_K L$ can be related to the twisted group ring of $G := \operatorname{Aut}_K(L)$ over~$L$, that is, to~$L\langle G \rangle = \{ \sum_{g \in G} \ell_g\,[g] \, | \, \ell_g \in L \}$ with  multiplication induced from $\ell\,[g] \cdot \ell'\,[g'] = \ell g(\ell')\,[gg']$. There is a group homomorphism $G \to (L \otimes_K L)^\times$ sending $g\in G$ to $\sum_i g(e_i) \otimes e'_i$. 
By the universal property of group rings 
it induces a map of $K$-algebras $L\langle G \rangle \to L \otimes_K L$, $\ell\,[g] \mapsto \sum_i \ell g(e_i) \otimes e'_i$.
If we further assume that $L/K$ is Galois (so $|G| = [L:K]$), this map is a bijection: it is injective since,
if $\sum_{g \in G} \ell_g \, [g]$ is in the kernel, then $\sum_{g \in G} \ell_g g(e_i) = 0$ for all~$i$. But by Dedekind's lemma the $(|G|\times |G|)$-matrix with entries $g(e_i) \in L$ is invertible, so $\ell_g=0$ for all~$g$. Surjectivity follows by dimension count.
	
In conclusion, if $L/K$ is a finite Galois extension, then according to Proposition~\ref{prop:XXdaggerInverseInBeq} we have $(K,I_K) \cong (L, L \otimes_K L)$ in $\Beq$, which by the above discussion is equivalent to $(L,L \langle G \rangle)$. In particular, the categories $\Beq((\mathbb{Z},I_{\mathbb{Z}}),(K,I_K))$ and $\Beq((\mathbb{Z},I_{\mathbb{Z}}),(L,L \langle G \rangle))$ are equivalent, i.\,e.~there is an equivalence $K$-$\operatorname{Vect} \cong \modu(L \langle G \rangle)$ between $K$-vector spaces and $L$-vector spaces with a skew-linear action of $G = \operatorname{Gal}(L/K)$. 
\end{example}

\subsection[$\Beq$ for pivotal bicategories]{$\boldsymbol{\Beq}$ for pivotal bicategories}\label{subsec:Beq-pivot}

In Sections \ref{subsec:Beq} and \ref{subsec:eqBeq} we assumed that the bicategory $\B$ has 1-morphism categories which are idempotent complete. We now demand in addition that $\B$ has adjoints satisfying $\dX  = X^\dagger$ for all 1-morphisms $X$, and that $\B$ is pivotal. 
We will show that under these additional assumptions, $\Beq$ has adjoints (but is not necessarily pivotal\footnote{$\Beq$ is, however, ``pivotal up to the action of Serre functors'' given by Nakayama twists, as follows from Proposition~\ref{prop:perfectlypairedinBeq} below together with an adaptation of the discussion in~\cite[Sect.\,7]{cm1208.1481}.}) and that the preferred adjunctions give rise to symmetric Frobenius algebras.

Recall that for any Frobenius algebra~$A$ there is the \textsl{Nakayama automorphism} 
\be\label{eq:Nakayama}
\gamma_A = 
\begin{tikzpicture}[very thick, scale=0.5,color=green!50!black, baseline=-0.35cm]
\draw (0,0.8) -- (0,2);
\draw[-dot-] (0,0.8) .. controls +(0,-0.5) and +(0,-0.5) .. (-0.75,0.8);
\draw[directedgreen, color=green!50!black] (-0.75,0.8) .. controls +(0,0.5) and +(0,0.5) .. (-1.5,0.8);
\draw[-dot-] (0,-1.8) .. controls +(0,0.5) and +(0,0.5) .. (-0.75,-1.8);
\draw[redirectedgreen, color=green!50!black] (-0.75,-1.8) .. controls +(0,-0.5) and +(0,-0.5) .. (-1.5,-1.8);
\draw (0,-1.8) -- (0,-3);
\draw (-1.5,0.8) -- (-1.5,-1.8);
\draw (-0.375,-0.2) node[Odot] (D) {}; 
\draw (-0.375,0.4) -- (D);
\draw (-0.375,-0.8) node[Odot] (E) {}; 
\draw (-0.375,-1.4) -- (E);
\end{tikzpicture}
\, , \quad
\gamma_A^{-1} = 
\begin{tikzpicture}[very thick, scale=0.5,color=green!50!black, baseline=-0.35cm]
\draw (0,0.8) -- (0,2);
\draw[-dot-] (0,0.8) .. controls +(0,-0.5) and +(0,-0.5) .. (0.75,0.8);
\draw[directedgreen, color=green!50!black] (0.75,0.8) .. controls +(0,0.5) and +(0,0.5) .. (1.5,0.8);
\draw[-dot-] (0,-1.8) .. controls +(0,0.5) and +(0,0.5) .. (0.75,-1.8);
\draw[redirectedgreen, color=green!50!black] (0.75,-1.8) .. controls +(0,-0.5) and +(0,-0.5) .. (1.5,-1.8);
\draw (0,-1.8) -- (0,-3);
\draw (1.5,0.8) -- (1.5,-1.8);
\draw (0.375,-0.2) node[Odot] (D) {}; 
\draw (0.375,0.4) -- (D);
\draw (0.375,-0.8) node[Odot] (E) {}; 
\draw (0.375,-1.4) -- (E);
\end{tikzpicture}
\, . 
\ee
It measures how far a Frobenius algebra is away from being symmetric in the sense that~$A$ is symmetric if and only if $\gamma_A = 1_A$, see e.\,g.~\cite{fs0901.4886}. 

For $X \in \Beq( (a,A), (b,B) )$ we write ${}_{\beta}X_{\alpha}$ for the same underlying 1-morphism~$X$ in~$\B$ whose left and right module actions however are `twisted' by pre-composition with algebra automorphisms~$\beta: B\to B$ and~$\alpha: A\to A$.

\begin{proposition}\label{prop:Beq-has-adjoints}
$\Beq$ has adjoints. The left and right adjoints of $X \in \Beq( (a,A), (b,B) )$ are 
$\deqX := {}_{\gamma_A^{-1}} (\dX)$ and $X^\star := (X^\dagger)_{\gamma_B}$, respectively. 
\end{proposition}

\begin{proof}  
The adjunction maps for $\deqX \dashv X$ in $\Beq$ are given by
\be\label{eq:LeftAdjunctionInBeq}
\ev_X = \!
\begin{tikzpicture}[very thick,scale=1.0,color=blue!50!black, baseline=.8cm]
\draw[line width=0pt] 
(1.75,1.75) node[line width=0pt, color=green!50!black] (A) {{\small$A\vphantom{\deqX }$}}
(1,0) node[line width=0pt] (D) {{\small$X\vphantom{\deqX }$}}
(0,0) node[line width=0pt] (s) {{\small$\deqX $}}; 
\draw[directed] (D) .. controls +(0,1.5) and +(0,1.5) .. (s);

\draw[color=green!50!black] (1.25,0.55) .. controls +(0.0,0.25) and +(0.25,-0.15) .. (0.86,0.95);
\draw[-dot-, color=green!50!black] (1.25,0.55) .. controls +(0,-0.5) and +(0,-0.5) .. (1.75,0.55);

\draw[color=green!50!black] (1.5,-0.1) node[Odot] (unit) {}; 
\draw[color=green!50!black] (1.5,0.15) -- (unit);

\fill[color=green!50!black] (0.86,0.95) circle (2pt) node (meet) {};

\draw[color=green!50!black] (1.75,0.55) -- (A);
\end{tikzpicture}
\circ \xi
\, , \quad
\coev_X =  \vartheta \circ 
\begin{tikzpicture}[very thick,scale=1.0,color=blue!50!black, baseline=-.8cm,rotate=180]
\draw[line width=0pt] 
(3.21,1.85) node[line width=0pt, color=green!50!black] (B) {{\small$B\vphantom{\deqX }$}}
(3,0) node[line width=0pt] (D) {{\small$X\vphantom{\deqX }$}}
(2,0) node[line width=0pt] (s) {{\small$\deqX $}}; 
\draw[redirected] (D) .. controls +(0,1.5) and +(0,1.5) .. (s);

\fill[color=green!50!black] (2.91,0.85) circle (2pt) node (meet) {};

\draw[color=green!50!black] (2.91,0.85) .. controls +(0.2,0.25) and +(0,-0.75) .. (B);

\end{tikzpicture}
\ee
where $\xi: \deqX \otimes_B X \rightarrow \deqX  \otimes X$ and $\vartheta: X \otimes \deqX  \rightarrow X \otimes_A \deqX $ are the splitting and projection maps, cf.~Lemma~\ref{lem:tensorprojector}. The case of right adjoints $X^\star = (X^\dagger)_{\gamma_B}$ is analogous. 
The $A$-$B$-bimodule structures on~$X^\star$, $\deqX$ are those of $X^\dagger$, $\dX$, e.\,g.
\be
\begin{tikzpicture}[very thick,scale=0.85,color=blue!50!black, baseline=0cm]
\draw[line width=0] 
(-0.5,-1.25) node[line width=0pt, color=green!50!black] (Algebra) {{\small $A\vphantom{X^\dagger}$}}
(0,1.25) node[line width=0pt] (A) {{\small $X^\dagger$}}
(0,-1.25) node[line width=0pt] (A2) {{\small $X^\dagger$}};
\draw (A) -- (A2);
\draw[color=green!50!black] (0,0) .. controls +(-0.5,-0.25) and +(0,1) .. (Algebra);
\fill[color=green!50!black] (0,0) circle (2.5pt) node[left] {};
\end{tikzpicture}
\!=\!
\begin{tikzpicture}[very thick,scale=0.85,color=blue!50!black, baseline=0cm]
\draw[line width=0] 
(0.5,-1.25) node[line width=0pt, color=green!50!black] (Algebra) {{\small $A\vphantom{X^\dagger}$}}
(-1,1.25) node[line width=0pt] (A) {{\small $X^\dagger$}}
(1,-1.25) node[line width=0pt] (A2) {{\small $X^\dagger$}}; 
\draw[redirected] (0,0) .. controls +(0,-1) and +(0,-1) .. (-1,0);
\draw[redirected] (1,0) .. controls +(0,1) and +(0,1) .. (0,0);
\draw[color=green!50!black] (0,0) .. controls +(0.5,-0.25) and +(0,1) .. (Algebra);
\fill[color=green!50!black] (0,0) circle (2.5pt) node[left] {};
\draw (-1,0) -- (A); 
\draw (1,0) -- (A2); 
\end{tikzpicture}
, 
\quad
\begin{tikzpicture}[very thick,scale=0.85,color=blue!50!black, baseline=0cm]
\draw[line width=0] 
(0.5,-1.25) node[line width=0pt, color=green!50!black] (Algebra) {{\small $B\vphantom{X^\dagger}$}}
(0,1.25) node[line width=0pt] (A) {{\small $X^\dagger$}}
(0,-1.25) node[line width=0pt] (A2) {{\small $X^\dagger$}};
\draw (A) -- (A2);
\draw[color=green!50!black] (0,0) .. controls +(0.5,-0.25) and +(0,1) .. (Algebra);
\fill[color=green!50!black] (0,0) circle (2.5pt) node[left] {};
\end{tikzpicture}
\!=\!
\begin{tikzpicture}[very thick,scale=0.85,color=blue!50!black, baseline=0cm]
\draw[line width=0] 
(1.5,-1.25) node[line width=0pt, color=green!50!black] (Algebra) {{\small $B\vphantom{X^\dagger}$}}
(-1,1.25) node[line width=0pt] (A) {{\small $X^\dagger$}}
(1,-1.25) node[line width=0pt] (A2) {{\small $X^\dagger$}}; 
\draw[redirected] (0,0) .. controls +(0,-1) and +(0,-1) .. (-1,0);
\draw[redirected] (1,0) .. controls +(0,1) and +(0,1) .. (0,0);
\fill[color=green!50!black] (0,0) circle (2.5pt) node[left] {};
\draw (-1,0) -- (A); 
\draw (1,0) -- (A2); 
\draw[color=green!50!black] (0,0) .. controls +(0.15,0.15) and +(0.2,0) .. (-0.25,-0.25);
\draw[color=green!50!black] (-0.25,-0.25) .. controls +(-0.6,0) and +(-0.85,0) .. (0.5,1.25);
\draw[color=green!50!black] (0.5,1.25) .. controls +(0.7,0) and +(0,1) .. (1.5,-0.25);
\draw[<-,color=green!50!black] (-0.21,-0.25) -- (-0.22,-0.25);
\draw[<-,color=green!50!black] (0.41,1.25) -- (0.42,1.25);
\draw[color=green!50!black] (1.5,-0.25) -- (Algebra); 
\end{tikzpicture}
, 
\ee
appropriately twisted by the (inverse) of the Nakayama automorphism. 

As an example we verify that the second Zorro move in~\eqref{eq:LeftZorroMoves} holds in $\Beq$. Written in terms of diagrams in~$\B$ its left-hand side is
\be
\begin{tikzpicture}[very thick,scale=1.2,color=blue!50!black, baseline=0cm]
\draw[line width=0] 
(1,1.25) node[line width=0pt] (A) {{\small $\deqX$}}
(-1,-1.25) node[line width=0pt] (A2) {{\small $\deqX$}}; 
\draw[directed] (0,0) .. controls +(0,1) and +(0,1) .. (-1,0);
\draw[directed] (1,0) .. controls +(0,-1) and +(0,-1) .. (0,0);
\draw (-1,0) -- (A2); 
\draw (1,0) -- (A); 
\draw[-dot-, color=green!50!black] (-0.15,0.2) .. controls +(0,-0.5) and +(0,-0.5) .. (-0.85,0.2);
\draw[color=green!50!black] (-0.15,0.2) .. controls +(0,0.25) and +(-0.0025,-0.025) .. (-0.09,0.45);
\draw[color=green!50!black] (-0.85,0.2) .. controls +(0,0.25) and +(0.0025,-0.025) .. (-0.91,0.45);
\fill[color=green!50!black] (-0.09,0.45) circle (1.5pt) node (meet) {};
\fill[color=green!50!black] (-0.91,0.45) circle (1.5pt) node (meet) {};
\draw[color=green!50!black] (-0.5,-0.45) node[Odot] (down) {}; 
\draw[color=green!50!black] (down) -- (-0.5,-0.15); 
\draw[-dot-, color=green!50!black] (0.9,0.1) .. controls +(0,-0.5) and +(0,-0.5) .. (0.1,0.1);
\draw[color=green!50!black] (0.9,0.1) .. controls +(0.0,0.1) and +(-0.0025,-0.025) .. (1.0,0.25);
\draw[color=green!50!black] (0.1,0.1) .. controls +(0,0.1) and +(0.0025,-0.025) .. (-0.02,0.25);
\fill[color=green!50!black] (1.0,0.25) circle (1.5pt) node (meet) {};
\fill[color=green!50!black] (-0.02,0.25) circle (1.5pt) node (meet) {};
\draw[color=green!50!black] (0.5,-0.55) node[Odot] (down) {}; 
\draw[color=green!50!black] (down) -- (0.5,-0.25); 
\end{tikzpicture}
=
\begin{tikzpicture}[very thick,scale=1.2,color=blue!50!black, baseline=0cm]
\draw[line width=0] 
(1,1.25) node[line width=0pt] (A) {{\small $\dX$}}
(-1,-1.25) node[line width=0pt] (A2) {{\small $\dX$}}; 
\draw[directed] (0,0) .. controls +(0,1) and +(0,1) .. (-1,0);
\draw[directed] (1,0) .. controls +(0,-1) and +(0,-1) .. (0,0);
\draw (-1,0) -- (A2); 
\draw (1,0) -- (A); 
\draw[-dot-, color=green!50!black] (-0.15,0.2) .. controls +(0,-0.5) and +(0,-0.5) .. (-0.85,0.2);
\draw[color=green!50!black] (-0.15,0.2) .. controls +(0,0.25) and +(-0.0025,-0.025) .. (-0.09,0.45);
\draw[color=green!50!black] (-0.85,0.2) .. controls +(0,0.25) and +(0.0025,-0.025) .. (-0.91,0.45);
\fill[color=green!50!black] (-0.09,0.45) circle (1.5pt) node (meet) {};
\fill[color=green!50!black] (-0.91,0.45) circle (1.5pt) node (meet) {};
\draw[color=green!50!black] (-0.5,-0.45) node[Odot] (down) {}; 
\draw[color=green!50!black] (down) -- (-0.5,-0.15); 
\draw[-dot-, color=green!50!black] (0.9,0.1) .. controls +(0,-0.5) and +(0,-0.5) .. (0.1,0.1);
\draw[color=green!50!black] (0.9,0.1) .. controls +(0.0,0.1) and +(-0.0025,-0.025) .. (1.0,0.25);
\draw[color=green!50!black] (0.1,0.1) .. controls +(0,0.1) and +(0.0025,-0.025) .. (-0.02,0.25);
\fill[color=green!50!black] (1.0,0.25) circle (1.5pt) node (meet) {};
\fill[color=green!50!black] (-0.02,0.25) circle (1.5pt) node (meet) {};
\draw[color=green!50!black] (0.5,-0.55) node[Odot] (down) {}; 
\draw[color=green!50!black] (down) -- (0.5,-0.25); 
\fill[color=green!50!black] (0.89,0.05) circle (1.5pt) node[left] (phi) {};
\fill[color=green!50!black] (0.98,0.14) circle (0pt) node[left] (phi) {{\small $\gamma_A^{-1}$}};
\end{tikzpicture}
=
\begin{tikzpicture}[very thick,scale=1.2,color=blue!50!black, baseline=0cm]
\draw[line width=0] 
(1,1.25) node[line width=0pt] (A) {{\small $\dX$}}
(-1,-1.25) node[line width=0pt] (A2) {{\small $\dX$}}; 
\draw[directed] (0,0) .. controls +(0,1) and +(0,1) .. (-1,0);
\draw[directed] (1,0) .. controls +(0,-1) and +(0,-1) .. (0,0);
\draw (-1,0) -- (A2); 
\draw (1,0) -- (A); 
\draw[-dot-, color=green!50!black] (-0.15,-0.1) .. controls +(0,-0.25) and +(0,-0.25) .. (-0.55,-0.1);
\draw[color=green!50!black] (-0.15,-0.1) .. controls +(0,0.25) and +(-0.0025,-0.025) .. (-0.00,0.2);
\draw[color=green!50!black] (-0.55,-0.1) .. controls +(0,0.25) and +(0.0025,-0.025) .. (-0.06,0.4);
\fill[color=green!50!black] (-0.00,0.2) circle (1.5pt) node (meet) {};
\fill[color=green!50!black] (-0.06,0.4) circle (1.5pt) node (meet) {};
\draw[color=green!50!black] (-0.35,-0.6) node[Odot] (down) {}; 
\draw[color=green!50!black] (down) -- (-0.35,-0.3); 
\draw[-dot-, color=green!50!black] (0.5,0.5) .. controls +(0,-0.25) and +(0,-0.25) .. (0.1,0.5);
\draw[directedgreen, color=green!50!black] (0.5,0.5) .. controls +(0,0.25) and +(0,0.25) .. (0.9,0.5);
\draw[
	color=green!50!black, decoration={markings, mark=at position 0.66 with {\arrow{>}}}, postaction={decorate}
	]
	 (0.9,0.5) .. controls +(0,-0.7) and +(0.6,-0.8) .. (0.0,-0.15);
\draw[color=green!50!black] (0.1,0.5) .. controls +(0,0.1) and +(0.0025,-0.025) .. (-0.18,0.6);
\fill[color=green!50!black] (0.0,-0.15) circle (1.5pt) node (meet) {};
\fill[color=green!50!black] (-0.18,0.6) circle (1.5pt) node (meet) {};
\draw[color=green!50!black] (0.3,-0.0) node[Odot] (down) {}; 
\draw[color=green!50!black] (down) -- (0.3,0.3); 
%
\fill[color=green!50!black] (0.5,0.5) circle (1.5pt) node[above left] (phi) {};
\fill[color=green!50!black] (0.69,0.43) circle (0pt) node[above left] (phi) {{\small $\gamma_A^{-1}$}};
\end{tikzpicture}
=
\begin{tikzpicture}[very thick,scale=1.2,color=blue!50!black, baseline=0cm]
\draw[line width=0] 
(0,1.25) node[line width=0pt] (A) {{\small $\deqX$}}
(0,-1.25) node[line width=0pt] (A2) {{\small $\deqX$}}; 
\draw (A2) -- (A); 
\end{tikzpicture} 
\ee
where we used the projection and splitting properties of the maps $\pi_A^{\deqX,X}$, $\pi_A^{X,\deqX}$ in~\eqref{eq:piA}, pivotality~\eqref{eq:pivotal}, the fact that~$A$ is separable Frobenius, and finally the Zorro move for~$\dX$ in~$\B$.  
\end{proof}

In the following theorem we summarise the relevant results from Propositions~\ref{prop:algebraXdaggerX} and~\ref{prop:XXdaggerInverseInBeq} and show in addition that $A= X^\dagger \otimes X$ is symmetric.

\begin{theorem}\label{thm:algebraXdaggerX-summary}
Let $X \in \B(a,b)$ have invertible $\dimr(X)$. Then $A:=X^\dagger \otimes X$ is a symmetric separable Frobenius algebra in $\B(a,a)$ and $X: (a,A) \rightleftarrows (b,I_b) : X^\dagger$ is an adjoint equivalence in $\Beq$.
\end{theorem}

\begin{proof}
Recall from Proposition~\ref{prop:algebraXdaggerX} that the algebra structure on $A$ is given by \eqref{eq:algebra-on-YX} (with $\tilde\varepsilon=\tev_X$ and $\tilde\eta=\tcoev_X$) while the coalgebra structure includes a twisting by the quantum dimension: with 
$\boldsymbol{\bullet} := \dimr(X)$ and 
$\boldsymbol{\star} := \dimr(X)^{-1}$ we have
\be
\begin{tikzpicture}[very thick,scale=0.75,color=green!50!black, baseline=-0.6cm, rotate=180]
\draw[-dot-] (3,0) .. controls +(0,1) and +(0,1) .. (2,0);
\draw (2.5,0.75) -- (2.5,1.5); 
\end{tikzpicture} 
=
\begin{tikzpicture}[very thick,scale=0.9,color=blue!50!black, baseline=-0.9cm, rotate=180]
\draw[line width=0pt] 
(3,0) node[line width=0pt] (D) {{\small$X\vphantom{X^\dagger}$}}
(2,0) node[line width=0pt] (s) {{\small${X^\dagger}$}}; 
\draw[redirected] (D) .. controls +(0,1) and +(0,1) .. (s);
\draw (2.5,1.13) node (D) {{\small$\boldsymbol{\star}$}}; 
\draw[line width=0pt] 
(3.45,0) node[line width=0pt] (re) {{\small${X^\dagger}$}}
(1.55,0) node[line width=0pt] (li) {{\small$X\vphantom{X^\dagger}$}}; 
\draw[line width=0pt] 
(2.7,2) node[line width=0pt] (ore) {{\small${X^\dagger}$}}
(2.3,2) node[line width=0pt] (oli) {{\small$X\vphantom{X^\dagger}$}}; 
\draw (li) .. controls +(0,0.75) and +(0,-0.25) .. (2.3,1.25);
\draw (2.3,1.25) -- (oli);
\draw (re) .. controls +(0,0.75) and +(0,-0.25) .. (2.7,1.25);
\draw (2.7,1.25) -- (ore);
\end{tikzpicture}
\, , \quad 
\begin{tikzpicture}[very thick,scale=0.75,color=green!50!black, baseline, rotate=180]
\draw (0,-0.5) node[Odot] (D) {}; 
\draw (D) -- (0,0.6); 
\end{tikzpicture} 
= 
\begin{tikzpicture}[very thick,scale=1.0,color=blue!50!black, baseline=.4cm]
\draw[line width=0pt] 
(3,0) node[line width=0pt] (D) {{\small$X\vphantom{X^\dagger}$}}
(2,0) node[line width=0pt] (s) {{\small$\dX$}}; 
\draw[directed] (D) .. controls +(0,1) and +(0,1) .. (s);
\draw (2.5,0.4) node (D) {{\small$\boldsymbol{\bullet}$}}; 
\end{tikzpicture}
\ee
It only remains to prove symmetry of $A$. But this is immediate from pivotality:
\be
\begin{tikzpicture}[very thick,scale=0.85,color=green!50!black, baseline=0cm]
\draw[-dot-] (0,0) .. controls +(0,1) and +(0,1) .. (-1,0);
\draw[directedgreen, color=green!50!black] (1,0) .. controls +(0,-1) and +(0,-1) .. (0,0);
\draw (-1,0) -- (-1,-1.5); 
\draw (1,0) -- (1,1.5); 
\draw (-0.5,1.2) node[Odot] (end) {}; 
\draw (-0.5,0.8) -- (end); 
\end{tikzpicture}
= 
\begin{tikzpicture}[very thick,scale=0.65,color=blue!50!black, baseline=0.2cm]
\draw[line width=1pt] 
(-2,-1.5) node[line width=0pt] (X) {}
(-1,-1.5) node[line width=0pt] (Y) {}
(2,2) node[line width=0pt] (XY) {};
\draw (-0.5,1.13) node (D) {{\small$\boldsymbol{\bullet}$}}; 
\draw[redirected] (0,0) .. controls +(0,1) and +(0,1) .. (-1,0);
\draw[directed] (1,0) .. controls +(0,2) and +(0,2) .. (-2,0);
\draw[directed] (2,0) .. controls +(0,-1) and +(0,-1) .. (0.5,0);
\draw (-1,0) -- (Y);
\draw (-2,0) -- (X);
\draw[dotted] (0,0) -- (1,0);
\draw (2,0) -- (XY);
\end{tikzpicture}
\stackrel{\eqref{eq:pivotal}}{=}
\begin{tikzpicture}[very thick,scale=0.65,color=blue!50!black, baseline=0.2cm]
\draw[line width=1pt] 
(2,-1.5) node[line width=0pt] (Y) {}
(3,-1.5) node[line width=0pt] (X) {}
(-1,2) node[line width=0pt] (XY) {};
\draw (1.5,1.13) node (D) {{\small$\boldsymbol{\bullet}$}}; 
\draw[directed] (1,0) .. controls +(0,1) and +(0,1) .. (2,0);
\draw[redirected] (0,0) .. controls +(0,2) and +(0,2) .. (3,0);
\draw[directed] (-1,0) .. controls +(0,-1) and +(0,-1) .. (0.5,0);
\draw (2,0) -- (Y);
\draw (3,0) -- (X);
\draw[dotted] (0,0) -- (1,0);
\draw (-1,0) -- (XY);
\end{tikzpicture}
=
\begin{tikzpicture}[very thick,scale=0.85,color=green!50!black, baseline=0cm]
\draw[redirectedgreen, color=green!50!black] (0,0) .. controls +(0,-1) and +(0,-1) .. (-1,0);
\draw[-dot-] (1,0) .. controls +(0,1) and +(0,1) .. (0,0);
\draw (-1,0) -- (-1,1.5); 
\draw (1,0) -- (1,-1.5); 
\draw (0.5,1.2) node[Odot] (end) {}; 
\draw (0.5,0.8) -- (end); 
\end{tikzpicture}
\, . 
\ee
\end{proof}

\subsection{Nondegenerate pairings}\label{subsec:nondeg_equi}

As in the previous section we let~$\B$ be a pivotal bicategory with idempotent complete 1-morphism categories.
Furthermore we assume that the 1-morphisms categories $\B(a,b)$ are $\C$-linear, and that for all $a\in \B$ we are given a linear map 
\be
\langle - \rangle_a : \End_{\B}(I_a) \lra \C
\ee
called \textsl{bulk correlator}. This allows us to define the \textsl{bulk pairing} $\langle -,- \rangle_a$ via $\langle \phi_1, \phi_2 \rangle_a := \langle \phi_1 \phi_2 \rangle_a$. More generally, for any 1-morphisms $X,Y:a\rightarrow b$ in~$\B$ we can consider the \textsl{defect pairing}
\be
\langle -,- \rangle_X : \Hom(Y,X) \times \Hom(X,Y) \lra \C \, , \quad 
\langle \Psi_1, \Psi_2 \rangle _X = 
\left\langle \, 
\begin{tikzpicture}[very thick,scale=0.6,color=blue!50!black, baseline]
\draw (0,0) circle (1.5);
\draw[<-, very thick] (0.100,-1.5) -- (-0.101,-1.5) node[above] {}; 
\draw[<-, very thick] (-0.100,1.5) -- (0.101,1.5) node[below] {}; 
\fill (-22.5:1.5) circle (3.3pt) node[left] {{\small$\Psi_2$}};
\fill (22.5:1.5) circle (3.3pt) node[left] {{\small$\Psi_1$}};
\fill (270:1.5) circle (0pt) node[above] {{\small$X$}};
\end{tikzpicture} 
\,  \right\rangle_{\raisemath{10pt}{\!\!\!a}} \, . 
\ee
Note that the bulk pairing $\langle -,- \rangle_a$ is the same as $\langle -,- \rangle_{I_a}$. 

If for a pair $X,Y$ the defect pairing in~$\B$ is nondegenerate, then we would like to know if it induces a nondegenerate pairing in $\Beq$. In general this will only be the case if the associated object~$(a,A) \in \Beq$ is contained in a certain full subbicategory $\Borb$ of $\Beq$. We will discuss $\Borb$ in the next section, while in the remainder of the present section we explain how close to nondegeneracy one can get in $\Beq$. 

Let us consider two objects $(a,A)$ and $(b,B)$ and a 1-morphism $X:(a,A) \rightarrow (b,B)$ in $\Beq$. Furthermore we fix Frobenius algebra automorphisms $\alpha \in \operatorname{Aut}(A)$, $\beta\in \operatorname{Aut}(B)$ and define the two operators ${}_\beta P, P_\alpha$ on $\End(X)$ by
\be
{}_\beta P : \Phi \lmt
\begin{tikzpicture}[very thick,scale=0.75,color=blue!50!black, baseline=0cm]

\draw (1.5,-1.5) -- (1.5,1.5); 

\draw (1.5,-1.75) node[right] (X) {};
\draw (1.5,2) node[right] (Xu) {};

\fill (1.5,0.7) circle (2pt) node[right] (phi) {{\small $\Phi$}};

\fill[color=green!50!black] (1.5,0.3) circle (2pt) node (meet) {};
\fill[color=green!50!black] (1.5,1.1) circle (2pt) node (meet) {};
\fill[color=green!50!black] (0.6,0.12) circle (2pt) node[left] (meet) {{\small$\beta$}};

\draw[color=green!50!black] (1,-0.25) .. controls +(0.0,0.25) and +(-0.25,-0.25) .. (1.5,0.3);
\draw[color=green!50!black] (0.5,-0.25) .. controls +(0.0,0.5) and +(-0.25,-0.25) .. (1.5,1.1);

\draw[-dot-, color=green!50!black] (0.5,-0.25) .. controls +(0,-0.5) and +(0,-0.5) .. (1,-0.25);

\draw[color=green!50!black] (0.75,-1.2) node[Odot] (unit) {}; 
\draw[color=green!50!black] (0.75,-0.6) -- (unit);

\end{tikzpicture}
, \quad
P_\alpha : \Phi \lmt
\begin{tikzpicture}[very thick,scale=0.75,color=blue!50!black, baseline=0cm]

\draw (1.5,-1.5) -- (1.5,1.5); 

\draw (1.5,-1.75) node[right] (X) {};
\draw (1.5,2) node[right] (Xu) {};

\fill (1.5,0.7) circle (2pt) node[left] (phi) {{\small $\Phi$}};

\fill[color=green!50!black] (1.5,0.3) circle (2pt) node (meet) {};
\fill[color=green!50!black] (1.5,1.1) circle (2pt) node (meet) {};
\fill[color=green!50!black] (2.4,0.12) circle (2pt) node[right] (meet) {{\small$\alpha$}};

\draw[color=green!50!black] (2,-0.25) .. controls +(0.0,0.25) and +(0.25,-0.25) .. (1.5,0.3);
\draw[color=green!50!black] (2.5,-0.25) .. controls +(0.0,0.5) and +(0.25,-0.25) .. (1.5,1.1);

\draw[-dot-, color=green!50!black] (2.5,-0.25) .. controls +(0,-0.5) and +(0,-0.5) .. (2,-0.25);

\draw[color=green!50!black] (2.25,-1.2) node[Odot] (unit) {}; 
\draw[color=green!50!black] (2.25,-0.6) -- (unit);

\end{tikzpicture}
. 
\ee
Setting $Y = {}_\beta X_\alpha$ in Lemma~\ref{lem:piAproj} we see that $({}_\beta P)^2 = {}_\beta P$ and $(P_\alpha)^2 = P_\alpha$. Note that the special cases ${}_{1_B}P$ and $P_{1_A}$ are precisely the projectors to left and right module maps discussed in Section~\ref{subsec:algebrasbimodules}. Hence we may call elements in the image of ${}_\beta P \circ P_\alpha=P_\alpha \circ {}_\beta P$ \textsl{$\beta$-$\alpha$-twisted $B$-$A$-bimodule maps}. 

The projectors~${}_\beta P$, $P_\alpha$ and the Nakayama automorphism~\eqref{eq:Nakayama} satisfy the following compatibility with defect pairings: 

\begin{proposition}\label{prop:perfectlypairedinBeq}
Let $\B$ and $X:(a,A) \rightarrow (b,B)$ be as above and such that
\be\label{eq:traceorre}
\left\langle \, 
\begin{tikzpicture}[very thick,scale=0.6,color=blue!50!black, baseline]
\draw (0,0) circle (1.5);
\draw[<-, very thick] (0.100,-1.5) -- (-0.101,-1.5) node[above] {}; 
\draw[<-, very thick] (-0.100,1.5) -- (0.101,1.5) node[below] {}; 
\fill (0:1.5) circle (3.3pt) node[left] {{\small$\Psi$}};
\fill (270:1.5) circle (0pt) node[above] {{\small$X$}};
\end{tikzpicture} 
\, \right\rangle_{\raisemath{10pt}{\!\!\!a}}
=
\left\langle \, 
\begin{tikzpicture}[very thick,scale=0.6,color=blue!50!black, baseline]
\draw (0,0) circle (1.5);
\draw[->, very thick] (0.100,-1.5) -- (-0.101,-1.5) node[above] {}; 
\draw[->, very thick] (-0.100,1.5) -- (0.101,1.5) node[below] {}; 
\fill (180:1.5) circle (3.3pt) node[right] {{\small$\Psi$}};
\fill (270:1.5) circle (0pt) node[above] {{\small$X$}};
\end{tikzpicture} 
\, \right\rangle_{\raisemath{10pt}{\!\!\!b}}
\ee
for all $\Psi: X \rightarrow X$ in~$\B$. Then we have
\be\label{eq:projcomppair}
\big\langle {}_\beta P(\Phi_1), \Phi_2 \big\rangle_X = \big\langle \Phi_1, {}_{\beta^{-1} \gamma_B} P(\Phi_2) \big\rangle_X 
\, , \quad
\big\langle P_\alpha(\Phi_1), \Phi_2 \big\rangle_X = \big\langle \Phi_1,  P_{\alpha^{-1} \gamma_A^{-1}}(\Phi_2) \big\rangle_X 
\ee
for all $\Phi_1: Y\rightarrow X$ and $\Phi_2: X\rightarrow Y$. 
\end{proposition}

\begin{proof}
We compute
\be
{}_{\beta^{-1}\gamma_B}P(\Phi_2) = 
\begin{tikzpicture}[very thick,scale=0.75,color=blue!50!black, baseline=0cm]

\draw (1.5,-1.75) -- (1.5,1.5); 

\draw (1.5,-1.75) node[right] (X) {};
\draw (1.5,2) node[right] (Xu) {};

\fill (1.5,0.7) circle (2pt) node[right] (phi) {{\scriptsize $\Phi_2$}};

\fill[color=green!50!black] (1.5,0.3) circle (2pt) node (meet) {};
\fill[color=green!50!black] (1.5,1.1) circle (2pt) node (meet) {};
\fill[color=green!50!black] (1.0,0.62) circle (2pt) node[left] (meet) {{\scriptsize$\beta^{-1}$}};
\fill[color=green!50!black] (0.5,-0.25) circle (2pt) node[left] (meet) {{\scriptsize$\gamma_B$}};

\draw[color=green!50!black] (1,-0.25) .. controls +(0.0,0.25) and +(-0.25,-0.25) .. (1.5,0.3);
\draw[color=green!50!black] (0.5,-0.25) .. controls +(0.0,0.5) and +(-0.25,-0.25) .. (1.5,1.1);

\draw[-dot-, color=green!50!black] (0.5,-0.25) .. controls +(0,-0.5) and +(0,-0.5) .. (1,-0.25);

\draw[color=green!50!black] (0.75,-1.2) node[Odot] (unit) {}; 
\draw[color=green!50!black] (0.75,-0.6) -- (unit);

\end{tikzpicture}
\stackrel{\eqref{eq:Nakayama}}{=}
\begin{tikzpicture}[very thick,scale=0.75,color=blue!50!black, baseline=0cm]

\draw (1.5,-1.75) -- (1.5,1.5); 

\draw (1.5,-1.75) node[right] (X) {};
\draw (1.5,2) node[right] (Xu) {};

\fill (1.5,0.7) circle (2pt) node[right] (phi) {{\scriptsize $\Phi_2$}};

\fill[color=green!50!black] (1.5,0.3) circle (2pt) node (meet) {};
\fill[color=green!50!black] (1.5,1.1) circle (2pt) node (meet) {};
\fill[color=green!50!black] (0.5,0.62) circle (2pt) node[above] (meet) {};
\draw (0.5,0.9) [green!50!black] node {{\scriptsize$\beta^{-1}$}};

\draw[color=green!50!black] (0.5,0.62) .. controls +(0.0,0.25) and +(-0.25,-0.25) .. (1.5,1.1);

\draw[-dot-, color=green!50!black] (0.5,0.62) .. controls +(0,-0.5) and +(0,-0.5) .. (0,0.62);
\draw[color=green!50!black] (0.25,-0.1) node[Odot] (unit) {}; 
\draw[color=green!50!black] (0.25,0.2) -- (unit);
\draw[directedgreen, color=green!50!black] (0,0.62) .. controls +(0,0.5) and +(0,0.5) .. (-0.5,0.62);
\draw[color=green!50!black] (-0.5,0.62) -- (-0.5,-1.2);
\draw[-dot-, color=green!50!black] (0.5,-1.2) .. controls +(0,0.5) and +(0,0.5) .. (0.,-1.2);
\draw[color=green!50!black] (0.25,-0.45) node[Odot] (unit2) {}; 
\draw[color=green!50!black] (0.25,-0.8) -- (unit2);
\draw[redirectedgreen, color=green!50!black] (0,-1.2) .. controls +(0,-0.5) and +(0,-0.5) .. (-0.5,-1.2);
\draw[-dot-, color=green!50!black] (0.5,-1.2) .. controls +(0,-0.5) and +(0,-0.5) .. (1,-1.2);
\draw[color=green!50!black] (0.75,-2) node[Odot] (unit3) {}; 
\draw[color=green!50!black] (0.75,-1.5) -- (unit3);
\draw[color=green!50!black] (1,-1.2) -- (1,-0.25);
\draw[color=green!50!black] (1,-0.25) .. controls +(0.0,0.25) and +(-0.25,-0.25) .. (1.5,0.3);
\end{tikzpicture}
\stackrel{\eqref{eq:Frob}}{=}
\begin{tikzpicture}[very thick,scale=0.75,color=blue!50!black, baseline=0cm]

\draw (1.5,-1.75) -- (1.5,1.5); 

\draw (1.5,-1.75) node[right] (X) {};
\draw (1.5,2) node[right] (Xu) {};

\fill (1.5,0.7) circle (2pt) node[right] (phi) {{\scriptsize $\Phi_2$}};

\fill[color=green!50!black] (1.5,0.3) circle (2pt) node (meet) {};
\fill[color=green!50!black] (1.5,1.1) circle (2pt) node (meet) {};
\fill[color=green!50!black] (0.5,0.62) circle (2pt) node[above] (meet) {};
\draw (0.5,0.9) [green!50!black] node {{\scriptsize$\beta^{-1}$}};

\draw[color=green!50!black] (0.5,0.62) .. controls +(0.0,0.25) and +(-0.25,-0.25) .. (1.5,1.1);

\draw[-dot-, color=green!50!black] (0.5,0.62) .. controls +(0,-0.5) and +(0,-0.5) .. (0,0.62);
\draw[color=green!50!black] (0.25,-0.1) node[Odot] (unit) {}; 
\draw[color=green!50!black] (0.25,0.2) -- (unit);
\draw[directedgreen, color=green!50!black] (0,0.62) .. controls +(0,0.5) and +(0,0.5) .. (-0.5,0.62);
\draw[color=green!50!black] (-0.5,0.62) -- (-0.5,-1.0);
\draw[directedgreen, color=green!50!black] (-0.5,-1.0) .. controls +(0,-0.5) and +(0,-0.5) .. (1,-1.0);
\draw[color=green!50!black] (1,-1.0) -- (1,-0.25);
\draw[color=green!50!black] (1,-0.25) .. controls +(0.0,0.25) and +(-0.25,-0.25) .. (1.5,0.3);
\end{tikzpicture}
=
\begin{tikzpicture}[very thick,scale=0.75,color=blue!50!black, baseline=0cm]

\draw (1.5,-1.75) -- (1.5,1.5); 

\draw (1.5,-1.75) node[right] (X) {};
\draw (1.5,2) node[right] (Xu) {};

\fill (1.5,0.7) circle (2pt) node[right] (phi) {{\scriptsize $\Phi_2$}};

\fill[color=green!50!black] (1.5,0.3) circle (2pt) node (meet) {};
\fill[color=green!50!black] (1.5,1.1) circle (2pt) node (meet) {};
\fill[color=green!50!black] (0,0.62) circle (2pt) node[right] (meet) {};
\draw (0.25,0.62) [green!50!black] node {{\scriptsize$\beta$}};

\draw[color=green!50!black] (0.5,0.62) .. controls +(0.0,0.25) and +(-0.25,-0.25) .. (1.5,1.1);

\draw[-dot-, color=green!50!black] (0.5,0.62) .. controls +(0,-0.5) and +(0,-0.5) .. (0,0.62);
\draw[color=green!50!black] (0.25,-0.1) node[Odot] (unit) {}; 
\draw[color=green!50!black] (0.25,0.2) -- (unit);
\draw[directedgreen, color=green!50!black] (0,0.62) .. controls +(0,0.5) and +(0,0.5) .. (-0.5,0.62);
\draw[color=green!50!black] (-0.5,0.62) -- (-0.5,-1.0);
\draw[directedgreen, color=green!50!black] (-0.5,-1.0) .. controls +(0,-0.5) and +(0,-0.5) .. (1,-1.0);
\draw[color=green!50!black] (1,-1.0) -- (1,-0.25);
\draw[color=green!50!black] (1,-0.25) .. controls +(0.0,0.25) and +(-0.25,-0.25) .. (1.5,0.3);
\end{tikzpicture}
\ee
where in the last step we used that~$\beta$ is an automorphism of Frobenius algebras. On the other hand, employing cyclicity of the trace and pivotality of~$\B$ we have 
\be
\left\langle
\begin{tikzpicture}[very thick,scale=0.75,color=blue!50!black, baseline]
\draw (0,0) circle (1.5);
\draw[<-, very thick] (0.100,-1.5) -- (-0.101,-1.5) node[above] {}; 
\draw[<-, very thick] (-0.100,1.5) -- (0.101,1.5) node[below] {}; 
\fill (65:1.5) circle (2.5pt) node[above] {{\scriptsize$\Phi_1$}};
\fill (35:1.5) circle (2.5pt) node[right] {{\scriptsize$\Phi_2$}};
\fill (270:1.5) circle (0pt) node[above] {};

\fill[color=green!50!black] (0.3,1.1) circle (0pt) node[left] (meet) {{\scriptsize$\beta$}};
\fill[color=green!50!black] (0.21,1.17) circle (2.5pt) node (meet1) {};

\fill[color=green!50!black] (50:1.5) circle (2.5pt) node (meet1) {};
\fill[color=green!50!black] (80:1.5) circle (2.5pt) node (meet2) {};

\draw[-dot-, color=green!50!black] (50:1.5) .. controls +(-0.5,-0.5) and +(-0.25,-0.25) .. (80:1.5);

\draw[color=green!50!black] (0.418,0.65) node[Odot] (unit) {}; 
\draw[color=green!50!black] (0.418,1.0) -- (unit);

\end{tikzpicture} 
\right\rangle_{\raisemath{15pt}{\!\!\!a}}
=
\left\langle
\begin{tikzpicture}[very thick,scale=0.75,color=blue!50!black, baseline]
\draw (0,0) circle (1.5);
\draw[<-, very thick] (0.100,-1.5) -- (-0.101,-1.5) node[above] {}; 
\draw[<-, very thick] (-0.100,1.5) -- (0.101,1.5) node[below] {}; 
\fill (65:1.5) circle (2.5pt) node[above] {{\scriptsize$\Phi_1$}};
\fill (35:1.5) circle (2.5pt) node[right] {{\scriptsize$\Phi_2$}};
\fill (270:1.5) circle (0pt) node[above] {};

\fill[color=green!50!black] (50:1.5) circle (2.5pt) node (meet1) {};
\fill[color=green!50!black] (20:1.5) circle (2.5pt) node (meet2) {};

\draw[-dot-, color=green!50!black] (0.7,0.5) .. controls +(0,-0.5) and +(0,-0.5) .. (0.2,0.5);
\draw[color=green!50!black] (0.7,0.5) .. controls +(0,0.2) and +(-0.2,-0.2) .. (50:1.5);

\fill[color=green!50!black] (0.11,0.5) circle (0pt) node[right] (meet) {{\scriptsize$\beta$}};
\fill[color=green!50!black] (0.2,0.5) circle (2.5pt) node (meet1) {};

\draw[directedgreen, color=green!50!black] (0.2,0.5) .. controls +(0,0.5) and +(0,0.5) .. (-0.3,0.5);
\draw[color=green!50!black] (-0.3,0.5) -- (-0.3,-0.8);
\draw[redirectedgreen, color=green!50!black] (0.2,-0.8) .. controls +(0,-0.5) and +(0,-0.5) .. (-0.3,-0.8);

\draw[color=green!50!black] (0.2,-0.8) .. controls +(0,0.2) and +(-0.1,-0.5) .. (20:1.5);

\draw[color=green!50!black] (0.447,-0.2) node[Odot] (unit) {}; 
\draw[color=green!50!black] (0.447,0.1) -- (unit);

\end{tikzpicture} 
\right\rangle_{\raisemath{15pt}{\!\!\!a}} \, .
\ee
This proves the first identity in~\eqref{eq:projcomppair}. With the help of~\eqref{eq:traceorre} the second identity follows analogously, basically by reflecting all diagrams above along the vertical $XYX$-line (which requires the assumption~\eqref{eq:traceorre} in the initial step). 
\end{proof}

We recall that projectors compatible with nondegenerate pairings lead to such pairings on the image: 

\begin{lemma}\label{lem:ProjCompPair}
Let $\langle-,-\rangle$ be a nondegenerate pairing of two vector spaces $U,V$, and let $P\in \End(U)$, $Q\in \End(V)$ be idempotents such that $\langle Pu, v \rangle = \langle u, Qv \rangle$ for all $u \in U$ and $v\in V$. Then the induced pairing of $P(U)$ and $Q(V)$ is nondegenerate. 
\end{lemma}

\begin{proof}
Let $\widetilde u \in P(U)$ be such that $\langle \widetilde u, \widetilde v \rangle = 0$ for all $\widetilde v \in Q(V)$. It follows that $0 = \langle \widetilde u, Qv \rangle = \langle P \widetilde u, v \rangle = \langle \widetilde u, v \rangle$ for all $v\in V$, and hence~$\widetilde u$ must be zero. 
\end{proof}

In our situation this means that if $\langle-,-\rangle_X$ is nondegenerate in~$\B$, then also the subspaces of $\beta$-$\alpha$-twisted and $(\beta^{-1}\gamma_B)$-$(\alpha^{-1}\gamma_A^{-1})$-twisted $B$-$A$-bimodule maps are perfectly paired. Setting $\alpha=1_A$, $\beta=1_B$ we find that $\langle-,-\rangle_X$ is nondegenerate also in $\Beq$ if the Nakayama automorphisms $\gamma_A, \gamma_B$ are identities. This is the case if and only if $A,B$ are both symmetric.

\section{Orbifold bicategory}\label{sec:orbibicat}

In the previous section we saw how far we can take our orbifold construction without asking the defect algebras involved to be symmetric. Symmetry is required for the orbifold construction of Section~\ref{sec:TFT-orbifold} and as we just saw it implies that nondegeneracy is preserved. As an application this will later allow us to prove that all generalised orbifolds of Landau-Ginzburg models give rise to open/closed TFTs. 

\subsection[Definition and properties of $\Borb$]{Definition and properties of $\boldsymbol{\Borb}$}

In the following we assume that~$\B$ is a pivotal bicategory whose categories of 1-morphisms are idempotent complete, so we can consider its equivariant completion $\Beq$. 

\begin{definition}
The \textsl{orbifold completion} $\Borb$ of~$\B$ is the full subbicategory of $\Beq$ whose objects are pairs $(a,A)$ with $a\in \B$ and $A\in \B(a,a)$ a symmetric separable Frobenius algebra. We refer to objects in $\Borb$ as \textsl{(generalised) orbifolds}. 
\end{definition}

The left and right quantum dimensions of the unit 1-morphisms are equal to the identity 2-morphism in any pivotal bicategory. It follows (e.\,g.~by the argument in the proof of Proposition~\ref{prop:AGisgoodalgebra} below) that $I_a$ is symmetric for all $a\in \B$. We thus have
\be
\B \subset \Borb \subset \Beq \, . 
\ee

In the previous section we saw that the equivariant completion $\Beq$ has adjoints $\deqX, X^\star$ which are obtained from the original adjoints $\dX, X^\dagger$ in~$\B$ by twisting them with Nakayama automorphisms. But since the latter are the identity in the case of symmetric Frobenius algebras it follows that the orbifold completion $\Borb$ has the same adjoints $\dX, X^\dagger$ as~$\B$ (while the adjunction maps are of course still different, as in~\eqref{eq:LeftAdjunctionInBeq}). Furthermore, pivotality of $\Borb$ now follows from pivotality of~$\B$. 

The results of Sections \ref{subsec:Beq}--\ref{subsec:Beq-pivot} 
also hold for the orbifold completion: First of all, the same proof as that of Proposition~\ref{prop:eqeqeq} shows that
\be\label{eq:orborborb}
(\Borb)_{\mathrm{orb}} \cong \Borb \, .
\ee
Let  $X:a \rightarrow b$ in~$\B$ be a 1-morphism with $\dimr(X)$ invertible. Since~$\B$ is pivotal, Theorem~\ref{thm:algebraXdaggerX-summary} tells us that $X^\dagger \otimes X$ is a symmetric separable Frobenius algebra, i.\,e.~$(a, X^\dagger \otimes X)$ lies in $\Borb$. And since $\Borb$ is a full subbicategory, again by Theorem~\ref{thm:algebraXdaggerX-summary} we have
\be
(a, X^\dagger \otimes X) 
	\cong 
 (b, I_b) 
\ee
in $\Borb$. 

\begin{remark}
In Section~\ref{sec:2dTFTwithDefects} we discussed the functorial definition of TFTs with defects. In contrast to the well-known case of open/closed TFT (which will be reviewed in Section~\ref{subsec:ocTFT}), a purely algebraic generators-and-relations description of TFTs with arbitrary defects has not yet been found. At the very least such a description would involve a pivotal bicategory~$\B$. Furthermore~$\B$ is expected to be monoidal (with the trivial theory corresponding to the unit object~$0$, and boundary conditions to 1-morphisms with source~$0$), suitably dualisable, and subject to additional constraints like the Cardy condition. 
\end{remark}

\subsection{Nondegenerate pairings}\label{subsec:nondeg}

Let us again assume that $\B(a,b)$ is $\C$-linear and that we have bulk correlators and defect pairings
\be
\langle-\rangle_a : \End_{\B}(I_a) \lra \C
\, , \quad 
\langle \Psi_1, \Psi_2 \rangle _X = 
\left\langle \, 
\begin{tikzpicture}[very thick,scale=0.6,color=blue!50!black, baseline]
\draw (0,0) circle (1.5);
\draw[<-, very thick] (0.100,-1.5) -- (-0.101,-1.5) node[above] {}; 
\draw[<-, very thick] (-0.100,1.5) -- (0.101,1.5) node[below] {}; 
\fill (-22.5:1.5) circle (3.3pt) node[left] {{\small$\Psi_2$}};
\fill (22.5:1.5) circle (3.3pt) node[left] {{\small$\Psi_1$}};
\fill (270:1.5) circle (0pt) node[above] {{\small$X$}};
\end{tikzpicture} 
\,  \right\rangle_{\raisemath{10pt}{\!\!\!a}} \, . 
\ee
as in Section~\ref{subsec:nondeg_equi}. From the discussion there it follows that under the right circumstances nondegenerate defect pairings in~$\B$ induce nondegenerate pairings in $\Borb$: 

\begin{corollary}\label{prop:traceflipgivesnondeg}
Let~$\B$ be as above and $X:(a,A) \rightarrow (b,B)$ in $\Borb$ and such that
\be\label{eq:traceorientationreversal}
\left\langle \, 
\begin{tikzpicture}[very thick,scale=0.6,color=blue!50!black, baseline]
\draw (0,0) circle (1.5);
\draw[<-, very thick] (0.100,-1.5) -- (-0.101,-1.5) node[above] {}; 
\draw[<-, very thick] (-0.100,1.5) -- (0.101,1.5) node[below] {}; 
\fill (0:1.5) circle (3.3pt) node[left] {{\small$\Psi$}};
\fill (270:1.5) circle (0pt) node[above] {{\small$X$}};
\end{tikzpicture} 
\, \right\rangle_{\raisemath{10pt}{\!\!\!a}}
=
\left\langle \, 
\begin{tikzpicture}[very thick,scale=0.6,color=blue!50!black, baseline]
\draw (0,0) circle (1.5);
\draw[->, very thick] (0.100,-1.5) -- (-0.101,-1.5) node[above] {}; 
\draw[->, very thick] (-0.100,1.5) -- (0.101,1.5) node[below] {}; 
\fill (180:1.5) circle (3.3pt) node[right] {{\small$\Psi$}};
\fill (270:1.5) circle (0pt) node[above] {{\small$X$}};
\end{tikzpicture} 
\, \right\rangle_{\raisemath{10pt}{\!\!\!b}}
\ee
for all $\Psi: X \rightarrow X$ in~$\B$. Then if the pairing $\langle -,- \rangle_X$ is nondegenerate in~$\B$, it restricts to a nondegenerate pairing in $\Borb$. 
\end{corollary}

\begin{proof}
Since $A,B$ are symmetric, the Nakayama automorphisms $\gamma_A$ and $\gamma_B$ are identities. Choosing $\alpha=1_A$, $\beta=1_B$ in Proposition~\ref{prop:perfectlypairedinBeq} shows that the conditions of Lemma~\ref{lem:ProjCompPair} hold.
\end{proof}

\section{Landau-Ginzburg models}\label{sec:LGmodels}

From now on we focus on the bicategory $\LG$ of Landau-Ginzburg models. In this section we start by recalling its definition and how all of the assumptions for the general construction of the previous section are satisfied. After an observation on the relation between central charges and invertible quantum dimensions, we review how every object in $\LG$ gives rise to an open/closed TFT, and we prove an analogous result for $\LGorb$. 

\subsection{Bicategory of Landau-Ginzburg models} 

By a Landau-Ginzburg model in this paper we mean the topological B-twist of an $\mathcal N = (2,2)$ supersymmetric Landau-Ginzburg model with affine target $k^n$, see \cite{v1991, ll9112051}, where we can take $k=\C$ or $k=\C[t_1,\ldots, t_d]$.\footnote{Note that as in \cite{cm1208.1481} all of our results hold for arbitrary commutative noetherian $\Q$-algebras~$k$.} The bulk sector of such a theory is described by a \textsl{potential}, i.\,e.~a polynomial~$W$ in the ring $R=k[x_1,\ldots,x_n]$ such that the Jacobi ring $\Jac(W) = R/(\partial_{x_1}W, \ldots, \partial_{x_n}W)$ is a finitely generated free $k$-module and the $\partial_{x_i}W$ form a regular sequence in~$R$. In the case $k=\C$ this simply means that $\dim_\C \Jac(W) < \infty$. 

We want to define the bicategory $\LG$ of Landau-Ginzburg models. From the above it is natural that its objects are given by potentials~$W$ in $R = k[x_1,\ldots,x_n]$ for all $n\in\N$. If we wish to stress which ring~$W$ is an element of, we will denote the associated object in $\LG$ as $(R,W)$. We will often abbreviate $k[x_1,\ldots,x_n]$ as $k[x]$, and our notation will follow \cite{cm1208.1481}. 

Similar to boundary conditions \cite{kl0210, bhls0305, hl0404}, defects in Landau-Ginzburg models are described by matrix factorisations \cite{br0707.0922}; these form the 1-morphisms in $\LG$. Recall that a \textsl{matrix factorisation} of $W\in R$ is a $\Z_2$-graded free $R$-module $X=X^0\oplus X^1$ together with an odd $R$-linear endomorphism~$d_X$, called \textsl{differential}, such that $d_X^2 = W \cdot 1_X$. A \textsl{morphism} between matrix factorisations $(X,d_X)$ and $(Y,d_Y)$ is an even $R$-linear map $\phi: X\rightarrow Y$ that is compatible with the differentials $d_X, d_Y$, i.\,e.~$d_Y \phi = \phi d_X$. Two morphisms $\phi, \psi: X\rightarrow Y$ are \textsl{homotopic} if there is an odd $R$-linear map $\lambda: X\rightarrow Y$ such that $d_Y \lambda + \lambda d_X = \psi - \phi$. 

Matrix factorisations of $W\in R$ are the objects in a category $\HMF(R,W)$, whose arrows are morphisms modulo homotopy relations. The full subcategory in $\HMF(R,W)$ of matrix factorisations whose underlying $R$-modules are of finite rank is denoted $\hmf(R,W)$. Both categories are naturally triangulated \cite{o0302304}, and we write~$[1]$ for their shift functors. For most practical purposes the categories of 1-morphisms in $\LG$ are given by the categories $\hmf(S\otimes_k R, V-W)$; the precise definition below is motivated by the tensor product of matrix factorisations which we discuss next. 

Let $W\in R$, $V\in S$, $U\in T$ be potentials and consider matrix factorisations $X\in \hmf(S\otimes_k R, V-W)$, $Y\in \hmf(T\otimes_k S, U-V)$. From this we define the \textsl{tensor product matrix factorisation} $Y\otimes X \in \HMF(T\otimes_k R, U-W)$ in terms of its underlying $(T\otimes_k R)$-module 
\be
Y\otimes X = \left( (Y^0\otimes_S X^0) \oplus (Y^1 \otimes_S X^1) \right) \oplus \left( (Y^0\otimes_S X^1) \oplus (Y^1 \otimes_S X^0) \right) 
\ee
and differential $d_{Y\otimes X} = d_Y\otimes 1 + 1\otimes d_X$. Whenever $S\neq k$ and $X,Y \neq 0$ this is an infinite-rank matrix factorisation over $T\otimes_k R$. However, as explained in~\cite[Sect.\,12]{dm1102.2957}, $Y\otimes X$ is (naturally isomorphic to) a direct summand of some finite-rank matrix factorisation in $\hmf(T\otimes_k R, U-W)$. 

While the homotopy category of finite-rank matrix factorisations is not necessarily idempotent complete (for an example see \cite[Ex.\,A.5]{keller}), this is the case for $\HMF(R,W)$ (being triangulated and having arbitrary coproducts, see \cite[Prop.\,1.6.8]{NeemanBook}). So to make sure that our categories are closed with respect to the tensor product we are lead to consider the idempotent closure $\hmf(R,W)^\omega$ of $\hmf(R,W)$ in $\HMF(R,W)$. This means that $\hmf(R,W)^\omega$ is the full subcategory of $\HMF(R,W)$ whose objects are (isomorphic to) direct summands of finite-rank matrix factorisations in $\HMF(R,W)$. Accordingly we define 1- and 2-morphisms in $\LG$ via
\be
\LG \big( (R,W), (S,V) \big) = \hmf(S\otimes_k R, V-W)^\omega \, ,
\ee
and the composition~\eqref{eq:cabc} of 1-morphisms is $\LG$ is given by the tensor product. 

To make $\LG$ into an honest bicategory it remains to specify the associator, the unit 1-morphism, and its left and right actions. The former is the obvious natural isomorphism $\alpha_{XYZ}: (X\otimes Y) \otimes Z \rightarrow X \otimes (Y\otimes Z)$ which we will leave implicit in the following. To discuss the unit 1-morphism $I_W$ for a potential $W\in R = k[x_1,\ldots,x_n]$, let us write $\Re = R\otimes_k R = k[x,x']$ and $\widetilde W = W\otimes 1 - 1\otimes W \in \Re$. We also fix~$n$ formal symbols~$\theta_i$ as a basis of $(\Re)^{\oplus n}$. Then the $\Re$-module underlying $I_W\in \hmf(\Re,\widetilde W)$ is the exterior algebra
\be
I_W = \bigwedge \Big( \bigoplus_{i=1}^n \Re \theta_i \Big)
\ee
on which the differential is given by
\be
d_{I_W} = \sum_{i=1}^n \left( (x_i - x'_i) \cdot \theta_i^* + \partial_{[i]} W \cdot \theta_i \wedge (-) \right)
\ee 
where $\partial_{[i]} W = (W(x'_1,\ldots,x'_{i-1}, x_i, \ldots, x_n) - W(x'_1,\ldots,x'_{i}, x_{i+1}, \ldots, x_n))/(x_i - x'_i)$ and~$\theta_i^*$ are contraction operators. We will sometimes simply write~$I$ for~$I_W$ when there is no danger of confusion. It is straightforward to verify that the endomorphisms of $I_W$ in $\hmf(\Re,\widetilde W)$ are given by the bulk space $\Jac(W)$, while $\Hom(I_W,I_W[1])=0$, see e.\,g.~\cite[Sect.\,3]{kr0405232}. 

Finally, the left and right actions of the unit matrix factorisation on $X\in \hmf(S\otimes_k R, V-W)$ are the natural maps
\be
\lambda_X : I_V \otimes X \lra X \, , \quad
\rho_X : X\otimes I_W \lra X
\ee
which are the composition of first projecting~$I$ to its $\theta$-degree zero component and then using the multiplication in the rings~$S$ and~$R$, respectively. While~$\alpha$ is an isomorphism of free modules, $\lambda$ and~$\rho$ are only invertible up to homotopy. Explicit expressions for these homotopy inverses will not be needed in the present paper, but they can be found in \cite[Sect.\,4]{cm1208.1481}. 

In summary, we have specified all the data of the bicategory $\LG$ of Landau-Ginzburg models. That the coherence axioms are indeed satisfied was checked in \cite{McNameethesis, cr0909.4381}. 

\medskip

Now we turn to adjunctions on the level of 1-morphisms in $\LG$. These were first constructed in \cite{cr1006.5609} in the one-variable case and then in a unified way for all of $\LG$ in \cite{cm1208.1481}; see also \cite{brs0909.0696} for earlier work. The evaluation and coevaluation maps for any 1-morphism are explicit expressions in terms of associative Atiyah classes and residues. In an effort to keep the presentation in the present paper compact and clear, we refrain from writing out the adjunction maps and explaining their constituents. For our purposes it will be enough to know that they exist and satisfy the properties we review below. For all further details we refer to \cite{cm1208.1481}. 

Let $W\in R = k[x_1,\ldots,x_n]$ and $V\in S = k[z_1,\ldots,z_m]$ be potentials. The left and right adjoints of $X\in \hmf(S\otimes_k R, V-W)$ in $\LG$ turn out to be the matrix factorisations 
\be\label{eq:leftrightadj}
\dX  = X^\vee \otimes_S S[m] \, , \quad
X^\dagger = R[n] \otimes_R X^\vee 
\ee
in $\hmf(R\otimes_k S, W-V)$. Here $X^\vee = \Hom_{S\otimes_k R}(X,S\otimes_k R)$ is the dual factorisation with differential given by $d_{X^\vee}(\nu) = (-1)^{|\nu| + 1} \nu \circ d_X$ for homogeneous $\nu\in X^\vee$. Similarly, on the level of 2-morphisms $\phi: X\rightarrow Y$ we have
\be\label{eq:adjoint-2-morphisms}
\dphi = \phi^\vee \otimes_S 1_{S[m]} : \dY \lra \dX  \, , \quad 
\phi^\dagger = 1_{R[n]}\otimes_R \phi^\vee : Y^\dagger \lra X^\dagger \, . 
\ee
We summarise the results of \cite[Sect.\,5.3\,\&\,8]{cm1208.1481} relevant for us, using the diagrammatic notation of Section~\ref{subsec:bicatadj}: 

\begin{theorem}\label{thm:LGbicatsummary}
\begin{enumerate}
\item
For any 1-morphism~$X$ in $\LG$ there are canonical adjunctions $\dX \dashv X \dashv X^\dagger$, so in particular $\LG$ has adjoints. 
Moreover, the induced morphisms $\Hom(\dY, \dX) \leftarrow \Hom(X, Y ) \rightarrow Hom(X^\dagger, Y^\dagger)$ are those from~\eqref{eq:adjoint-2-morphisms}: 
\be
\begin{tikzpicture}[very thick,scale=0.75,color=blue!50!black, baseline=.4cm]
\fill (3,0.5) circle (2.5pt) node[right] {{\small $\dphi$}};
\draw (3,0) -- (3,1.5)
node[above] {{{\small$\dX$}}};
\draw (3,0) -- (3,-0.5)
node[below] {{{\small$\dY$}}};
\end{tikzpicture} 
=
\begin{tikzpicture}[very thick,scale=0.75,color=blue!50!black, baseline=.4cm]
\draw[line width=0pt] 
(3,0.5) node[line width=0pt] (D) {}
(2,0.5) node[line width=0pt] (s) {}; 
\draw[directed] (D) .. controls +(0,-1) and +(0,-1) .. (s);
\fill (2,0.5) circle (2.5pt) node[right] {{\small $\phi$}};
\draw (2,0.32) -- (2,0.68);
\draw[line width=0pt] 
(2,0.5) node[line width=0pt] (D) {}
(1,0.5) node[line width=0pt] (s) {}; 
\draw[directed] (D) .. controls +(0,+1) and +(0,+1) .. (s);
\draw (3,0.32) -- (3,1.5)
node[above] {{{\small$\dX$}}};
\draw (1,0.68) -- (1,-0.5)
node[below] {{{\small$\dY$}}};
\end{tikzpicture} 
, \quad 
\begin{tikzpicture}[very thick,scale=0.75,color=blue!50!black, baseline=.4cm]
\fill (3,0.5) circle (2.5pt) node[right] {{\small $\phi^\dagger$}};
\draw (3,0) -- (3,1.5)
node[above] {{{\small$X^\dagger$}}};
\draw (3,0) -- (3,-0.5)
node[below] {{{\small$Y^\dagger$}}};
\end{tikzpicture} 
=
\begin{tikzpicture}[very thick,scale=0.75,color=blue!50!black, baseline=.4cm]
\draw[line width=0pt] 
(3,0.5) node[line width=0pt] (D) {}
(2,0.5) node[line width=0pt] (s) {}; 
\draw[redirected] (D) .. controls +(0,1) and +(0,1) .. (s);
\fill (2,0.5) circle (2.5pt) node[left] {{\small $\phi$}};
\draw (2,0.32) -- (2,0.68);
\draw[line width=0pt] 
(2,0.5) node[line width=0pt] (D) {}
(1,0.5) node[line width=0pt] (s) {}; 
\draw[redirected] (D) .. controls +(0,-1) and +(0,-1) .. (s);
\draw (1,0.32) -- (1,1.5)
node[above] {{{\small$X^\dagger$}}};
\draw (3,0.68) -- (3,-0.5)
node[below] {{{\small$Y^\dagger$}}};
\end{tikzpicture} 
.
\ee
\item 
For $\Phi: X \rightarrow X$ in $\hmf(k[z,x], V-W)$, $\phi \in \End(I_V)$ and $\psi \in \End(I_W)$ we have
\begin{align}
\label{eq:DlMF}
\mathcal D_{\mathrm{l}}^\Phi(X)(\phi) & := 
\begin{tikzpicture}[very thick,scale=0.45,color=blue!50!black, baseline]
\draw (0,0) circle (2);
\draw[<-, very thick] (0.100,-2) -- (-0.101,-2) node[above] {}; 
\draw[<-, very thick] (-0.100,2) -- (0.101,2) node[below] {}; 
\fill (180:0) circle (3.9pt) node[left] {{\small$\phi$}};
\fill (0:2) circle (3.9pt) node[left] {{\small$\Phi$}};
\fill (270:2) circle (0pt) node[above] {{\small$X$}};
\end{tikzpicture} 
 = (-1)^{n+1\choose 2} \Res \left[ \frac{\phi(z) \str( \Phi \, \Lambda_X) \,  \underline{\operatorname{d}\! z}}{\partial_{z_1}V, \ldots, \partial_{z_m}V} \right] , \\
\mathcal D_{\mathrm{r}}^\Phi(X)(\psi) & := 
\begin{tikzpicture}[very thick,scale=0.45,color=blue!50!black, baseline]
\draw (0,0) circle (2);
\draw[->, very thick] (0.100,-2) -- (-0.101,-2) node[above] {}; 
\draw[->, very thick] (0.100,2) -- (0.101,2) node[below] {}; 
\fill (180:0) circle (3.9pt) node[right] {{\small$\psi$}};
\fill (180:2) circle (3.9pt) node[right] {{\small$\Phi$}};
\fill (270:2) circle (0pt) node[above] {{\small$X$}};
\end{tikzpicture} 
 = (-1)^{m+1\choose 2} \Res \left[ \frac{\psi(z) \str( \Phi \, \Lambda_X) \,  \underline{\operatorname{d}\! x}}{\partial_{x_1}W, \ldots, \partial_{x_n}W} \right] \label{eq:rightdecorateddefectaction}
\end{align}
if $m,n \in 2\Z$, where $\Lambda_X = \partial_{x_1}d_X \ldots \partial_{x_n} d_X \, \partial_{z_1}d_X \ldots \partial_{z_m} d_X$. 
\item Write $\mathcal D_h(X) := \mathcal D_h^{1_X}(X)$ for $h\in \{ \mathrm{l}, \mathrm{r} \}$. Then we have
\begin{align}
& \mathcal D_{\mathrm{l}}(I) = 1 = \mathcal D_{\mathrm{r}}(I) \, , \quad
\mathcal D_{\mathrm{l}}(X) = \mathcal D_{\mathrm{r}}(X^\vee) \, , \quad
\mathcal D_{\mathrm{r}}(X) = \mathcal D_{\mathrm{l}}(X^\vee) \, , \\ 
& \mathcal D_{\mathrm{l}}(X) \circ \mathcal D_{\mathrm{l}}(Y) = \mathcal D_{\mathrm{l}}(Y\otimes X) \, , \quad
\mathcal D_{\mathrm{r}}(X) \circ \mathcal D_{\mathrm{r}}(Y) = \mathcal D_{\mathrm{r}}(X\otimes Y) \label{eq:doubledefectwraps}
\end{align}
for all composable 1-morphisms $X,Y$ in $\LG$. 
\end{enumerate}
\end{theorem}

Note that part~(ii) of the above theorem in particular provides us with an explicitly computable expression for quantum dimensions.\footnote{If~$m$ and~$n$ are odd then we only have `pivotality up to shifts' as explained in \cite[Sect.\,7]{cm1208.1481}. In this case one can still define close pendants of quantum dimensions which up to signs (originating from the shifts) are still given by the expressions~\eqref{eq:DlMF} and~\eqref{eq:rightdecorateddefectaction}.\label{ftn:qdims-pivotal-up-to-shifts}}
For a quick review on how to compute with residues we refer to \cite[Sect.\,2.5]{cm1208.1481} and to Section~\ref{subsec:minimalmodels} below; general residue theory is developed in \cite{Lipman}. 

\medskip

It is clear from their definition in~\eqref{eq:leftrightadj} that for $m=n$ mod~2 the left and right adjoints of $X \in \hmf(k[z_1,\ldots,z_m,x_1,\ldots,x_n], V-W)$ are isomorphic and that $X\cong X^{\dagger\dagger}$. However for $m\neq n$ mod~2 we have $\dX \not\cong X^\dagger$ and $X\not\cong X^{\dagger\dagger}$ for most~$X$, such that in general it does not make sense to talk about pivotality or quantum dimensions in the standard sense in $\LG$. Both of these issues have a natural resolution (see \cite[Sect.\,7\,\&\,8]{cm1208.1481}), but at the price of a more sophisticated discussion of shifts and their compatibility with the adjunction maps. 
For example, if $m\neq n$ mod~2 then $D_l^\Phi(X)$ is a map from $\End(I_V)$ to $\Hom(I_W,I_W[1])$, and since the latter space is zero we have $D_l^\Phi(X) = 0$ and similarly $D_r^\Phi(X)=0$. 
To keep our presentation simple and uncluttered we only refer to \cite{cm1208.1481} for a detailed discussion of these points, but we make the following two remarks. 

Firstly, the full subbicategory $\LG'$ whose objects $(R,W)$ depend on an even number of variables is pivotal in the standard sense. Thus all the constructions of Sections~\ref{sec:equibicat} and~\ref{sec:orbibicat} are directly applicable to $\LG'$. Secondly, however, we stress that the restriction on the number of variables to be even can be lifted as the full $\LG$ is ``pivotal up to shifts'', compare \cite[Prop.\,7.1]{cm1208.1481} with~\eqref{eq:pivotal}. Indeed, it is not necessarily pivotality but rather its implications such as the relations~\eqref{eq:doubledefectwraps} that are relevant for our construction. Since these identities hold in all of $\LG$ there is no need for restrictions.  

\medskip

Recall that by Theorem~\ref{thm:algebraXdaggerX-summary} having 1-morphisms $X:a\rightarrow b$ with invertible quantum dimension leads to equivalences in the orbifold completion. If the original bicategory~$\B$ has a ``trivial object''~$0$ then this  implies that the associated category of ``boundary conditions'' $\B(0,b)$ is equivalent to modules over $X^\dagger \otimes X \in \B(a,a)$ in $\B(0,a)$. This is in particular the case for Landau-Ginzburg models: 

\begin{proposition}\label{prop:hmfSVeqmodXX}
Let $X: (R,W) \rightarrow (S,V)$ with invertible $\dimr(X)$ in $\LG$. Then $m = n$ mod~2, $\dX \cong X^\dagger$ and
\be
X^\dagger \otimes (-) : 
\hmf( S, V )^\omega \cong \modu( X^\dagger \otimes X ) 
: X \otimes (-) \, . 
\ee
\end{proposition}

\begin{proof}
If $m \neq n$ mod~2 then the quantum dimensions of $X$ are zero. For $m = n$ mod~2, by definition $\dX \cong X^\dagger$.
By Theorem~\ref{thm:algebraXdaggerX-summary} we have that in $\LGorb$ 
\be
X: ( W, X^\dagger \otimes X ) \lra ( V, I_V)
\ee
is a 1-isomorphism, and hence 
\begin{align}
\hmf(S,V)^\omega & 
= \LG(0, V) 
= \LGorb\big( (0,I_0) , (V,I_V) \big) 
\cong \LGorb\big( (0,I_0), ( W, X^\dagger \otimes X ) \big) 
\nonumber \\
& = \modu(X^\dagger \otimes X) \, . 
\end{align}
\end{proof}

\subsection{Graded matrix factorisations and central charge}

In order to make more detailed comparisons with conformal field theory one should study R-charge in Landau-Ginzburg models. This is encoded in categories $\hmf(R,W)^{\mathrm{gr}}$ of \textsl{graded matrix factorisations} as follows \cite[App.\,A.4]{cr1006.5609}. Let $R=k[x_1,\ldots,x_n]$ be graded via the assignment of degrees $|x_i|\in \Q_{\geqslant 0}$ to the variables~$x_i$, and let $W \in R$ be homogeneous of degree~2. Objects of $\hmf(R,W)^{\mathrm{gr}}$ are matrix factorisations $(X,d_X)$ where in addition~$X$ is a graded module and~$d_X$ is homogeneous of degree~1, and morphisms are as in $\hmf(R,W)$, but with the additional condition that they must have degree zero. 

\begin{definition}
Let $W \in k[x_1,\ldots,x_n]$ be a homogeneous potential as above. Its \textsl{central charge} is 
\be
\widehat c_W = \sum_{i=1}^n ( 1 - |x_i| ) \, . 
\ee
\end{definition}

In the general construction of Section~\ref{sec:equibicat} the condition of quantum dimensions for 1-morphisms being invertible plays a central role, and it is natural to ask when this condition can be satisfied. In the case of rational CFT it is known (see \cite{ffrs0909.5013} and references therein) that any two theories with identical central charge and identical left and right symmetry algebras are related by the construction of Section~\ref{sec:equibicat}. Furthermore, the idea of computing correlators in one theory by inserting ``islands'' of another theory on the worldsheet, separated by defect lines~$X$, fundamentally hinges on the topological nature of~$X$ as argued in Section~\ref{introduction}, and topological defects exist only between CFTs of the same central charge. In this sense the next result on defects in Landau-Ginzburg models is not unexpected: 

\begin{proposition}\label{prop:centralchargesandqdims}
Let $X \in \hmf(k[z_1,\ldots,z_m,x_1,\ldots,x_n],V-W)^{\mathrm{gr}}$ with $m=n$ mod~2. $X$ can have invertible quantum dimensions only if $\widehat c_W = \widehat c_V$. 
\end{proposition}
\begin{proof}
By Theorem~\ref{thm:LGbicatsummary}(ii) (and Footnote~\ref{ftn:qdims-pivotal-up-to-shifts}), up to a sign 
the right quantum dimension of~$X$ is the polynomial 
\be
\Res \left[ \frac{\str( \partial_{x_1}d_X \ldots \partial_{x_n} d_X \, \partial_{z_1} d_X \ldots \partial_{z_m} d_X) \, \underline{\operatorname{d}\! x}} {\partial_{x_1}W, \ldots, \partial_{x_n} W} \right] \in k[z] \, . 
\ee
This is invertible if and only if it is a nonzero constant in~$k$, so we have to ask when $\dimr(X)$ has degree zero. Since $|W|=2$ we have $|\partial_{x_i}W| = 2 - |x_i|$, and residue theory \cite[(1.10.5)]{Lipman} tells us that $\Res[ \frac{(-) \, \underline{\operatorname{d}\! x}}{\partial_{x_1}W, \ldots, \partial_{x_n} W}]$ is homogeneous of degree $\sum_{i=1}^n ( 2|x_i| -2 ) $. Furthermore, $|d_X| = 1$ and $\partial_{x_i}d_X$, $\partial_{z_j} d_X$ are homogeneous of degree $1 - |x_i|$, $1 - |z_j|$, respectively. It follows that
\be
	|\dimr(X) \,| = \Big| \sum_{i=1}^n ( 2 |x_i| - 2 ) + \sum_{j=1}^m ( 1 - |z_j| ) + \sum_{i=1}^n ( 1- |x_i| ) \Big| = | \widehat c_V - \widehat c_W | 
\ee
which vanishes if and only if $\widehat c_W = \widehat c_V$. The argument for $\diml(X)$ is exactly the same, thus completing the proof. 
\end{proof}

\subsection{Open/closed topological field theory}\label{subsec:ocTFT}

In this section we establish that generalised orbifolds of Landau-Ginzburg models give rise to open/closed topological field theories (TFTs). We start with a concise review of the generators-and-relations description of two-dimensional TFT, and explain how ordinary (non-orbifolded) Landau-Ginzburg models give rise to this structure. 

Recall from~\cite{l0010269, ms0609042} that one way to present a two-dimensional open/closed TFT (see also Remark~\ref{rem:trivialandproperdefectTFTs}(i)) is by the data of 
\begin{itemize}
\item a commutative Frobenius algebra~$C$, 
\item a Calabi-Yau category~$\mathcal O$ (see \cite[Sect\,7.2]{cw1007.2679}), 
\item \textsl{bulk-boundary maps} $\beta_X: C \rightarrow \End_{\mathcal O}(X)$ and \textsl{boundary-bulk maps} $\beta^X: \End_{\mathcal O}(X) \rightarrow C$ for all $X\in \mathcal O$. 
\end{itemize}
These data are subject to the following conditions. 
\begin{itemize}
\item The bulk-boundary maps $\beta_X$ are morphisms of unital algebras that map into the centre of $\End_{\mathcal O}(X)$. 
\item $\beta_X$ and $\beta^X$ are mutually adjoint with respect to the nondegenerate pairings $\langle -,- \rangle$ on~$C$ and $\langle -,- \rangle_X$ on $\End_{\mathcal O}(X)$ (which are part of the Frobenius and Calabi-Yau structure): 
$$
\big\langle \beta_X(\phi) , \Psi \big\rangle_X = \big\langle \phi, \beta^X(\Psi) \big\rangle
$$
for all $\phi\in C$ and $\Psi\in \End_{\mathcal O}(X)$. 
\item The \textsl{Cardy condition} is satisfied, i.\,e.~we have
$$
\str ( {}_\Psi m_\Phi ) = \big\langle \beta^X(\Phi), \beta^Y(\Psi) \big\rangle
$$
for all $\Phi: X \rightarrow X$, $\Psi: Y \rightarrow Y$ where ${}_\Psi m_\Phi (\alpha) = \Psi \circ \alpha\circ \Phi$ for all $\alpha \in \Hom_{\mathcal O}(X,Y)$. 
\end{itemize}

Every Landau-Ginzburg potential $W\in R$ gives rise to a TFT with closed state space and open sector category
\be\label{eq:COnonorbi}
C = R/(\partial W) \, , \quad
\mathcal O = \hmf(R,W) \, . 
\ee
Given two boundary conditions $X,Y \in \hmf(R,W)$ the space $\Hom^\bullet(X,Y)$ of boundary operators comprises both the even and odd cohomology of the differential $d_Y\circ (-) - (-1)^{|-|} (-) \circ d_X$, while by definition $\Hom(X,Y) := \Hom_{\hmf(R,W)}(X,Y)$ is only the even part. Nevertheless, thanks to its triangulated structure the category $\hmf(R,W)$ is sufficient to describe the full open/closed TFT since
\be
\Hom^1(X,Y) \cong \Hom^0(X,Y[1]) \, . 
\ee

Below we will specify the remaining TFT data. The difficult part in checking that these data satisfy the TFT axioms is to establish the Cardy condition and the nondegeneracy of the open sector pairing. This was first done in \cite{pv1002.2116} and \cite{m0912.1629}, respectively; see also \cite{buchweitzflenner08, dyckmurf, dm1102.2957, cm1208.1481}. 

The bulk pairing for a Landau-Ginzburg model with potential $W\in R = k[x_1,\ldots,x_n]$ is given by \cite{v1991}
\begin{align}
& \langle -,- \rangle_W : R/(\partial W) \times R/(\partial W) \lra \C \, , \nonumber \\
& \langle \phi_1, \phi_2 \rangle_W 
 = \Res \left[ \frac{\phi_1 \phi_2 \, \underline{\operatorname{d}\! x}}{\partial_{x_1} W, \ldots, \partial_{x_n} W} \right] , \label{eq:Vafapairing} 
\end{align}
and we will also write $\langle \phi \rangle_W  = \langle \phi, 1 \rangle_W $. 
The boundary pairings for $X,Y \in \hmf(R,W)$ are \cite{kl0305, hl0404}
\begin{align}
& \langle -,- \rangle_X : \Hom(X,Y) \times \Hom(Y,X)[n] \lra \C \, , \nonumber \\
& \langle \Psi_1, \Psi_2 \rangle_X  
 = \Res \left[ \frac{ \str( \Psi_1 \Psi_2 \, \partial_{x_1} d_X \ldots \partial_{x_n} d_X) \, \underline{\operatorname{d}\! x}}{\partial_{x_1} W, \ldots, \partial_{x_n} W} \right] . \label{eq:KapLi}
\end{align}
Furthermore, the bulk-boundary and boundary-bulk maps are \cite{kr0405232} (with $\Lambda_X = \partial_{x_1} d_X \ldots \partial_{x_n} d_X$)
\be\label{eq:bubos}
\beta_X(\phi) = 
\begin{tikzpicture}[very thick,scale=0.5,color=blue!50!black, baseline]
\draw[very thick] (0,-3) -- (0,0); 
\draw[very thick] (0,0) -- (0,3);
\draw[dashed] (0,-2) .. controls +(-0.5,1) and +(0,-1) .. (-1,0);
\draw[dashed] (-1,0) .. controls +(0,1) and +(-0.5,-1) .. (0,2);
\fill (0,-3) circle (0pt) node[right] {{\small$X$}};
\fill (-1,0) circle (3.3pt) node[left] {{\small$\phi$}};
\fill (-1.1,-1.3) circle (0pt) node {{\small$I_W$}};
\end{tikzpicture} 
 = \phi \cdot 1_X \, , \quad
\beta^X(\Psi) = 
\begin{tikzpicture}[very thick,scale=0.6,color=blue!50!black, baseline]
\draw (0,0) circle (1.5);
\draw[->, very thick] (0.100,-1.5) -- (-0.101,-1.5) node[above] {}; 
\draw[->, very thick] (-0.100,1.5) -- (0.101,1.5) node[below] {}; 
\fill (180:1.5) circle (3.3pt) node[right] {{\small$\Psi$}};
\fill (270:1.5) circle (0pt) node[above] {{\small$X$}};
\end{tikzpicture} 
 = (-1)^{n+1 \choose 2} \str( \Psi \, \Lambda_X) \, .  
\ee

\begin{remark}\label{rem:LGocTFT} 
Conjecturally, Landau-Ginzburg models give rise to a TFT with defects $\tau_{\mathrm{LG}}$, with $\LG$ equivalent to the bicategory of worldsheet phases $\mathcal{D}_{\tau_{\mathrm{LG}}}$ associated to $\tau_{\mathrm{LG}}$ as discussed in Section~\ref{subsec:worldsheetbicategory}. 
Since a generators-and-relations description of general 2d TFTs with defects (as opposed to open/closed TFTs, say) is not known it is difficult to prove this conjecture. 
Und the assumption that it is true, the above construction would be a special case of Remark~\ref{rem:trivialandproperdefectTFTs}(i). 
\end{remark}

We will now show that we can also associate an open/closed TFT to every object $(W,A)$ in the generalised orbifold category $\LGorb$. The natural choices for the bulk space and the boundary category are
\be
C_{\text{orb}} = \End_{AA}(A) \, , \quad
\mathcal O_{\text{orb}} = \LGorb \big( 0, (W,A) \big)
\ee
which of course reduce to the unorbifolded case~\eqref{eq:COnonorbi} in the special case $A = I_W$. Next we specify the bulk and boundary pairings in $\LGorb$. The latter are simply given by~\eqref{eq:KapLi} when restricted to $A$-modules~$X$ and $A$-module maps $\Psi_1, \Psi_2$, while the orbifold bulk pairing is
\be\label{eq:orbibulkpairing}
\langle \phi_1, \phi_2 \rangle_{(W,A)} = 
\left\langle \, 
\begin{tikzpicture}[very thick,scale=0.6,color=green!50!black, baseline]
\draw (0,0) circle (1.5);
\draw[<-, very thick] (0.100,-1.5) -- (-0.101,-1.5) node[above] {}; 
\draw[<-, very thick] (-0.100,1.5) -- (0.101,1.5) node[below] {}; 
\fill (-22.5:1.5) circle (3.3pt) node[left] {{\small$\phi_2$}};
\fill (22.5:1.5) circle (3.3pt) node[left] {{\small$\phi_1$}};
\end{tikzpicture} 
\, \right\rangle_{\raisemath{10pt}{\!\!\!W}} \, . 
\ee
This makes $C_{\text{orb}}$ into a Frobenius algebra; commutativity follows from the bimodule map property.
Finally, the bulk-boundary and boundary-bulk maps in $\LGorb$ are 
\be\label{eq:orbibubos}
\beta_X^{\text{orb}}(\phi) = 
\begin{tikzpicture}[very thick,scale=0.75,color=blue!50!black, baseline]
\draw (0,-1) -- (0,1); 
\draw (0,-1) node[right] (X) {{\small$X$}};
\draw (0,1) node[right] (Xu) {{\small$X$}};

\fill[color=green!50!black] (-0.465,0.0) circle (2.5pt) node[left] {{\small$\phi$}};

\draw[color=green!50!black]  (-0.5,-0.5) node[Odot] (unit) {}; 
\fill[color=green!50!black]  (0,0.6) circle (2.5pt) node (meet) {};
\draw[color=green!50!black]  (unit) .. controls +(0,0.5) and +(-0.5,-0.5) .. (0,0.6);
\end{tikzpicture} 
, \quad
\beta^X_{\text{orb}}(\Psi) = 
\begin{tikzpicture}[very thick,scale=0.6,color=blue!50!black, baseline]
\draw (0,0) circle (1.5);
\draw[->, very thick] (0.100,-1.5) -- (-0.101,-1.5) node[above] {}; 
\draw[->, very thick] (-0.100,1.5) -- (0.101,1.5) node[below] {}; 
\fill (202.5:1.5) circle (3.3pt) node[right] {{\small$\Psi$}};
\fill (270:1.5) circle (0pt) node[above] {{\small$X$}};

\fill[color=green!50!black] (157.5:1.5) circle (3.3pt) node[right] {};

\draw[-dot-, color=green!50!black] (-3,-0.5) .. controls +(0,-1) and +(0,-1) .. (-2,-0.5);

\draw[color=green!50!black] (-2,-0.5) .. controls +(0,0.5) and +(-0.5,-0.5) .. (157.5:1.5);

\draw[color=green!50!black] (-2.5,-1.3) -- (-2.5,-2);
\draw[color=green!50!black] (-3,-0.5) -- (-3,2);

\end{tikzpicture} 
\, . 
\ee

\begin{theorem}\label{thm:LGorbiTFT}
Every $(W,A) \in \LGorb$ gives rise to an open/closed TFT via the above data. 
\end{theorem}
\begin{proof}
We need to check the nondegeneracy of the bulk and boundary pairings, the adjunction between $\beta_X^{\text{orb}}$ and $\beta^X_{\text{orb}}$, and the Cardy condition; the other axioms are clear. 

That the boundary pairings are nondegenerate in the orbifold completion would be the special case of Corollary~\ref{prop:traceflipgivesnondeg} where the defect~$X$ is of the form $(k,0)\rightarrow (R,W)$, but only if $\LGorb$ was pivotal in the standard sense. This is not the case, but inspection of the proof shows that it is enough to assume instead of pivotality that the identities~\eqref{eq:doubledefectwraps} continue to hold when an arbitrary 2-morphism is inserted on the tensor product. This is the case as follows directly from the proof of \cite[Prop.\,8.5(iii)]{cm1208.1481}. 

To show that the bulk pairings~\eqref{eq:orbibulkpairing} are nondegenerate we use Corollary~\ref{prop:traceflipgivesnondeg}. Its assumptions are satisfied since it follows from Theorem~\ref{thm:LGbicatsummary}(ii) (see also \cite[Cor.\,8.3]{cm1208.1481}) that 
\be
\left\langle \, 
\begin{tikzpicture}[very thick,scale=0.6,color=green!50!black, baseline]
\draw (0,0) circle (1.5);
\draw[<-, very thick] (0.100,-1.5) -- (-0.101,-1.5) node[above] {}; 
\draw[<-, very thick] (-0.100,1.5) -- (0.101,1.5) node[below] {}; 
\fill (-22.5:1.5) circle (3.3pt) node[left] {{\small$\phi_2$}};
\fill (22.5:1.5) circle (3.3pt) node[left] {{\small$\phi_1$}};
\end{tikzpicture} 
\, \right\rangle_{\raisemath{10pt}{\!\!\!W}}
=
\left\langle \, 
\begin{tikzpicture}[very thick,scale=0.6,color=green!50!black, baseline]
\draw (0,0) circle (1.5);
\draw[->, very thick] (0.100,-1.5) -- (-0.101,-1.5) node[above] {}; 
\draw[->, very thick] (-0.100,1.5) -- (0.101,1.5) node[below] {}; 
\fill (157.5:1.5) circle (3.3pt) node[right] {{\small$\phi_1$}};
\fill (202.5:1.5) circle (3.3pt) node[right] {{\small$\phi_2$}};
\end{tikzpicture} 
\, \right\rangle_{\raisemath{10pt}{\!\!\!W}} \, ,
\ee
and furthermore 
\be
\left\langle
\begin{tikzpicture}[very thick,scale=0.6,color=green!50!black, baseline]
\draw (0,0) circle (1.5);
\draw[<-, very thick] (0.100,-1.5) -- (-0.101,-1.5) node[above] {}; 
\draw[<-, very thick] (-0.100,1.5) -- (0.101,1.5) node[below] {}; 
\fill (-22.5:1.5) circle (3.3pt) node[left] {{\small$\phi_2$}};
\fill (22.5:1.5) circle (3.3pt) node[left] {{\small$\phi_1$}};
\end{tikzpicture} 
\right\rangle_{\raisemath{10pt}{\!\!\!W}}
= (-1)^{n+1 \choose 2} \Res\! \left[ \frac{ \str( \phi_1 \phi_2 \, \partial_{x_1} d_A \ldots \partial_{x_n} d_A \, \partial_{x'_1} d_A \ldots \partial_{x'_n} d_A) \, \underline{\operatorname{d}\! x} \underline{\operatorname{d}\! x'}}{\partial_{x_1} W, \ldots, \partial_{x_n} W, \partial_{x'_1} W, \ldots, \partial_{x'_n} W} \right]
\ee
is nondegenerate as a pairing of~$\phi_1$ and~$\phi_2$, because up to a sign it is the boundary pairing for~$A$ viewed as a matrix factorisation of~$\widetilde W$. 

Next we wish to show that the maps  $\beta_X^{\text{orb}}, \beta^X_{\text{orb}}$ in~\eqref{eq:orbibubos} are adjoint in the sense
\be\label{eq:orbibuboadj}
\big\langle \phi, \beta^X_{\text{orb}}(\Psi) \big\rangle_{(W,A)} = 
\big\langle \beta_X^{\text{orb}}(\phi), \Psi \big\rangle_X \, . 
\ee
By definition the left-hand side equals
\be
\left\langle \,
\begin{tikzpicture}[very thick,scale=0.6,color=blue!50!black, baseline]
\draw (0,0) circle (1.5);
\draw[->, very thick] (0.100,-1.5) -- (-0.101,-1.5) node[above] {}; 
\draw[->, very thick] (-0.100,1.5) -- (0.101,1.5) node[below] {}; 
\fill (202.5:1.5) circle (3.3pt) node[right] {{\small$\Psi$}};
\fill (270:1.5) circle (0pt) node[above] {{\small$X$}};

\fill[color=green!50!black] (140:1.5) circle (3.3pt) node[right] {};

\draw[color=green!50!black] (-4,0) circle (1.5);
\coordinate (Psi2) at ($ (-4,0) + (22.5:1.5) $);
\fill[color=green!50!black] (Psi2) circle (3.3pt) node[left] {{\small$\phi$}};

\coordinate (Psi3) at ($ (-4,0) + (-22.5:1.5) $);
\fill[color=green!50!black] (Psi3) circle (3.3pt) node[left] {};

\draw[color=green!50!black] (Psi3) -- (140:1.5);

\draw[<-, very thick, color=green!50!black] (-3.900,-1.5) -- (-3.901,-1.5) node[above] {}; 
\draw[->, very thick, color=green!50!black] (-4.100,1.5) -- (-4.101,1.5) node[below] {}; 

\end{tikzpicture} 
\, \right\rangle_{\raisemath{10pt}{\!\!\!W}}
\, , 
\ee
part of which we compute as follows: 
\be
\begin{tikzpicture}[very thick,scale=0.55,color=green!50!black, baseline]
\draw (-4,0) circle (1.5);
\coordinate (Psi2) at ($ (-4,0) + (22.5:1.5) $);
\fill (Psi2) circle (3.3pt) node[right] {};

\coordinate (Psi3) at ($ (-4,0) + (-22.5:1.5) $);
\fill (Psi3) circle (3.3pt) node[right] {{\small$\phi$}};

\draw (Psi2) .. controls +(0.2,0.1) and +(-0.2,-0.2) .. (-1.5,1.5);

\draw[<-, very thick] (-3.900,-1.5) -- (-3.901,-1.5) node[above] {}; 
\draw[->, very thick] (-4.100,1.5) -- (-4.101,1.5) node[below] {}; 

\end{tikzpicture} 
\stackrel{(1)}{=}
\begin{tikzpicture}[very thick,scale=0.55,color=green!50!black, baseline]

\draw[-dot-] (0,1) .. controls +(0,1) and +(0,1) .. (-1,1);
\draw[directedgreen, color=green!50!black] (1,1) .. controls +(0,-1) and +(0,-1) .. (0,1);
\draw (-0.5,2.2) node[Odot] (end) {}; 
\draw (-0.5,1.8) -- (end); 

\draw[-dot-] (0,-1) .. controls +(0,-1) and +(0,-1) .. (-1,-1);
\draw[redirectedgreen, color=green!50!black] (1,-1) .. controls +(0,1) and +(0,1) .. (0,-1);
\draw (-0.5,-2.2) node[Odot] (end) {}; 
\draw (-0.5,-1.8) -- (end); 

\draw (-1,1) -- (-1,-1);

\draw[redirectedgreen, color=green!50!black] (1,1) .. controls +(0,1) and +(0,1) .. (2,1);

\draw[directedgreen, color=green!50!black] (1,-1) .. controls +(0,-1) and +(0,-1) .. (2,-1);

\draw (2,1) -- (2,-1);

\fill (2,-0.5) circle (3.3pt) node[right] {{\small$\phi$}};

\fill (2,0.5) circle (3.3pt) node[right] {};

\draw (2,0.5) .. controls +(0.2,0.1) and +(-0.2,-0.2) .. (3,1.5);

\end{tikzpicture} 
\stackrel{(2)}{=}
\begin{tikzpicture}[very thick,scale=0.55,color=green!50!black, baseline]

\draw[-dot-] (-5.5,0) .. controls +(0,2) and +(0,2) .. (-2.5,0);
\draw[-dot-] (-5.5,0) .. controls +(0,-2) and +(0,-2) .. (-2.5,0);

\coordinate (Psi2) at ($ (-4,0) + (22.5:1.5) $);
\fill (Psi2) circle (3.3pt) node[right] {};

\coordinate (Psi3) at ($ (-4,0) + (-22.5:1.5) $);
\fill (Psi3) circle (3.3pt) node[right] {{\small$\phi$}};

\draw (Psi2) .. controls +(0.2,0.1) and +(-0.2,-0.2) .. (-1.5,1.5);

\draw (-4,2.2) node[Odot] (end) {}; 
\draw (-4,1.5) -- (end);

\draw (-4,-2.2) node[Odot] (end) {}; 
\draw (-4,-1.5) -- (end);

\end{tikzpicture} 
\stackrel{(3)}{=}
\begin{tikzpicture}[very thick,scale=0.55,color=green!50!black, baseline]

\draw[-dot-] (-1,-1) .. controls +(0,1) and +(0,1) .. (1,-1);
\draw[-dot-] (-1,-1) .. controls +(0,-1) and +(0,-1) .. (1,-1);

\fill (1,-1) circle (3.3pt) node[right] {{\small$\phi$}};

\draw (0,-2.4) node[Odot] (end) {}; 
\draw (0,-1.7) -- (end);

\draw[-dot-] (-1,1.4) .. controls +(0,-1) and +(0,-1) .. (1,1.4);

\draw (-1,2.0) node[Odot] (end) {}; 
\draw (-1,1.4) -- (end);

\draw (0,-0.3) -- (0,0.7);

\draw (1,1.4) .. controls +(0,0.5) and +(-0.2,-0.2) .. (1.8,1.9);

\end{tikzpicture} 
\stackrel{(4)}{=}
\begin{tikzpicture}[very thick,scale=0.55,color=green!50!black, baseline]

\draw[-dot-] (-1,-1) .. controls +(0,1) and +(0,1) .. (1,-1);
\draw[-dot-] (-1,-1) .. controls +(0,-1) and +(0,-1) .. (1,-1);

\fill (0,0.2) circle (3.3pt) node[right] {{\small$\phi$}};

\draw (0,-2.4) node[Odot] (end) {}; 
\draw (0,-1.7) -- (end);

\draw[-dot-] (-1,1.4) .. controls +(0,-1) and +(0,-1) .. (1,1.4);

\draw (-1,2.0) node[Odot] (end) {}; 
\draw (-1,1.4) -- (end);

\draw (0,-0.3) -- (0,0.7);

\draw (1,1.4) .. controls +(0,0.5) and +(-0.2,-0.2) .. (1.8,1.9);

\end{tikzpicture} 
\stackrel{(5)}{=}
\begin{tikzpicture}[very thick,scale=0.55,color=blue!50!black, baseline]
\fill[color=green!50!black] (-0.465,0.0) circle (3.3pt) node[right] {{\small$\phi$}};

\draw[color=green!50!black]  (-0.5,-1.5) node[Odot] (unit) {}; 
\draw[color=green!50!black]  (unit) .. controls +(0,0.5) and +(-0.5,-0.5) .. (0,1.9);
\end{tikzpicture} 
\, . 
\ee
Here we used (1) that~$A$ is symmetric Frobenius, (2) Zorro moves, (3) the Frobenius property of~$A$, (4) that~$\phi$ is a bimodule map, and (5) that~$A$ is a separable Frobenius algebra. Thus the left-hand side of~\eqref{eq:orbibuboadj} is
\be
\left\langle
\begin{tikzpicture}[very thick,scale=0.6,color=blue!50!black, baseline]
\draw (0,0) circle (1.5);
\draw[->, very thick] (0.100,-1.5) -- (-0.101,-1.5) node[above] {}; 
\draw[->, very thick] (-0.100,1.5) -- (0.101,1.5) node[below] {}; 
\fill (202.5:1.5) circle (3.3pt) node[right] {{\small$\Psi$}};
\fill (270:1.5) circle (0pt) node[above] {{\small$X$}};

\fill[color=green!50!black] (140:1.5) circle (3.3pt) node[right] {};

\coordinate (Psi2) at ($ (-4,0) + (-22.5:1.5) $);
\draw[color=green!50!black]  (Psi2) node[Odot] (unit) {}; 

\fill[color=green!50!black] (-2.15,0.2) circle (3.3pt) node[left] {{\small$\phi$}};

\draw[color=green!50!black] (Psi2) .. controls +(0.1,0.5) and +(-0.5,-0.2) .. (140:1.5);

\end{tikzpicture} 
\, \right\rangle_{\raisemath{10pt}{\!\!\!W}}
= \big\langle 1, \beta^X( \beta_X^{\text{orb}}(\phi) \cdot \Psi) \big\rangle_W 
= \big\langle \beta_X^{\text{orb}}(\phi), \Psi \big\rangle_X
\ee
where we used that $\beta^X$ is adjoint to $\beta_X$ in $\LG$ and $\beta_X(1) = 1_X$. 

Finally we want to prove the Cardy condition in $\LGorb$. Thus we compute
\begin{align}
& \big\langle \beta^X_{\text{orb}}(\Phi) , \beta^Y_{\text{orb}}(\Psi) \big\rangle_{(W,A)} = 
\left\langle \,
\begin{tikzpicture}[very thick,scale=0.55,color=blue!50!black, baseline]
\draw[color=green!50!black] (-4,0) circle (1.5);
\draw[<-, very thick, color=green!50!black] (-3.900,-1.5) -- (-3.901,-1.5) node[above] {}; 
\draw[->, very thick, color=green!50!black] (-4.100,1.5) -- (-4.101,1.5) node[below] {}; 
\coordinate (Psi2) at ($ (-4,0) + (22.5:1.5) $);
\fill[color=green!50!black] (Psi2) circle (3.3pt) node[right] {};
\coordinate (Psi3) at ($ (-4,0) + (-22.5:1.5) $);
\fill[color=green!50!black] (Psi3) circle (3.3pt) node[right] {};
\draw (-0.5,0.95) circle (0.75);
\coordinate (meet2) at ($ (-0.5,0.95) + (150:0.75) $);
\fill[color=green!50!black] (meet2) circle (3.3pt) node[right] {};
\draw[color=green!50!black] (Psi2) -- (meet2);
\coordinate (Phi) at ($ (-0.5,0.95) + (210:0.75) $);
\fill (Phi) circle (3.3pt) node[right] {{\small $\Phi$}};
\coordinate (X) at ($ (-0.5,0.95) + (0:0.6) $);
\fill (X) circle (0pt) node[right] {{\small $X$}};
\draw[->, very thick] (-0.5001,1.7) -- (-0.49,1.7) node[above] {}; 
\draw[->, very thick] (-0.51,0.2) -- (-0.5201,0.2) node[above] {}; 
\draw (-0.5,-0.95) circle (0.75);
\coordinate (meet3) at ($ (-0.5,-0.95) + (150:0.75) $);
\fill[color=green!50!black] (meet3) circle (3.3pt) node[right] {};
\draw[color=green!50!black] (Psi3) -- (meet3);
\coordinate (Psi) at ($ (-0.5,-0.95) + (210:0.75) $);
\fill (Psi) circle (3.3pt) node[right] {{\small $\Psi$}};
\coordinate (Y) at ($ (-0.5,-0.95) + (0:0.6) $);
\fill (Y) circle (0pt) node[right] {{\small $Y$}};
\draw[->, very thick] (-0.49,-1.7) -- (-0.5001,-1.7) node[above] {}; 
\draw[->, very thick] (-0.5201,-0.2) -- (-0.51,-0.2) node[above] {}; 
\end{tikzpicture} 
\, \right\rangle_{\raisemath{12pt}{\!\!\!W}}
\nonumber \\
& =
\left\langle \,
\begin{tikzpicture}[very thick,scale=0.6,color=blue!50!black, baseline]
\draw (0,0) circle (1.5);
\draw[->, very thick] (0.100,-1.5) -- (-0.101,-1.5) node[above] {}; 
\draw[->, very thick] (-0.100,1.5) -- (0.101,1.5) node[below] {}; 
\fill (202.5:1.5) circle (3.3pt) node[right] {{\small$\Phi$}};
\fill (270:1.5) circle (0pt) node[above] {{\small$X$}};
\fill[color=green!50!black] (140:1.5) circle (3.3pt) node[right] {};
%
\draw[color=green!50!black] (-6,0) circle (1.0);
\coordinate (Psi2) at ($ (-6,0) + (22.5:1.0) $);
\fill[color=green!50!black] (Psi2) circle (3.3pt) node[right] {};
\coordinate (Psi3) at ($ (-6,0) + (-22.5:1.0) $);
\fill[color=green!50!black] (Psi3) circle (3.3pt) node[left] {};
\draw[color=green!50!black] (Psi2) -- (140:1.5);
\draw[<-, very thick, color=green!50!black] (-5.900,-1.0) -- (-5.901,-1.0) node[above] {}; 
\draw[->, very thick, color=green!50!black] (-6.100,1.0) -- (-6.101,1.0) node[below] {}; 
\draw (-3.3,-0.95) circle (0.75);
\coordinate (meet3) at ($ (-3.3,-0.95) + (150:0.75) $);
\fill[color=green!50!black] (meet3) circle (3.3pt) node[right] {};
\draw[color=green!50!black] (Psi3) -- (meet3);
\coordinate (Psi) at ($ (-3.3,-0.95) + (210:0.75) $);
\fill (Psi) circle (3.3pt) node[right] {{\small $\Psi$}};
\coordinate (Y) at ($ (-3.3,-0.95) + (0:0.6) $);
\fill (Y) circle (0pt) node[right] {{\small $Y$}};
\draw[->, very thick] (-3.29,-1.7) -- (-3.3001,-1.7) node[above] {}; 
\draw[->, very thick] (-3.3201,-0.2) -- (-3.31,-0.2) node[above] {}; 
\end{tikzpicture} 
\, \right\rangle_{\raisemath{13pt}{\!\!\!W}} \nonumber \\
& 
=
\,
\begin{tikzpicture}[very thick,scale=0.6,color=blue!50!black, baseline]
\draw (0,0) circle (2.5);
\draw[<-, very thick] (0.100,-2.5) -- (-0.101,-2.5) node[above] {}; 
\draw[<-, very thick] (-0.100,2.5) -- (0.101,2.5) node[below] {}; 
\fill (-22.5:2.5) circle (3.3pt) node[right] {{\small$\Phi$}};
\fill (270:2.5) circle (0pt) node[above] {{\small$X$}};
\fill[color=green!50!black] (22.5:2.5) circle (3.3pt) node[right] {};
%
\draw[color=green!50!black] (-1,0) circle (1.0);
\coordinate (Psi2) at ($ (-1,0) + (22.5:1.0) $);
\fill[color=green!50!black] (Psi2) circle (3.3pt) node[right] {};
\coordinate (Psi3) at ($ (-1,0) + (-22.5:1.0) $);
\fill[color=green!50!black] (Psi3) circle (3.3pt) node[left] {};
\draw[color=green!50!black] (Psi2) -- (22.5:2.5);
\draw[<-, very thick, color=green!50!black] (-0.900,-1.0) -- (-0.901,-1.0) node[above] {}; 
\draw[->, very thick, color=green!50!black] (-1.100,1.0) -- (-1.101,1.0) node[below] {}; 
\draw (1.3,-0.35) circle (0.75);
\coordinate (meet3) at ($ (1.3,-0.35) + (150:0.75) $);
\fill[color=green!50!black] (meet3) circle (3.3pt) node[right] {};
\draw[color=green!50!black] (Psi3) -- (meet3);
\coordinate (Psi) at ($ (1.3,-0.35) + (210:0.75) $);
\fill (Psi) circle (3.3pt) node[right] {{\small $\Psi$}};
\coordinate (Y) at ($ (1.3,-0.35) + (270:0.8) $);
\fill (Y) circle (0pt) node[below] {{\small $Y$}};
\draw[<-, very thick] (1.29,-1.1) -- (1.3001,-1.1) node[above] {}; 
\draw[<-, very thick] (1.3201,0.4) -- (1.31,0.4) node[above] {}; 
\end{tikzpicture} 
\, =\,
\begin{tikzpicture}[very thick,scale=0.6,color=blue!50!black, baseline]
\draw (0,0) circle (2.5);
\draw[<-, very thick] (0.100,-2.5) -- (-0.101,-2.5) node[above] {}; 
\draw[<-, very thick] (-0.100,2.5) -- (0.101,2.5) node[below] {}; 
\fill (-22.5:2.5) circle (3.3pt) node[right] {{\small$\Phi$}};
\fill (270:2.5) circle (0pt) node[above] {{\small$X$}};
\fill[color=green!50!black] (45:2.5) circle (3.3pt) node[right] {};
%
\coordinate (Psi2) at ($ (-1.5,-1.0) + (0,0) $);
\fill[color=green!50!black] (Psi2) node[Odot] {};
\coordinate (Psi3) at ($ (-2.05,0) + (-22.5:1.0) $);
\fill[color=green!50!black] (Psi3) circle (3.3pt) node[left] {};
\draw[color=green!50!black] (-1.48,-0.95)  .. controls +(0.1,0.9) and +(-0.2,-0.2) .. (45:2.5);
\draw (1.3,-0.35) circle (0.75);
\coordinate (meet3) at ($ (1.3,-0.35) + (150:0.75) $);
\fill[color=green!50!black] (meet3) circle (3.3pt) node[right] {};
\draw[color=green!50!black] (Psi3) -- (meet3);
\coordinate (Psi) at ($ (1.3,-0.35) + (210:0.75) $);
\fill (Psi) circle (3.3pt) node[right] {{\small $\Psi$}};
\coordinate (Y) at ($ (1.3,-0.35) + (270:0.8) $);
\fill (Y) circle (0pt) node[below] {{\small $Y$}};
\draw[<-, very thick] (1.29,-1.1) -- (1.3001,-1.1) node[above] {}; 
\draw[<-, very thick] (1.3201,0.4) -- (1.31,0.4) node[above] {}; 
\end{tikzpicture} 
\,
=
\,
\begin{tikzpicture}[very thick,scale=0.6,color=blue!50!black, baseline]
\draw (0,0) circle (2.5);
\draw[<-, very thick] (0.100,-2.5) -- (-0.101,-2.5) node[above] {}; 
\draw[<-, very thick] (-0.100,2.5) -- (0.101,2.5) node[below] {}; 
\fill (170.8:2.5) circle (3.3pt) node[right] {{\small$\Phi^\vee$}};
\fill (270:2.5) circle (0pt) node[above] {{\small$X$}};
\fill[color=green!50!black] (190:2.5) circle (3.3pt) node[left] {};
\draw[-dot-, color=green!50!black] (-1.75,-0.3) .. controls +(0,-0.7) and +(0,-0.7) .. (-1,-0.3);
\draw (0,0) circle (0.75);
\coordinate (meet3) at ($ (0,0) + (180:0.75) $);
\fill[color=green!50!black] (meet3) circle (3.3pt) node[right] {};
\draw[color=green!50!black] (-1,-0.3)  .. controls +(0.0,0.3) and +(-0.1,-0.1) .. (meet3);
\draw[color=green!50!black] (-1.75,-0.3)  .. controls +(-0.0,0.4) and +(-0.1,0.3) .. (190:2.5);
\coordinate (Psi) at ($ (0,0) + (150:0.75) $);
\fill (Psi) circle (3.3pt) node[left] {{\small $\Psi$}};
\coordinate (Y) at ($ (0,0) + (270:0.8) $);
\fill (Y) circle (0pt) node[below] {{\small $Y$}};
\draw[<-, very thick] (-0.09,-0.75) -- (0.01,-0.75) node[above] {}; 
\draw[<-, very thick] (0.0201,0.75) -- (0.01,0.75) node[above] {}; 
\draw[color=green!50!black]  (-1.375,-1.4) node[Odot] (unit) {}; 
\draw[color=green!50!black] (unit) -- (-1.375,-0.8); 
\end{tikzpicture} 
\,
\end{align}
where we used Theorem~\ref{thm:LGbicatsummary}(ii) together with~\eqref{eq:Vafapairing} in the third step and Theorem~\ref{thm:LGbicatsummary}(i) in the last. Note that after applying the orientation reversal identity~\eqref{eq:traceorientationreversal} we could drop the brackets for the last three expressions since $X: 0\rightarrow W$ and $\langle-\rangle_0 = (-)$. 
Now we observe that
\be
\pi = 
\begin{tikzpicture}[very thick,scale=0.75,color=blue!50!black, baseline]

\draw (-1,-1) node[left] (X) {{\small$X^\dagger$}};
\draw (1,-1) node[right] (Xu) {{\small$Y\vphantom{X^\dagger}$}};

\draw (-1,-1) -- (-1,1); 
\draw (1,-1) -- (1,1); 

\fill[color=green!50!black] (-1,-0.2) circle (2pt) node (meet) {};
\fill[color=green!50!black] (1,0.6) circle (2pt) node (meet) {};

\draw[-dot-, color=green!50!black] (0.35,-0.0) .. controls +(0,-0.5) and +(0,-0.5) .. (-0.35,-0.0);

\draw[color=green!50!black] (0.35,-0.0) .. controls +(0,0.25) and +(-0.25,-0.25) .. (1,0.6);
\draw[color=green!50!black] (-0.35,-0.0) .. controls +(0,0.5) and +(0.25,0.5) .. (-1,-0.2);

\draw[color=green!50!black] (0,-0.75) node[Odot] (down) {}; 
\draw[color=green!50!black] (down) -- (0,-0.35); 

\end{tikzpicture} 
\ee
is the projector $\pi_A^{X^\dagger,Y}$ to $X^\dagger \otimes_A Y$, see~\eqref{eq:piA}. Thus we have splitting maps $\xi : X^\dagger \otimes_A Y \rightarrow X^\dagger \otimes Y$ and $\vartheta: X^\dagger \otimes Y \rightarrow X^\dagger \otimes_A Y$ such that $\xi \vartheta = \pi$ and $\vartheta \xi = 1$, and we can compute
\be
\begin{tikzpicture}[very thick,scale=0.6,color=blue!50!black, baseline]
\draw (0,0) circle (2.5);
\draw[<-, very thick] (0.100,-2.5) -- (-0.101,-2.5) node[above] {}; 
\draw[<-, very thick] (-0.100,2.5) -- (0.101,2.5) node[below] {}; 
\fill (170.8:2.5) circle (3.3pt) node[right] {{\small$\Phi^\vee$}};
\fill (270:2.5) circle (0pt) node[above] {{\small$X$}};
\fill[color=green!50!black] (190:2.5) circle (3.3pt) node[left] {};
\draw[-dot-, color=green!50!black] (-1.75,-0.3) .. controls +(0,-0.7) and +(0,-0.7) .. (-1,-0.3);
\draw (0,0) circle (0.75);
\coordinate (meet3) at ($ (0,0) + (180:0.75) $);
\fill[color=green!50!black] (meet3) circle (3.3pt) node[right] {};
\draw[color=green!50!black] (-1,-0.3)  .. controls +(0.0,0.3) and +(-0.1,-0.1) .. (meet3);
\draw[color=green!50!black] (-1.75,-0.3)  .. controls +(-0.0,0.4) and +(-0.1,0.3) .. (190:2.5);
\coordinate (Psi) at ($ (0,0) + (150:0.75) $);
\fill (Psi) circle (3.3pt) node[left] {{\small $\Psi$}};
\coordinate (Y) at ($ (0,0) + (270:0.8) $);
\fill (Y) circle (0pt) node[below] {{\small $Y$}};
\draw[<-, very thick] (-0.09,-0.75) -- (0.01,-0.75) node[above] {}; 
\draw[<-, very thick] (0.0201,0.75) -- (0.01,0.75) node[above] {}; 
\draw[color=green!50!black]  (-1.375,-1.4) node[Odot] (unit) {}; 
\draw[color=green!50!black] (unit) -- (-1.375,-0.8); 
\end{tikzpicture} 
=
\begin{tikzpicture}[very thick,scale=0.6,color=blue!50!black, baseline]
\draw[directed] (-1,1.5) .. controls +(0,1) and +(0,1) .. (1,1.5);
\draw[redirected] (-2,1.5) .. controls +(0,2) and +(0,2) .. (2,1.5);

\draw[redirected] (-1,-1.5) .. controls +(0,-1) and +(0,-1) .. (1,-1.5);
\draw[directed] (-2,-1.5) .. controls +(0,-2) and +(0,-2) .. (2,-1.5);

\draw (-1,1.5) -- (-1,-1.5);
\draw (-2,1.5) -- (-2,-1.5);
\draw (1,1.5) -- (1,-1.5);
\draw (2,1.5) -- (2,-1.5);

\draw
(-1.5,0.75) node[rectangle, very thick, blue!50!black,draw,fill=white] 
(alpha) {{\small$ \Phi^\vee \otimes \Psi $}} ;
\draw
(-1.5,-0.75) node[rectangle, very thick, blue!50!black,draw,fill=white] 
(alpha) {{\small$ \;\; \pi \;\; $}} ;

\fill (2.2,0) circle (0pt) node[left] {{\small $Y$}};
\fill (1.2,0) circle (0pt) node[left] {{\small $X$}};

\end{tikzpicture} 
=
\begin{tikzpicture}[very thick,scale=0.6,color=blue!50!black, baseline]
\draw (0,0) circle (1.5);
\draw[->, very thick] (0.100,-1.5) -- (-0.101,-1.5) node[above] {}; 
\draw[->, very thick] (-0.100,1.5) -- (0.101,1.5) node[below] {}; 
\fill (180:1.5) circle (3.3pt) node[right] {{\small$\Phi^\vee \otimes_A \Psi$}};
\fill (270:1.5) circle (0pt) node[below] {{\small$X^\vee \otimes_A Y$}};
\end{tikzpicture} 
\ee
which equals $\str({}_\Psi m_\Phi)$, 
where in the last step we used the morphism migration properties of Theorem~\ref{thm:LGbicatsummary}(i) and $\Psi^\vee \otimes_A \Phi = \vartheta \circ (\Psi^\vee \otimes \Phi) \circ \xi$. 
\end{proof}

\section{Examples of Landau-Ginzburg orbifolds}\label{sec:exLGorbi}

In this section we give several concrete examples of (generalised) orbifolds. We start with an explanation of how conventional orbifolds of Landau-Ginzburg models fit into our formalism. Then we go on to show that the phenomenon of Kn\"orrer periodicity can be viewed as a generalised orbifold equivalence. Finally we explicitly construct the generalised orbifold between A- and D-type singularities, and comment on further applications. 

\subsection{Equivariant matrix factorisations}\label{subsec:equiMF}

To speak of generalised orbifolds in a meaningful way, our first example should be to explain how conventional orbifolds of Landau-Ginzburg models and matrix factorisations can be recovered from the constructions of Sections~\ref{sec:equibicat}--\ref{sec:LGmodels}. We will do so in this section, by recalling the standard notion of $G$-equivariant matrix factorisations and subsequently showing how it embeds into the generalised orbifold formalism. 

Let $W\in R = k[x_1, \ldots, x_n]$ be a potential and let~$G$ be a \textsl{symmetry group of~$W$}, i.\,e.~a finite subgroup of those $R$-automorphisms that leave~$W$ invariant. From the symmetry group~$G$ we construct the \textsl{category of $G$-equivariant matrix factorisations} $\hmf(R,W)^G$ as follows \cite{add0401}. Denote by ${}_g(-)$ the functor that sends an $R$-module~$X$ to the $R$-module which as a set equals~$X$, but whose $R$-action is twisted by $g\in \operatorname{Aut}(R)$ in the sense that $(r,m) \mapsto g^{-1}(r).m$ for all $r\in R$, $m\in X$. Objects in $\hmf(R,W)^G$ are objects $(X, d_X)$ in $\hmf(R,W)$ together with a set of isomorphisms $\{ \varphi_g : {}_g X \rightarrow X \}_{g\in G}$ such that $\varphi_e = 1_X$ and the diagram
\be\label{eq:phi_g-condition}
\begin{tikzpicture}[
			     baseline=(current bounding box.base), 
			     >=stealth,
			     descr/.style={fill=white,inner sep=2.5pt}, 
			     normal line/.style={->}
			     ] 
\matrix (m) [matrix of math nodes, row sep=3em, column sep=2.5em, text height=1.5ex, text depth=0.25ex] {%
{}_{gh} X && {}_g X && X \\
};
\path[font=\scriptsize] (m-1-1) edge[->] node[auto] {$ {}_g (\varphi_h) $} (m-1-3)
				  (m-1-3) edge[->] node[auto] {$ \varphi_g $} (m-1-5); 
\path[font=\scriptsize] 
				 (m-1-1) edge[->, bend right=35] node[auto] {$ \varphi_{gh} $} (m-1-5); 
\end{tikzpicture}
\ee
commutes. Morphisms in $\hmf(R,W)^G$  are morphisms $\Psi: X \rightarrow Y$ in $\hmf(R,W)$  that make the following diagram commute: 
\be
\begin{tikzpicture}[
			     baseline=(current bounding box.base), 
			     >=stealth,
			     descr/.style={fill=white,inner sep=2.5pt}, 
			     normal line/.style={->}
			     ] 
\matrix (m) [matrix of math nodes, row sep=3em, column sep=2.5em, text height=1.5ex, text depth=0.25ex] {%
X && Y\\
{}_g X && {}_g Y \\
};
\path[font=\scriptsize] (m-1-1) edge[->] node[auto] {$ \Psi $} (m-1-3);
\path[font=\scriptsize] (m-2-1) edge[->] node[auto] {$ {}_g \Psi $} (m-2-3) 
				  (m-2-1) edge[->]  node[auto] {$ \varphi_{g}^{(X)} $} (m-1-1); 
\path[font=\scriptsize] (m-2-3) edge[->] node[auto, swap] {$ \varphi_{g}^{(Y)} $} (m-1-3); 
\end{tikzpicture}
\ee

We now claim that $\hmf(R,W)^G$ is equivalent to the category of modules over a particular algebra object in $\hmf(\Re,\widetilde W)$. Its underlying matrix factorisation is given by
\be
A_G = \bigoplus_{g\in G} {}_g I
\ee
where ${}_g I$ is the identity defect twisted by the group element~$g$ as explained above. 
From this one finds 
\be
	\diml ( {}_g I ) = \det(g^{-1}) 
\, , \quad 
	\dimr ( {}_g I ) = \det(g) 
\, ,
\ee
where $\det(g)$ is the determinant of the matrix representing the $g$-action on the variables $x_1, \ldots, x_n$
\cite[Sect.\,3.1]{BCP1}.

To make $A_G$ into an algebra, we specify the multiplication
\be
\mu = \sum_{g,h\in G} \mu_{g,h}: A_G \otimes A_G \lra A_G \, , \quad
\mu_{g,h} = {}_g(\lambda_{{}_h I}): {}_g I \otimes {}_h I \lra {}_{gh} I
\ee
in terms of the unit isomorphism $\lambda_{{}_h I}: I \otimes {}_h I \lra {}_{h} I$, together with the obvious unit $I \hookrightarrow A_G$. Furthermore, $A_G$ is a coalgebra with comultiplication 
\be
\Delta = \frac{1}{|G|} \sum_{g,h\in G} \Delta_{g,h}: A_G \lra A_G \otimes A_G \, , \quad 
\Delta_{g,h} = {}_g (\lambda^{-1}_{{}_h I}) : {}_{gh} I \lra {}_g I \otimes {}_h I \, , 
\ee
and counit given by the projection $A_G \twoheadrightarrow I$ multiplied by $|G|$. 

\begin{proposition}\label{prop:AGisgoodalgebra}
\begin{enumerate}
\item
$A_G$ is a separable Frobenius algebra, hence $(W,A_G) \in \LGeq$. 
\item 
If $\dimr( {}_g I ) = 1$ for all $g \in G$, then $A_G$ is also symmetric, and $(W,A_G) \in \LGorb$. 
\end{enumerate}
\end{proposition}

\begin{proof}
We first check that $A_G$ is an algebra. It is clear that it is unital, and the associativity of the product~$\mu$ amounts to the commutativity of the diagram
\be
\begin{tikzpicture}[
			     baseline=(current bounding box.base), 
			     >=stealth,
			     descr/.style={fill=white,inner sep=2.5pt}, 
			     normal line/.style={->}
			     ] 
\matrix (m) [matrix of math nodes, row sep=3em, column sep=2.5em, text height=1.5ex, text depth=0.25ex] {%
{}_g I \otimes {}_h I \otimes {}_k I && {}_g ( I \otimes {}_{hk} I ) \\
&& \\
{}_{gh} I \otimes {}_k I && {}_{ghk} I \, .  \\
};
\path[font=\scriptsize] (m-1-1) edge[->] node[auto] {$ {}_g ( 1_I \otimes {}_h (\lambda_{{}_k I}) ) $} (m-1-3)
				  (m-1-1) edge[->] node[above, sloped] {$ {}_g ( \lambda_{{}_h I}) \otimes 1_{{}_k I} $} (m-3-1)
				  (m-1-1) edge[->] node[below, swap, sloped] {$ = {}_g ( \lambda_{{}_h I \otimes {}_k I}) $} (m-3-1);
\path[font=\scriptsize] (m-1-3) edge[->] node[above, sloped] {$ {}_g ( \lambda_{{}_{hk} I} ) $} (m-3-3);
\path[font=\scriptsize] (m-3-1) edge[->] node[below, sloped] {$ {}_g ( {}_h ( \lambda_{{}_k I} ) ) $} (m-3-3);
\end{tikzpicture}
\ee
But this is ${}_g(-)$ applied to
\be
\begin{tikzpicture}[
			     baseline=(current bounding box.base), 
			     >=stealth,
			     descr/.style={fill=white,inner sep=2.5pt}, 
			     normal line/.style={->}
			     ] 
\matrix (m) [matrix of math nodes, row sep=3em, column sep=2.5em, text height=1.5ex, text depth=0.25ex] {%
I \otimes {}_h I \otimes {}_k I &&  I \otimes {}_{hk} I \\
&& \\
{}_{h} I \otimes {}_k I && {}_{hk} I \, \\
};
\path[font=\scriptsize] (m-1-1) edge[->] node[auto] {$ 1_I \otimes {}_h (\lambda_{{}_k I}) $} (m-1-3)
				  (m-1-1) edge[->] node[above, sloped] {$ \lambda_{{}_h I} \otimes 1_{{}_k I} $} (m-3-1)
				  (m-1-1) edge[->] node[below, swap, sloped] {$ = \lambda_{{}_h I \otimes {}_k I} $} (m-3-1);
\path[font=\scriptsize] (m-1-3) edge[->] node[above, sloped] {$ \lambda_{{}_{hk} I} $} (m-3-3);
\path[font=\scriptsize] (m-3-1) edge[->] node[below, sloped] {$ {}_h ( \lambda_{{}_k I} ) $} (m-3-3);
\end{tikzpicture}
\ee
which commutes by naturality of~$\lambda$. Similarly, it follows that $A_G$ is a coalgebra by reversing all arrows above. 

The fact that $A_G$ is separable is manifest in the definition of its (co)algebra structure, 
\begin{align}
\begin{tikzpicture}[very thick,scale=0.85,color=green!50!black, baseline=0cm]
\draw[-dot-] (0,0) .. controls +(0,-1) and +(0,-1) .. (1,0);
\draw[-dot-] (0,0) .. controls +(0,1) and +(0,1) .. (1,0);
\draw (0.5,-0.8) -- (0.5,-1.2); 
\draw (0.5,0.8) -- (0.5,1.2); 
\end{tikzpicture}
& = \frac{1}{|G|} \sum_{g,h\in G} {}_g ( \lambda_{{}_{g^{-1}h} I} ) \circ {}_g (\lambda_{{}_{g^{-1}h} I}^{-1} )
= \frac{1}{|G|} \sum_{g,h\in G} 1_{{}_h I} = 1_{A_G} \, . 
\end{align}

For the Frobenius property we observe that
\be
\begin{tikzpicture}[very thick,scale=0.85,color=green!50!black, baseline=0cm]
\draw[line width=0pt] 
(-1,1.5) node[line width=0pt] (g) {{\small$ {}_g I $}}
(0.5,1.5) node[line width=0pt] (hk) {{\small$ {}_{hk} I $}}
(1,-1.5) node[line width=0pt] (k) {{\small$ {}_k I $}}
(-0.5,-1.5) node[line width=0pt] (gh) {{\small$ {}_{gh} I $}}; 
\draw[-dot-] (0,0) .. controls +(0,-1) and +(0,-1) .. (-1,0);
\draw[-dot-] (1,0) .. controls +(0,1) and +(0,1) .. (0,0);
\draw (-1,0) -- (g); 
\draw (1,0) -- (k); 
\draw (0.5,0.8) -- (hk); 
\draw (-0.5,-0.8) -- (gh); 
\end{tikzpicture}
=
\begin{tikzpicture}[very thick,scale=0.85,color=green!50!black, baseline=0cm]
\draw[line width=0pt] 
(0,1.5) node[line width=0pt] (g) {{\small$ {}_g I $}}
(1,1.5) node[line width=0pt] (hk) {{\small$ {}_{hk} I $}}
(1,-1.5) node[line width=0pt] (k) {{\small$ {}_k I $}}
(0,-1.5) node[line width=0pt] (gh) {{\small$ {}_{gh} I $}}; 
\draw[-dot-] (g) .. controls +(0,-1) and +(0,-1) .. (hk);
\draw[-dot-] (gh) .. controls +(0,1) and +(0,1) .. (k);
\draw (0.5,-0.7) -- (0.5,0.7); 
\end{tikzpicture}
\ee
is equivalent to the commutativity of
\be
\begin{tikzpicture}[
			     baseline=(current bounding box.base), 
			     >=stealth,
			     descr/.style={fill=white,inner sep=2.5pt}, 
			     normal line/.style={->}
			     ] 
\matrix (m) [matrix of math nodes, row sep=3em, column sep=2.5em, text height=1.5ex, text depth=0.25ex] {%
{}_{gh} I \otimes {}_k I &&&& {}_g I \otimes {}_h I \otimes {}_k I \\
& {}_g ( {}_h I \otimes {}_k I ) && {}_g ( I \otimes {}_h I \otimes {}_k I ) & \\
&&&& \\
& {}_g ( {}_{hk} I ) && {}_g ( I \otimes {}_{hk} I ) & \\
{}_{ghk} I &&&& {}_g I \otimes {}_{hk} I \\
};
\path[font=\scriptsize] (m-1-1) edge[->] node[auto, sloped] {$ {}_{g} (\lambda_{{}_h I}^{-1}) \otimes {}_k I $} (m-1-5)
				  (m-1-1) edge[double equal sign distance] node[auto] {} (m-2-2)
				  (m-1-1) edge[->] node[below, sloped] {$ {}_{gh} (\lambda_{{}_k I}) $} (m-5-1);
\path[font=\scriptsize] (m-1-5) edge[->] node[above, sloped] {$ 1_{{}_g I} \otimes {}_h (\lambda_{{}_k I}) $} (m-5-5)
				  (m-1-5) edge[double equal sign distance] node[auto] {} (m-2-4);
\path[font=\scriptsize] (m-2-2) edge[->] node[below, sloped] {$ {}_g ( {}_h ( \lambda_{{}_k I}) ) $} (m-4-2);
\path[font=\scriptsize] (m-5-1) edge[->] node[below, sloped] {$ {}_g ( \lambda_{{}_{hk} I}^{-1}) $} (m-5-5)
				  (m-5-1) edge[double equal sign distance] node[auto] {} (m-4-2);
\path[font=\scriptsize] (m-2-4) edge[->] node[above, sloped] {$ {}_g (\lambda_{{}_h I \otimes {}_k I}) $} (m-2-2)
				  (m-2-4) edge[->] node[above, sloped] {$ {}_{g} ( 1_I \otimes {}_h (\lambda_{{}_k I}) ) $} (m-4-4);
\path[font=\scriptsize] (m-4-4) edge[->] node[below, sloped] {$ {}_g ( \lambda_{{}_{hk} I}) $} (m-4-2)
				  (m-4-4) edge[double equal sign distance] node[auto] {} (m-5-5);
\end{tikzpicture}
\ee
which again holds since~$\lambda$ is natural. Similarly one proves the other Frobenius property
\be
\begin{tikzpicture}[very thick,scale=0.85,color=green!50!black, baseline=0cm]
\draw[line width=0pt] 
(-0.5,1.5) node[line width=0pt] (g) {{\small$ {}_{gh} I $}}
(0.5,-1.5) node[line width=0pt] (hk) {{\small$ {}_{hk} I $}}
(1,1.5) node[line width=0pt] (k) {{\small$ {}_{k} I $}}
(-1,-1.5) node[line width=0pt] (gh) {{\small$ {}_{g} I $}}; 
\draw[-dot-] (0,0) .. controls +(0,1) and +(0,1) .. (-1,0);
\draw[-dot-] (1,0) .. controls +(0,-1) and +(0,-1) .. (0,0);
\draw (-1,0) -- (gh); 
\draw (1,0) -- (k); 
\draw (0.5,-0.8) -- (hk); 
\draw (-0.5,0.8) -- (g); 
\end{tikzpicture}
=
\begin{tikzpicture}[very thick,scale=0.85,color=green!50!black, baseline=0cm]
\draw[line width=0pt] 
(0,1.5) node[line width=0pt] (g) {{\small$ {}_{gh} I $}}
(1,1.5) node[line width=0pt] (hk) {{\small$ {}_{k} I $}}
(1,-1.5) node[line width=0pt] (k) {{\small$ {}_{hk} I $}}
(0,-1.5) node[line width=0pt] (gh) {{\small$ {}_{g} I $}}; 
\draw[-dot-] (g) .. controls +(0,-1) and +(0,-1) .. (hk);
\draw[-dot-] (gh) .. controls +(0,1) and +(0,1) .. (k);
\draw (0.5,-0.7) -- (0.5,0.7); 
\end{tikzpicture}
. 
\ee

Finally we have to show that $A_G$ is symmetric if $\dim( {}_g I ) = 1$ for all $g \in G$. As follows from the definition, $A_G$ is symmetric if the two maps
\be
L= 
\begin{tikzpicture}[very thick,scale=0.85,color=green!50!black, baseline=0cm]
\draw[-dot-] (0,0) .. controls +(0,1) and +(0,1) .. (-1,0);
\draw[line width=0pt] 
(1,1.5) node[line width=0pt] (target) {{\small$ ({}_g I)^\dagger $}}
(-1,-1.5) node[line width=0pt] (source) {{\small$ {}_{g^{-1}} I $}};
\draw[directedgreen, color=green!50!black] (1,0) .. controls +(0,-1) and +(0,-1) .. (0,0);
\draw (-1,0) -- (source); 
\draw (1,0) -- (target); 
\draw (-0.5,1.2) node[Odot] (end) {}; 
\draw (-0.5,0.8) -- (end); 
\end{tikzpicture}
, \quad
R= 
\begin{tikzpicture}[very thick,scale=0.85,color=green!50!black, baseline=0cm]
\draw[redirectedgreen, color=green!50!black] (0,0) .. controls +(0,-1) and +(0,-1) .. (-1,0);
\draw[-dot-] (1,0) .. controls +(0,1) and +(0,1) .. (0,0);
\draw[line width=0pt] 
(-1,1.5) node[line width=0pt] (target) {{\small$ ({}_g I)^\dagger $}}
(1,-1.5) node[line width=0pt] (source) {{\small$ {}_{g^{-1}} I $}};
\draw (-1,0) -- (target); 
\draw (1,0) -- (source); 
\draw (0.5,1.2) node[Odot] (end) {}; 
\draw (0.5,0.8) -- (end); 
\end{tikzpicture}
\ee
are identical for all $g\in G$. Since $\Delta_{g,g^{-1}}$ and $\tev_{{}_g I}$ are isomorphisms, the identity $L=R$ is equivalent to $\tev_{{}_g I} \circ (1_{{}_g I} \otimes L) \circ \Delta_{g,g^{-1}} = \tev_{{}_g I} \circ (1_{{}_g I} \otimes R) \circ \Delta_{g,g^{-1}}$. Using the Frobenius property we find that the left-hand side of the latter identity is $\dimr({}_g I)$, while a Zorro move together with separability reveal that the right-hand side is unity. 
\end{proof}

We are now ready to recover equivariant matrix factorisations as the category $\modu(A_G) = \LGeq( (0,I), (W,A_G) )$ in our framework of equivariant completion. This justifies our choice of nomenclature in Section~\ref{sec:equibicat}. 

\begin{theorem}
$\hmf(R,W)^G \cong \modu(A_G)$. 
\end{theorem}

\begin{proof}
Let $X \in \hmf(R,W)^G$ with isomorphisms $\{ \varphi_g : {}_g X \rightarrow X \}_{g\in G}$. We define 
\be
\rho = \sum_{g\in G} \!\!
\begin{tikzpicture}[
			     baseline=(current bounding box.base)
			     ] 
\matrix (m) [matrix of math nodes, row sep=3em, column sep=2.5em, text height=1.5ex, text depth=0.25ex] {%
\Big( A_G \otimes X & {}_g I \otimes X & {}_g X & X \Big) \\
};
\path[font=\scriptsize] (m-1-1) edge[->>] node[auto] {} (m-1-2);
\path[font=\scriptsize] (m-1-2) edge[->] node[auto] {$ {}_g (\lambda_X) $} (m-1-3);
\path[font=\scriptsize] (m-1-3) edge[->] node[auto] {$ \varphi_g $} (m-1-4);
\end{tikzpicture}
,
\ee
that is, $\rho = \bigoplus_{g\in G} \rho_g$ with $\rho_g = \varphi_g \circ {}_g(\lambda_X)$. The map~$\rho$ satisfies $\rho \circ (\mu \otimes 1_X) = \rho \circ ( 1_{A_G} \otimes \rho)$ and thus is a left action of $A_G$ on~$X$. This follows from the commutativity of the diagram
\be
\begin{tikzpicture}[
			     baseline=(current bounding box.base), 
			     >=stealth,
			     descr/.style={fill=white,inner sep=2.5pt}, 
			     normal line/.style={->}
			     ] 
\matrix (m) [matrix of math nodes, row sep=3em, column sep=2.5em, text height=1.5ex, text depth=0.25ex] {%
{}_g I \otimes {}_h I \otimes X &&&&&& {}_g I \otimes X \\
&&& {}_g I \otimes {}_h X &&& \\
&&&&& {}_g X & \\
&&& {}_{gh} X &&& \\
{}_{gh} I \otimes X &&&&&& X \\
};
\path[font=\scriptsize] (m-1-1) edge[->] node[auto] {$ 1_{{}_g I} \otimes \rho_h $} (m-1-7)
				  (m-1-1) edge[->] node[above, sloped] {$ 1_{{}_g I} \otimes {}_h (\lambda_X) $} (m-2-4)
				  (m-1-1) edge[->] node[below, swap, sloped] {$ = {}_g (1_I \otimes {}_h(\lambda_X)) $} (m-2-4)
				  (m-1-1) edge[->] node[below, sloped] {$ \mu_{g,h} \otimes 1_X $} (m-5-1)
				  (m-1-1) edge[->, out = -45, in=45] node[above, sloped] {$ {}_g (\lambda_{{}_h I} \otimes1_X) $} (m-5-1);
\path[font=\scriptsize] (m-2-4) edge[->] node[above, sloped] {$ 1_{{}_g I} \otimes \varphi_h $} (m-1-7)
				  (m-2-4) edge[->] node[below, sloped] {$ = {}_g (1_I \otimes \varphi_h) $} (m-1-7)
				  (m-2-4) edge[->] node[below, sloped] {$ {}_g (\lambda_{{}_h X})$} (m-4-4);
\path[font=\scriptsize] (m-1-7) edge[->, out = -120, in=90] node[above, sloped] {$ {}_g (\lambda_X) $} (m-3-6)
				  (m-1-7) edge[->] node[above, sloped] {$ \rho_g $} (m-5-7);
\path[font=\scriptsize] (m-3-6) edge[->, out = -90, in=120] node[above, sloped] {$ \varphi_g $} (m-5-7);
\path[font=\scriptsize] (m-4-4) edge[->] node[above, sloped] {$ {}_g (\varphi_h) $} (m-3-6)
				  (m-4-4) edge[->] node[above, sloped] {$ \varphi_{gh} $} (m-5-7);
\path[font=\scriptsize] (m-5-1) edge[->] node[above, sloped] {$ {}_g ( {}_h (\lambda_X)) $} (m-4-4)
				  (m-5-1) edge[->] node[above, sloped] {$ \rho_{gh} $} (m-5-7);
\end{tikzpicture}
\ee
where each subdiagram commutes either by construction or by naturality of~$\lambda$ and the coherence theorem. Furthermore, $\rho_e = \varphi_e \circ {}_e (\lambda_X) = \lambda_X$, so that the conditions~\eqref{eq:leftcomp} are satisfied and the equivariant matrix factorisation~$X$ is endowed with an $A_G$-module structure. Conversely an object in $\modu(A_G)$ is made into one in $\hmf(R,W)^G$ by inverting the above argument. 

So far we have constructed two mutually inverse functors $\hmf(R,W)^G \leftrightarrow \modu(A_G)$ on objects, and it remains to check that they are well-defined on morphisms. Thus we have to show that one of the two diagrams
\be
\begin{tikzpicture}[
			     baseline=(current bounding box.base)
			     ] 
\matrix (m) [matrix of math nodes, row sep=3em, column sep=2.5em, text height=1.5ex, text depth=0.25ex] {%
X && Y\\
{}_g X && {}_g Y \\
};
\path[font=\scriptsize] (m-1-1) edge[->] node[auto] {$ \Psi $} (m-1-3);
\path[font=\scriptsize] (m-2-1) edge[->] node[auto] {$ {}_g \Psi $} (m-2-3) 
				  (m-2-1) edge[->]  node[auto] {$ \varphi_{g}^{(X)} $} (m-1-1); 
\path[font=\scriptsize] (m-2-3) edge[->] node[auto, swap] {$ \varphi_{g}^{(Y)} $} (m-1-3); 
\end{tikzpicture}
\, , \quad
\begin{tikzpicture}[
			     baseline=(current bounding box.base)
			     ] 
\matrix (m) [matrix of math nodes, row sep=3em, column sep=2.5em, text height=1.5ex, text depth=0.25ex] {%
A_G \otimes X && A_G \otimes Y\\
X && Y \\
};
\path[font=\scriptsize] (m-1-1) edge[->] node[auto] {$ 1\otimes \Psi $} (m-1-3);
\path[font=\scriptsize] (m-2-1) edge[->] node[auto] {$ \Psi $} (m-2-3) 
				  (m-1-1) edge[->]  node[auto, swap] {$ \rho^{(X)} $} (m-2-1); 
\path[font=\scriptsize] (m-1-3) edge[->] node[auto] {$ \rho^{(Y)} $} (m-2-3); 
\end{tikzpicture}
\ee
commutes if and only if the other does. This follows from the naturality of~$\lambda$, by which the middle square of the diagram 
\be
\begin{tikzpicture}[
			     baseline=(current bounding box.base)
			     ] 
\matrix (m) [matrix of math nodes, row sep=3em, column sep=2.5em, text height=1.5ex, text depth=0.25ex] {%
{}_g I \otimes X && {}_g I \otimes Y\\
{}_g( I \otimes X) && {}_g (I \otimes Y) \\
{}_g X && {}_g Y\\
 X && Y\\
};
\path[font=\scriptsize] (m-1-1) edge[->] node[auto] {$ 1_{{}_g I} \otimes \Psi $} (m-1-3)
				  (m-1-1) edge[double equal sign distance] node[auto] {} (m-2-1);
\path[font=\scriptsize] (m-1-3) edge[double equal sign distance] node[auto] {} (m-2-3);
\path[font=\scriptsize] (m-2-1) edge[->] node[auto] {$ {}_g ( 1_I \otimes \Psi) $} (m-2-3)
				  (m-2-1) edge[->] node[auto, swap] {$ {}_g ( \lambda_X ) $} (m-3-1); 
\path[font=\scriptsize] (m-2-3) edge[->] node[auto] {$ {}_g(\lambda_Y) $} (m-3-3);
\path[font=\scriptsize] (m-3-1) edge[->] node[auto] {$ {}_g \Psi $} (m-3-3)
				  (m-3-1) edge[->] node[auto, swap] {$ \varphi_g^{(X)}  $} (m-4-1); 
\path[font=\scriptsize] (m-3-3) edge[->] node[auto] {$ \varphi_g^{(Y)}  $} (m-4-3); 
\path[font=\scriptsize] (m-4-1) edge[->] node[auto] {$ \Psi $} (m-4-3); 
\end{tikzpicture}
\ee
commutes. 
\end{proof}

\begin{remark}
Similarly, the category $\hmf(\Re,\widetilde W)^G$ of $G$-equivariant defects studied in \cite{br0712.0188} is obtained as $\operatorname{bimod}(A_G)$, and the bulk fields of \cite{v1989, IntriligatorVafa1990} are described as $\End_{A_G A_G}(A_G)$. Hence the complete conventional equivariant theory of Landau-Ginzburg models embeds into $\LGeq$. 
\end{remark}

Whenever $(W,A_G)$ is an object in the orbifold completion $\LGorb$ the general theory of the previous sections applies. This is the case when $A_G$ is symmetric which by Proposition~\ref{prop:AGisgoodalgebra} is equivalent to the condition that $\dim {}_g I = 1$ for all ${}_g I$.
As an immediate consequence this proves that for symmetric $A_G$, $\hmf(R,W)^G$ is a Calabi-Yau category.  

\begin{theorem}\label{thm:WAGocTFT}
Let~$W$ be a homogeneous potential in a graded ring~$R$, and let~$G$ be a symmetry group of~$W$ such that $\dim( {}_g I ) = 1$ for all $g\in G$. Then $(W,A_G)$ gives an open/closed TFT. In particular, the Cardy condition holds for equivariant matrix factorisations, and the Kapustin-Li pairing
\be\label{eq:G-equiv-Cardy}
\langle \Psi_1, \Psi_2 \rangle_X  = 
\Res \left[ \frac{ \str( \Psi_1 \Psi_2 \, \partial_{x_1} d_X \ldots \partial_{x_n} d_X) \, \underline{\operatorname{d}\! x}}{\partial_{x_1} W, \ldots, \partial_{x_n} W} \right]
\ee
is nondegenerate when restricted to $G$-equivariant morphisms $\Psi_1: Y\rightarrow X[n]$, $\Psi_2: X\rightarrow Y$ in $\hmf(R,W)^G$. 
\end{theorem}

\begin{proof}
This follows directly from Theorem~\ref{thm:LGorbiTFT} and Proposition~\ref{prop:AGisgoodalgebra}. 
\end{proof}

\begin{remark}
\begin{enumerate}
\item 
The $G$-equivariant Cardy condition \eqref{eq:G-equiv-Cardy} was already proved in \cite{pv1002.2116} by different methods and without any assumption on $\dimr ( {}_g I )$. A proof from the perspective of the present paper and equally without the assumption that $\dimr ( {}_g I )=1$ was given later in \cite[Prop.\,3.16]{BCP1}. 
\item 
It is not true that the Kapustin-Li pairing always induces nondegenerate pairings on $\hmf(R,W)^G$. A counterexample is $W=x^d$ ($d \geqslant 3$) with the action of the symmetry group $G=\Z_d$ generated by $g_0 : x \mapsto \eta x$, where $\eta = \E^{2\pi\I/d}$. Namely, consider the equivariant matrix factorisation $(X=\C[x]^{\oplus 2}, d_X = (\begin{smallmatrix} 0 & x^n \\ x^{d-n} & 0 \end{smallmatrix}))$ with 
$\varphi_{g_0} = (\begin{smallmatrix}g_0(-) & 0 \\ 0 & \eta^{-n}g_0(-) \end{smallmatrix}): {}_{g_0} X \rightarrow X$ (this fixes the remaining $\varphi_g$ uniquely via \eqref{eq:phi_g-condition}).
The only equivariant endomorphisms of $X$ are proportional to the identity, but $\langle 1_X, 1_X \rangle_X = 0$. This is consistent with the fact that from Theorem~\ref{thm:LGbicatsummary}(ii) (or \cite[(3.25--26)]{cr1006.5609}) one finds $\diml({}_{g_0} I) = \eta^{-1}$ and $\dimr({}_{g_0} I) = \eta$ which are both $\neq 1$. 
\item
On the other hand, we note that the condition $\dim( {}_g I ) = 1$ is naturally satisfied for a large class of models. In particular this is the case in the CY/LG correspondence of \cite{o0503632, hhp0803.2045}, which states that for a quasi-homogeneous potential $W \in R = k[x_1,\ldots,x_n]$ of degree~$d$ the bounded derived category of coherent sheaves on the hypersurface $\mathcal X = \{ W = 0\}$ in weighted projective space is equivalent to $(\hmf(R,W)^{\mathrm{gr}})^G$ if~$\mathcal X$ is a Calabi-Yau variety. Here~$G = \Z_d$ acts diagonally as $x_i \mapsto \E^{2\pi\I |x_i|/d} x_i$, and the Calabi-Yau condition on~$\mathcal X$ is $\sum_{i=1}^n |x_i| = d$. This condition implies $\dim( {}_g I ) = 1$ for all $g \in G$; in the Fermat case this follows directly from part~(i) and the fact that quantum dimensions are multiplicative, while the general case can be reduced to $\dim(I)=1$ (after a variable rescaling in~\eqref{eq:rightdecorateddefectaction} and application of the Calabi-Yau condition).
\end{enumerate}
\end{remark}

\subsection{Kn\"orrer periodicity}\label{subsec:Knoerrer}

Next we turn to the classical result \cite{knoe1987} that for an algebraically closed field~$k$ and a potential $W\in k[x] = k[x_1,\ldots,x_n]$, there is an equivalence $\hmf(k[x], W)^\omega \cong \hmf(k[x,u,v], W + u^2 - v^2)^\omega$. To understand this from the perspective of our construction in Section~\ref{sec:equibicat}, we will show that $(k[x],W) \cong (k[x,u,v], W + u^2 - v^2)$ in $\LG$. 

Consider the Koszul matrix factorisation $K = k[u,v]^{\oplus 2}$ with differential $d_K = (\begin{smallmatrix} 0 & u-v \\ u+v & 0 \end{smallmatrix})$ and define
\be
X = I_W \otimes_k K \in \LG(W, W + u^2 - v^2) \, . 
\ee
Using the explicit expressions in Theorem~\ref{thm:LGbicatsummary} one finds $\diml(K) = -\frac12$ and $\dimr(K)=-2$. From these expressions it is also straightforward to check that quantum dimensions behave multiplicatively under external products $\otimes_k$ (up to signs if odd numbers of variables are included), so that $\dim_{\mathrm{l}/\mathrm{r}}(X) = \pm\dim_{\mathrm{l}/\mathrm{r}}(K)$.
In particular both quantum dimensions are invertible. Thus by Proposition~\ref{prop:hmfSVeqmodXX} we know that $\hmf(k[x,u,v], W + u^2 - v^2)^\omega \cong \modu(A)$ where $A=X^\dagger \otimes X$, and it remains to show that $A\cong I_W$ since $\hmf(k[x],W)^\omega = \modu(I_W)$. For this we observe
\be
K^\vee \otimes_{k[u,v]} K 
\cong ( K^\vee \otimes_{k[u]} K) \otimes_{k[v]^{\text{e}}} k[v] 
\cong I^\vee_{-v^2} \otimes_{k[v]^{\text{e}}} k[v] 
\cong I_0 \, , 
\ee
where the second equivalence follows since $K$ is the identity matrix factorisation for~$u^2$. 
Thus $X^\dagger \otimes X \cong I_W$, and we recover Kn\"orrer periodicity from the defect~$X$ in our setting. 

\subsection{Orbifold equivalences between minimal models}\label{subsec:minimalmodels}

Simple singularities have an ADE classification \cite[Sect.\,15.1]{AGV-book}.
For an even number of variables, the associated polynomials  are
\begin{align}
W^{(\mathrm{A}_{d-1})} & = u^d - v^2 \, , \quad 
W^{(\mathrm{D}_{d+1})} = x^d - xy^2 \, , \\
W^{(\mathrm{E}_6)} & = x^3 + y^4 \, , \quad
W^{(\mathrm{E}_7)} = x^3 + xy^3 \, , \quad
W^{(\mathrm{E}_8)} = x^3 + y^5 \, . 
\end{align}
Landau-Ginzburg models with these potentials are believed to correspond to $\mathcal N=2$ minimal conformal field theories \cite{m1989, vw1989, howewest, kl0306001, bg0506208, kr0610175, cr0909.4381}. These rational conformal field theories are known to be (generalised) orbifolds of each other \cite{g0812.1318, ffrs0909.5013}. Inspired by this fact in this section we will obtain similar results for matrix factorisations. 

Let us consider the matrix factorisation $X = k[u,v,x,y]^{\oplus 4}$ with differential
\be
d_{X} = 
\begin{pmatrix}
0 & x - u^2 \\
\frac{x^d - u^{2d}}{x-u^2} - y^2 & 0
\end{pmatrix}
\otimes_k 
\begin{pmatrix}
0 & v - uy \\
v + uy & 0
\end{pmatrix}
\ee
which we view as a 1-morphism in $\LG( W^{(\mathrm{A}_{2d-1})}, W^{(\mathrm{D}_{d+1})} )$, i.\,e.~a defect between minimal models of type~A and~D. Put differently, $X$ is the stabilisation of the module $k[u, v, x, y]/(x - u^2, v - uy)$. We claim that this defect implements an orbifold equivalence
between the two theories. Invoking Theorem~\ref{thm:algebraXdaggerX-summary} all we have to do to prove this is to check that~$X$ has invertible (right) quantum dimension. 

By Theorem~\ref{thm:LGbicatsummary} the left and right quantum dimensions are given by
\begin{align}
\diml(X) & = - \Res \left[ \frac{\str \big( \partial_u d_X \partial_v d_X \partial_x d_X \partial_y d_X \big) \operatorname{d}\! x \operatorname{d}\! y}{\partial_x W^{(\mathrm{D}_{d+1})}, \, \partial_y W^{(\mathrm{D}_{d+1})} } \right] , \label{eq:ADleftdim} \\ 
\dimr(X) & = - \Res \left[ \frac{\str \big( \partial_u d_X \partial_v d_X \partial_x d_X \partial_y d_X \big) \operatorname{d}\! u \operatorname{d}\! v}{\partial_u W^{(\mathrm{A}_{2d-1})}, \, \partial_v W^{(\mathrm{A}_{2d-1})} } \right] .
\end{align}
A direct computation yields 
\be
\str \big( \partial_u d_X \partial_v d_X \partial_x d_X \partial_y d_X \big) 
= 4y^2 + \sum_{i=0}^{d-2} 4 d u^{2i + 2} x^{d-2-i} \, ,
\ee
and with $\partial_u W^{(\mathrm{A}_{2d-1})} = 2du^{2d-1}$, $\partial_v W^{(\mathrm{A}_{2d-1})} = -2v$ we find that $\dimr(X) = 1$, which is invertible. As an exercise we also compute the left quantum dimension, for which we use the transformation formula 
\be
\Res \left[ \frac{\phi \, \underline{\operatorname{d}\! x}}{f_1, \ldots, f_n} \right] 
= 
\Res \left[ \frac{\det(C) \phi \, \underline{\operatorname{d}\! x}}{g_1, \ldots, g_n} \right] 
\quad \text{if }
g_i = \sum_{j=1}^n C_{ij} f_j 
\ee
to convert the residue in~\eqref{eq:ADleftdim} to one with only monomials in the denominator. Indeed, if we set $C = (\begin{smallmatrix} 2x & -y \\ 2y & dx^{d-2}\end{smallmatrix})$ and $f_1 = \partial_x  W^{(\mathrm{D}_{d+1})}$, $f_2 = \partial_y  W^{(\mathrm{D}_{d+1})}$, then 
\be
g_1 = \sum_{j=1}^2 C_{1j} f_j = 2dx^d \, , \quad 
g_2 = \sum_{j=1}^2 C_{2j} f_j = -2y^3 \, , 
\ee
from which we find that $\diml(X) = 2$. Thus both quantum dimensions of~$X$ are invertible, but already from the invertibility of $\dimr(X)$ we conclude: 

\begin{theorem}\label{thm:ADorbifold}
With $A_d := X^\dagger \otimes X$ we have $(W^{(\mathrm{D}_{d+1})}, I_{W^{(\mathrm{D}_{d+1})}}) \cong (W^{(\mathrm{A}_{2d-1})}, A_d)$ in $\LGorb$. In particular, 
\begin{align}
\hmf(k[x,y], W^{(\mathrm{D}_{d+1})} )^\omega & \cong \modu(A_d) \, , \label{eq:ADorbi-categories} \\
\Jac(W^{(\mathrm{D}_{d+1})}) & \cong \End_{A_d A_d}(A_d) \, . \label{eq:ADorbi-bulkspace}
\end{align}
\end{theorem}

\begin{proof}
This follows from Theorem~\ref{thm:algebraXdaggerX-summary} and Proposition~\ref{prop:hmfSVeqmodXX}. 
\end{proof}

We have computed that
\be
A_d \cong I_{W^{(\mathrm{A}_{2d-1})}} \oplus J_d 
\quad\text{with}\quad 
J_d \otimes J_d \cong I_{W^{(\mathrm{A}_{2d-1})}}
\ee
where $J_d = P_{\{d\}}[1] \otimes (\begin{smallmatrix} 0&v-v'\\v+v'&0\end{smallmatrix})$ in the notation of \cite[(6.2)]{br0707.0922} for $d\in \{ 2,3,\ldots,10 \}$, and we believe both relations to hold in general. This would be in accordance with the situation in CFT, making the equivalence $(W^{(\mathrm{D}_{d+1})}, I_{W^{(\mathrm{D}_{d+1})}}) \cong (W^{(\mathrm{A}_{2d-1})}, A_d)$ into a $\Z_2$-orbifold. Also note that checking the equivalences~\eqref{eq:ADorbi-categories}, \eqref{eq:ADorbi-bulkspace} directly would be a rather painful enterprise. By our general construction in Section~\ref{sec:equibicat} all we had to do is produce a matrix factorisation~$X$ of $W^{(\mathrm{D}_{d+1})} - W^{(\mathrm{A}_{2d-1})}$ with invertible quantum dimension, a condition that is easily checked thanks to the explicit residue expressions.\footnote{Results similar to~\eqref{eq:ADorbi-categories}, with the appearance of skew group algebras instead of orbifolds, are \cite[\S\,2.1]{ReitenRiedtmann} and \cite[Thm.\,1]{d0902.1390}, the relation to matrix factorisation being due to \cite[Thm.\,3.1]{kst0511155}. We thank Bernhard Keller and Daniel Murfet for pointing this out.}

\medskip

It would be very useful to have a constructive method of producing 1-morphisms~$X$ in $\LG$ with invertible quantum dimension between any given pair of potentials $V,W$ whenever they exist. More ambitiously one could even aim for a classification of such matrix factorisations. For many potentials there will be obstructions to the existence of such~$X$ (as for example the condition on central charges in Proposition~\ref{prop:centralchargesandqdims}), but any matrix factorisation with invertible quantum dimension could potentially give rise to previously unknown equivalences between triangulated categories. 

From the above result and from the analogous situation in rational CFT, we conjecture that that there are also 1-morphisms with invertible quantum dimension between minimal models of type~A and~E which produce the following equivalences: 
\begin{align}
\hmf( k[x,y], W^{(\mathrm{E}_6)} )^\omega & \cong \modu( A_{6} ) \, , \label{eq:E6A11} \\ 
\hmf( k[x,y], W^{(\mathrm{E}_7)} )^\omega & \cong \modu( A_{9} ) \, , \label{eq:E7A17} \\
\hmf( k[x,y], W^{(\mathrm{E}_8)} )^\omega & \cong \modu( A_{15} ) \, , \label{eq:E8A29}
\end{align}
both in the $\Z_2$- and in the $\Z$-graded situation, where the $A_d$-modules are taken in $\hmf(k[u,v] , W^{(\mathrm{A}_{2d-1})})$. While naive attempts at constructing matrix factorisations with invertible quantum dimension of say 
$W^{(\mathrm{E}_6)} - W^{(\mathrm{A}_{11})}$ 
have failed so far, we are confident that a more systematic approach will successfully establish the above equivalences.\footnote{This conjecture is proven in \cite{CRCR}.} 
By the central charge condition (Proposition~\ref{prop:centralchargesandqdims}),  in the $\Z$-graded case the equivalences~\eqref{eq:E6A11}--\eqref{eq:E8A29} between A- and E-models (and the corresponding D-models) together with those between A- and D- models treated in Theorem~\ref{thm:ADorbifold} would exhaust all generalised orbifold equivalences~$X$ between minimal models. 
We also note that for such~$X$ Proposition~\ref{prop:hmfSVeqmodXX} says that the module categories $\modu(X^\dagger \otimes X)$ are always of finite type, because minimal models are. 

Looking further ahead we stress that there is no reason to believe that interesting equivalences are confined to simple minimal models. To the contrary, our construction applies to all Landau-Ginzburg models, including (but not limited to) those that are related to Calabi-Yau hypersurfaces as in \cite{o0503632, hhp0803.2045}.

\end{document}